\documentclass[10pt,a4paper,twoside]{article}

\usepackage{amsmath, amssymb, amsfonts, amsthm}
\usepackage{mathrsfs}

\usepackage{tikz}
\usetikzlibrary{arrows,calc,intersections}
\usetikzlibrary{calc,patterns,matrix}
\usetikzlibrary{decorations}
\usetikzlibrary{decorations.pathreplacing}
\usetikzlibrary{decorations.markings}
\usetikzlibrary{plotmarks}
\usetikzlibrary{shapes.misc}
\usetikzlibrary{shapes,snakes}

\usepackage{listings}

\usepackage[plainpages=false,pdfpagelabels,naturalnames=true]{hyperref}
\hypersetup{colorlinks=true,linkcolor=red,pdfpagemode=None,pdfstartview=FitH}

\usepackage{geometry}
\newtheorem{theo}{Theorem}[section]

\newtheorem{defi}[theo]{Definition}

\newtheorem{prop}[theo]{Proposition}

\DeclareMathOperator{\Gal}{Gal}

\DeclareMathOperator{\Hom}{Hom}

\DeclareMathOperator{\Cl}{Cl}

\DeclareMathOperator{\Mat}{Mat}
\DeclareMathOperator{\Diag}{Diag}

\def\FF{\mathbb{F}}

\def\F{\mathbf{F}}
\def\P{\mathbf{P}}

\def\E{\mathbf{E}}
\def\Z{\mathbf{Z}}
\def\A{\mathbf{A}}
\def\D{\mathbf{D}}

\def\Ccal{\mathcal{C}}

\def\EE{\mathcal{E}}
\def\Fcal{\mathcal{F}}
\def\Rcal{\mathcal{R}}
\def\O{\mathcal{O}}
\def\Pcal{\mathcal{P}}
\def\L{\mathcal{L}}

\def\RR{\mathscr{R}}

\def\plongement{\varphi}

\def\Weil{q+1+\lfloor 2\sqrt{q}\rfloor}
\newcommand{\Nqg}[2]{\operatorname{N}_{#1}\left(#2\right)}
\def\contraction{\chi}
\def\eclatement{\pi}

\DeclareMathOperator{\divcartier}{CDiv}
\DeclareMathOperator{\divweil}{WDiv}
\DeclareMathOperator{\cldivcartier}{CaCl}
\DeclareMathOperator{\cldivweil}{Cl}
\newcommand{\dmin}{d_{\text{min}}}

\newcounter{remark}
\newcommand{\remark}{\addtocounter{remark}{1}
                       \par \quad {\bf \arabic{remark}}.\,
                      }

\author{Régis Blache and Emmanuel Hallouin}

\title{\bfseries Construction of good codes from weak Del Pezzo surfaces}
\date{\today}

\begin{document}

\maketitle

\tableofcontents

\section{Introduction}

The geometric construction of error correcting codes goes back to Reed-Solomon and Goppa for curves and to Reed-Muller for affine or
projective spaces. In this work, we focus on evaluation codes from algebraic surfaces whose construction works as follows.
Given~$X$ an algebraic surface over~$\F_q$ and~$D$ a (Cartier) divisor on~$X$, we denote by~$X(\F_q)$ the set of rational points of~$X$
and by~$H^0(X,D)$ the space of global sections of the divisor~$D$. The associated {\em evaluation code} is the code whose codewords are
the evaluations of the functions of~$H^0(X,D)$ at the points of~$X(\F_q)$ (see definition~\ref{defCode}). The length
of such a code is thus~$\#X(\F_q)$, the dimension is~$\dim\left(H^0(X,D)\right)$ (at least if the
evaluation map is injective) and the third invariant, the minimum distance, is more difficult to control. It is related to
the maximum of rational points that can contain a curve in the linear system~$\left|D\right|$
(proposition~\ref{propParameters}). The lower this maximum, the better is the minimum distance. 

The well known Reed-Muller code of degree~$d$ for the projective plane~$\P^2$ over~$\F_q$ is a nice example.
Its codewords are the evaluations of the homogeneous polynomials of degree~$d$ at the rational points of~$\P^2(\F_q)$.
In the geometric setting above, this is the {\em evaluation code}
associated to the algebraic surface~$\P^2$, the divisor~$d\ell$ where~$\ell$ denotes any line, and the whole set
of rational points~$\P^2(\F_q)$. Its parameters are well known when~$q > d$: the length
is the number of rational points~$\#\P^2(\F_q) = q^2+q+1$, the dimension is the dimension of the space of global
section~$\dim\left(H^0(\P^2,d\ell)\right) = \binom{d+2}{2}$, and the minimum distance is~$q(q-d+1)$. 
This minimum distance can be written~$(q^2+q+1) - (1+dq)$
and~$(1+dq)$ is nothing else than the maximal number of rational points of a curve lying in the linear
system~$\left| d\ell\right|$ (i.e. the set of plane curves of degree~$d$).
In fact, this is the number of points of the union of $d$ lines of~$\P^2$ meeting at one point.
The existence of this kind of extremely reducible curve over~$\F_q$ impacts negatively the minimum distance, since they contain
too many rational points.
Among other things, this is why evaluation codes associated to more general algebraic surfaces have been considered.

One can distinguish several strategies in the literature to get rid of reducible curves with many components in the linear system.
The idea of Couvreur \cite{Alain} is to work with sublinear systems of~$\P^2$ by adding constraints that remove the very reducible
sections. In fact by choosing carefully the sublinear system, this kind of sections is no longer defined over the base field,
but only over an extension. The number of rational points of irreducible curves that are not absolutely irreducible may fail drastically.
In the preceding example of the
$d$ intersecting lines, if they are not defined over~$\F_q$ but only conjugate over~$\F_q$, then one can easily convince ourselves that
their union only contains one rational point, their meeting point. Some other examples can be found in Edoukou \cite{Edoukou}
or Couvreur \& Duursma \cite{CouvreurDuursma}.

Following Zarzar \cite{zarzar}, another fairly repeated strategy is to concentrate on surfaces whose (arithmetic) Cartier class group is
free of rank~$1$. Indeed, this is a natural way to overcome the difficulty of the existence of (very) reducible sections in the linear
system~$|D|$.
Little and Schenck \cite{LittleSchenck} have studied anticanonical codes on degree~$3$ and~$4$ del Pezzo surfaces having rank~$1$.
In our previous work \cite{AntiCanonical}, we could say that we fill a gap in the study of algebraic
geometric codes constructed from del Pezzo surfaces of rank~$1$. Let us remark that, even if the elements of the linear
system~$\left|D\right|$ are all irreducible, some of them may be absolutely reducible. As in the example of conjugate lines,
it is expected that these configurations do not contain too many points but this requires a proof.

In this work, we continue the investigation of codes constructed from del Pezzo surfaces.
We do not restrict ourselves to rank one surfaces but above all we consider more general surfaces,
that is {\em non ordinary weak del Pezzo surfaces}. As ordinary del Pezzo surfaces, non ordinary weak del Pezzo surfaces admit a blowing-up description; in the ordinary case, the
points that are blown up are in {\em general position} but in the non ordinary case, they are only in {\em almost general position}
(three points can be colinear and six points can be conconic). The main consequences of these weaker hypotheses on the configuration
of points are twofold. First, the surface contains $-2$-curves (and not only $-1$-curves). Secondly, the anticanonical divisor is not ample
anymore but only big and nef and the anticanonical model is singular with rational double points.

In a concomitant work \cite{Classification}, we have computed explicit models for all the {\em arithmetic types}
of weak del Pezzo surfaces of degree at least~$3$ over a finite field (these types lead to a classification that is coarser than the isomorphism one
but that permits to distinguish the main arithmetic properties of the weak del Pezzo surfaces).
Taking advantage of this knowledge, we select eight types of (non ordinary) weak del Pezzo that are well suited for coding applications.
More precisely, we consider~$X$ a (smooth) weak del Pezzo surface of degree~$d$ over~$\F_q$ and we denote by~$X_s$ its (singular)
anticanonical model; this is the image of the surface~$X$ by the morphism~$\plongement$ associated to~$-K_X$ the anticanonical divisor
of~$X$. Since~$-K_X$ is not ample, the surface~$X_s$ is singular with a finite number of rational double points.
We study the evaluation code associated to the (singular) surface~$X_s$, the Cartier divisor~$-K_{X_s} = \plongement_*(-K_X)$
and the whole set of rational points of~$X_s$ (definition~\ref{defiCode}).
Except for small values of~$q$, this code has length~$n=\#X_s(\F_q)$, dimension~$k=d+1$.
The last invariant, the minimum distance~$\dmin$, is much more subtle to control and requires preparatory calculations.

Before going into details, let us discuss the advantages and disadvantages of considering such weak Del Pezzo surfaces.
In the process of construction of a del Pezzo surface, there are blowing-up and blowing-down. The blowing-up may add rational
points and thus may increase the length. The blowing-down permits to contract some lines and thus decreases the types
of reducible configurations.
Since the anticanonical model is no longer smooth, besides the exceptional curves some other curves, in fact the effective roots, can be
contracted. If these curves are components of the most reducible sections of the anticanonical divisor on the weak del Pezzo,
the parameters of the code could be improved. This is the positive aspect of considering
anticanonical model of weak Del Pezzo surfaces. But we should also mention a negative one:
because of the singularity of~$X_s$, the notions of Cartier and Weil divisors are not equivalent and this makes it difficult
to calculate the minimum distance~$\dmin$ as we will see below.

The computation of~$\dmin$ reduces to compute the number:
$$
N_q\left(-K_{X_s}\right)
=
\max \left\{\#C(\F_q) \mid C\in \left|-K_{X_s}\right|\right\}.
$$
Since every curve of the linear system~$\left|-K_{X_s}\right|$ is of arithmetic genus~$1$ (adjunction formula),
all its absolutely irreducible curves have a number of rational points which is bounded above by the classic:
$$
N_q(1) = \max \left\{\#C(\F_q) \mid \text{$C$ absolutely irreducible, smooth, genus~$1$, curves  over~$\F_q$}\right\}.
$$
By the Weil-Serre bound, we know that~$N_q(1)\leq q+1 + \lfloor2\sqrt{q}\rfloor$; in fact, except for very special values of~$q$,
the Weil-Serre bound turns to be sharp:
$$
N_q(1)
=
\begin{cases}
q + \lfloor 2\sqrt{q}\rfloor &\text{if~$q= p^e$, $e\geq 5$, $e$ odd and~$p \mid \lfloor 2\sqrt{q}\rfloor$,}\\
\Weil & \text{otherwise}
\end{cases}
$$
(\cite[Chap~2, Th~6.3]{SerreRationalPoints}). Anyway, this bound does not permit
to control the number of rational points of reducible or absolutely reducible curves of~$\left|-K_{X_s}\right|$.
Due to the singularities of~$X_s$ or more specifically to the difference between the Cartier or Weil
divisors or class groups, the expectation that irreducible, but absolutely reducible curves in the linear system do not contain too many
points is more difficult to verify. Even if the Cartier class group~$\cldivcartier(X_s)$ is free of rank~$1$, generated by~$-K_{X_s}$,
this does not mean that the curves of
the linear system~$\left|-K_{X_s}\right|$ are all irreducible since they can decompose in the Weil class group~$\cldivweil(X_s)$,
that is into a sum of Weil irreducible divisors that are not Cartier divisors. To overcome this difficulty, we took fall advantage
of the fact that in the context of weak del Pezzo surfaces, explicit models of all the class groups can be computed.
This permits us to accurately measure the difference between the Cartier and the Weil divisors. This step uses some basic methods
on lattices computations. 
Then to explicitly compute the maximum~$N_q\left(-K_{X_s}\right)$, we list all the kinds of decompositions into irreducible components
that may appear in the linear system~$\left|-K_{X_s}\right|$. In general, this can be a difficult issue but in our context this task is
greatly facilitated by the fact that all the considered surfaces are blowing-up and down of the projective plane: as explained
in Hartshorne's classic \cite[Chap~V, beginning of~\S4 \& Remark~4.8.1]{Hartshorne}, we are brought back to the study of some sub-linear
systems of plane curves.

We choose examples that illustrate the variety of situations that may occur.
In the column~$\cldivcartier(X_s) \hookrightarrow \cldivweil(X)$ in the tabular below, we see that
the Cartier class group~$\cldivcartier(X_s)$ always embeds in
the Weil class group~$\cldivweil(X_s)$, and via this embedding~$\cldivcartier(X_s)$
may be equal to~$\cldivweil(X_s)$, or of finite index into~$\cldivweil(X_s)$, or of positive co-rank into~$\cldivweil(X_s)$.
Note also that the lattice~$\cldivcartier(X_s)$ is always free, whereas~$\cldivweil(X_s)$ may have a torsion subgroup.
In the column~$N_q\left(-K_{X_s}\right)$, one can see that this is not always the absolutely irreducible curves of the linear
system that contains the maximum of rational points. Only a case-by-case proof and a carefully study of all the geometric properties
permits to estimate the three invariants~$[n,k,\dmin]$ that are contained in the last column.
%
%
$$
{\renewcommand{\arraystretch}{1.5}
\begin{array}{|c|l|l|l|l|l|}
\hline
&\text{Deg.} &\text{Sing.} & \cldivcartier(X_s) \hookrightarrow \cldivweil(X) & N_q\left(-K_{X_s}\right) &[n,k,\dmin]
\\
\hline
\S\ref{sDeg6A1} & 6 & \mathbf{A}_1 & 2\Z\hookrightarrow\Z&2q+1&\left[q^2+1,7,q^2-2q\right]
\\
\hline
\S\ref{sDeg52A1} & 5 & 2\mathbf{A}_1 &\Z\oplus 2\Z\hookrightarrow \Z\oplus\Z &2q+2& \left[q^2+q+1,6,q^2-q-1\right]
\\
\hline
\S\ref{sDeg4A1} & 4 & \mathbf{A}_1 & \Z\simeq\Z &\leq N_q(1)& \left[q^2-q+1,5,\geq q^2-q+1-N_q(1)\right]
\\
\hline
\S\ref{sDeg44A1} & 4 & 4\mathbf{A}_1 &\Z\hookrightarrow\Z\oplus\Z/2\Z &\leq N_q(1)& \left[q^2+1,5,\geq q^2+1-N_q(1)\right]
\\
\hline
\S\ref{sDeg4A2} & 4 & \mathbf{A}_2 &\Z\simeq\Z &\leq N_q(1)& \left[q^2+1,5,\geq q^2+1-N_q(1)\right]
\\
\hline
\S\ref{sDeg4D5} & 4 & \mathbf{D}_5 &4\Z\hookrightarrow\Z &2q+1& \left[q^2+q+1,5,q^2-q\right]
\\
\hline
\S\ref{sDeg3A1} & 3 & \mathbf{A}_1 &\Z\simeq\Z &\leq N_q(1)& \left[q^2+1,4,\geq q^2+1-N_q(1)\right]
\\
\hline
\S\ref{sDeg33A2} & 3 & 3\mathbf{A}_2 &\Z\hookrightarrow\Z\oplus\Z/3\Z  &\leq N_q(1)& \left[q^2+q+1,4,\geq q^2+q+1-N_q(1)\right]
\\
\hline
\end{array}}
$$
In the tabular above, the inequality~$N_q\left(-K_{X_s}\right) \leq N_q(1)$ means that the curves of the linear system~$\left|-K_{X_s}\right|$
that contain the maximum number of points are the absolutely irreducible ones. They are all of arithmetic genus~$1$, but it may happen that
none of these curves is maximal (i.e. has a number of rational points equal to~$N_q(1)$). This is why in these cases, one can only
give an upper bound for~$N_q\left(-K_{X_s}\right)$ and thus a lower bound for~$\dmin$. It turns out that we recover two examples
of Koshelev \cite{Koshelev} (\S\ref{sDeg44A1} and~\ref{sDeg33A2}), where he proved that the linear systems cannot contain a maximal
genus one curve for certain finite fields. This permits to
increase the lower bound of the minimum distance by one over these fields.

All the presented codes can be easily constructed using a mathematics software system. On the second
author's {\tt webpage}, we put a {\tt magma} program that permits to construct all our codes.

\section{Generalities on weak del Pezzo surfaces}

Let~$k$ be a finite field (most of the results remain true on any field), $\overline{k}$ its algebraic closure,
$\Gamma = \Gal(\overline{k}/k)$
its absolute Galois group and let~$\sigma$ be the Frobenius automorphism.

 In this section, we recall the classical properties of del Pezzo surfaces. In particular, we focus on the
specificities of {\em non ordinary weak} del Pezzo surfaces compared to the {\em ordinary ones}. The essential references for the content
of this section are the book of Manin~\cite{Manin} or the more recent one of Dolgachev~\cite[\S8]{Dolgachev}.

\subsection{Ordinary versus non ordinary weak del Pezzo surfaces}

There are several definitions of a del Pezzo surfaces, even in the Dolgachev's classic \cite{Dolgachev}; let us start with
the definition~8.1.18 of this book.

\begin{defi}
A smooth projective surface~$X$ is a {\bfseries weak del Pezzo} surface if its anticanonical divisor~$-K_X$ is:
\begin{enumerate}
\item big, which means that~$K_X^{\cdot 2} > 0$,
\item\label{item_nef} and nef, which means that~$(-K_X) \cdot D \geq 0$ for any effective divisor~$D$ on~$X$.
\end{enumerate}
The self-intersection~$K_X^{\cdot 2}$ is the {\bfseries degree} of the del Pezzo surface~$X$.
\end{defi}

Thanks to the Nakai-Moishezon criterion \cite[Chap~V, Theorem~1.10]{Hartshorne}, these kinds of surfaces are divided into
two cases:
\begin{enumerate}
\item either the inequalities in~$(\ref{item_nef})$ are all strict ($(-K_X) \cdot D > 0$):
the anticanonical divisor is thus {\bfseries ample} and we say that the del Pezzo surface is an {\bfseries ordinary} one;
\item or there exists an effective divisor~$D$ such that~$(-K_X)\cdot D = 0$:
the anticanonical divisor is {\bfseries not ample} and we say that the del Pezzo surface is a {\bfseries non ordinary} one.
\end{enumerate}
These properties have consequences on the {\bfseries negative curves} on~$X$, those whose self-intersection is negative.
Indeed, let~$C$ be an absolutely irreducible curve on~$X$ of
arithmetic genus~$\gamma(C)$. By adjunction formula, we know
that~$C^{\cdot 2} = 2\gamma(C)-2 + C\cdot(-K_X)$ and since~$\gamma(C)\geq 0$ and~$C\cdot(-K_X)\geq 0$, we deduce that~$C^{\cdot 2} \geq -2$.
Thus, negative curves on~$X$ have self-intersection~$-2$ or~$-1$. Moreover~$C^{\cdot 2} = -2$ if and only if~$\gamma(C) = 0$
and~$C\cdot(-K_X)=0$; this means that only non ordinary del Pezzo surfaces can contain $(-2)$-curves. We also prove the same way that
$(-1)$-curves on weak del Pezzo surfaces must have arithmetic genus equal to~$0$.
This motivates the following definition which deals with negative curves but also divisor classes of curves.

\begin{defi}
Let $X$ be a weak del Pezzo surface over a field~$k$, let~$\overline{X} = X\otimes\overline{k}$ be its extension
to the algebraic closure~$\overline{k}$ and let~$\cldivweil(\overline{X})$ denote the divisor class group of~$\overline{X}$.
\begin{enumerate}
\item A divisor class~$D\in \cldivweil(\overline{X})$ is an {\bfseries exceptional class} if~$D^{\cdot 2}=D\cdot K_X=-1$;
an absolutely irreducible curve~$C$ on $X$ whose class is exceptional is an {\bfseries exceptional curve}.
\item A divisor class~$D\in \cldivweil(\overline{X})$ is a {\bfseries root} if~$D^{\cdot 2}=-2$ and $D\cdot K_X=0$;
a curve~$C$ on $X$ whose class is a root is an {\bfseries effective root} and if such a curve is absolutely irreducible then~$C$
is called a \emph{$(-2)$-curve}.
\end{enumerate}
\end{defi}

It is well known that the geometry of weak del Pezzo surfaces depends to a large extent of these {\bfseries negative curves}.
For example,
if~$X$ is a weak ordinary del Pezzo surface then all the exceptional classes are the classes of a (unique) exceptional curve and no root
is effective. On the contrary, if~$X$ is weak non ordinary del Pezzo surface then some exceptional classes may be represented by reducible
curves and some roots are effective. These differences of behaviours appear naturally in the blowup description of the generalized
del Pezzo surfaces.

\subsection{The blow-up model}


Over~$\overline{k}$, every del Pezzo surface can be obtained by a sequence of blowing ups starting from the projective plane~$\P^2$.
This description makes most of the invariants of the surface very explicit.

Recall that if~$\pi : Y \to X$ is the blowing up of a smooth surface~$X$ at a point~$p$, with exceptional divisor~$E$,
then~$\cldivweil(Y) = \pi^*\cldivweil(X) \oplus \Z E$, the intersection pairing on~$Y$
satisfying~$E^2 = -1$, $\pi^*D \cdot E = 0$ and~$\pi^* D \cdot \pi^* D' = D\cdot D'$ for all divisors~$D$ and~$D'$ of~$X$
(the blowing up is an isometry for the intersections pairings). Moreover~$K_{Y} = \pi^* K_X + E$.

We recall that~$r\leq 8$ points in~$\P^2(\overline{k})$ are said to be
\begin{itemize}
\item in {\em general position} if and only if no three lie on a line, no six lie on a conic, and there is no cubic through seven of them having a singular point at the eighth;
\item in {\em almost general position} if and only if no four lie on a line and no seven lie on a conic.
\end{itemize}
Del Pezzo surfaces can always be described as follows \cite[Th~8.1.15]{Dolgachev}. 

\begin{theo}
Let~$X$ be a generalized del Pezzo surface over $k$ and let~$\overline{X} = X\otimes_k \overline{k}$ its extension to~$\overline{k}$.
If~$X$ is of degree~$d$, with~$3\leq d\leq 6$, then~$\overline{X}$ is the blowing up of~$\P^2$
at $r = 9-d$ points~$p_1,\ldots,p_r$ in almost general position; more precisely~$\overline{X}$ results in $r$ successive
blowups~$\pi_1,\ldots,\pi_r$:
$$
\overline{X} \overset{\pi_r}{\longrightarrow} X_{r} \longrightarrow \cdots \longrightarrow X_2 \overset{\pi_1}{\longrightarrow} X_1 := \P^2_{\overline{k}}
$$
where~$p_i\in X_{i}$ are in almost general position.
\end{theo}

Let~$E_0$ be the class of a line in~$\P^2$ and let~$E_1,\ldots,E_r$ be the exceptional curves at each stage. Then the divisor class group
of~$\overline{X}$ with its intersection pairing can be easily described:
\begin{align*}
&\cldivweil(\overline{X}) = \Z E_0 \oplus \Z E_1 \oplus \cdots \oplus \Z E_r
&
&\Mat\left(\_\cdot\_, (E_i)_{0\leq i \leq r}\right)
=
\Diag(1,-1,\ldots,-1),
\end{align*}
where~$\Diag$ denotes the diagonal matrix. This also gives explicitly the canonical class:
\begin{equation}\label{eqCanonicalClass}
K_X = -3E_0 + \sum_{i=1}^r E_i.
\end{equation}
The negative classes can be expressed in terms of the basis~$E_0,E_1,\ldots,E_r$ \cite[\S8.2]{Dolgachev}.
\begin{center}
\begin{tabular}{|p{0.2\textwidth}|p{0.35\textwidth}|p{0.35\textwidth}|}
\hline
{\bf Name} & {\bf Exceptional classes} $E$ & {\bf Roots} $R$ \\
\hline
{\bf Conditions} & $E^{\cdot 2} = -1$ and $E\cdot (-K_X) = 1$
                  & $R^{\cdot 2} = -2$ and $R\cdot K_X = 0$
                  \\ 
\hline
{\bf Expression}
&
$E_i$,\hfill$i\in\{1,\ldots,r\}$
&
$R_{ij} = E_i-E_j$,\hfill$\{i,j\}\subset\{1,\ldots,r\}$
\\
{\bf in terms of the $E_i$}
&
$E_{ij} = E_0-E_i-E_j$,\hfill$\{i,j\}\subset\{1,\ldots,r\}$
&
$R_{ijk} = E_0-E_i-E_j-E_k$, $-R_{ijk}$
\\
&
$E_{i_1\cdots i_5} = 2E_0 - \sum_{j=1}^5 E_{i_j}$
&
\hfill$\{i,j,k\}\subset \{1,\ldots,r\}$
\\
&
\hfill$\{i_1,i_2,i_3,i_4,i_5\}\subset\{1,\ldots,r\}$
& $R_{i_1\cdots i_6} = 2E_0 - \sum_{j=1}^6 E_{i_j}$, $-R_{i_1\cdots i_6}$
\\
&&\hfill$\{i_1,i_2,i_3,i_4,i_5,i_6\}\subset\{1,\ldots,r\}$
\\
\hline
\end{tabular}
\end{center}

We note that, in the notation~$E_{ij}$, the indices are unordered (which leads to~$\binom{r}{2}$ possibilities), whereas they are ordered
in the notation~$R_{ij}$ since~$R_{ji} = -R_{ij}$ (which leads to~$2\binom{r}{2}$ possibilities).

Not all these divisor classes are effective and the effectiveness of certain of these classes differentiate some types of Del Pezzo surface.

$\bullet$ In the ordinary case, each exceptional class of divisor is represented by a unique irreducible curve.
Either it is one exceptional curve~$E_i$ for some~$1\leq i\leq r$ or the strict transform of the line of~$\P^2$ passing
through~$p_i$ and~$p_j$ for the class~$E_{ij}$ or the strict transform
of the (unique) conic of~$\P^2$ passing through
the~$p_{i_1},\ldots,p_{i_5}$
for~$E_{i_1i_2i_3i_4i_5}$. Their intersection graph is an important
invariant of the ordinary Del Pezzo surfaces; figures of these graphs for~$3\leq r\leq 5$ can be found in Manin~\cite[\S26.9]{Manin} or
in Dolgachev \cite[\S8.6.3, Figure~8.5]{Dolgachev}. As for the root classes, no one is
effective.

$\bullet$ In the non ordinary cases, where the points are no longer in {\em general} position but only in {\em almost general} position,
the exceptional divisors
are still effective but not necessarily represented by irreducible curves anymore.

For example, if~$p_1,p_2,p_3$ are collinear then the
root~$R_{123}$ becomes effective since it is the class of the strict transform of the line passing through~$p_1,p_2,p_3$ and the four
exceptional classes~$E_{12},E_{13},E_{23},E_{12345}$ are represented by reducible curves since
\begin{align*}
&E_{12} = R_{123} + E_3,
&
&E_{13} = R_{123} + E_2,
&
&E_{23} = R_{123} + E_1,
&
&E_{12345} = R_{123} + E_{45}.
\end{align*}
In this case, all other exceptional divisors are still represented by irreducible curves.

Another simple example: if~$p_2$ is chosen to be on~$E_1$, $p_2 \succ p_1$, then the root~$E_1-E_2$ becomes effective since it is the
strict transform of~$E_1$. The exceptional classes~$E_1$ and~$E_{1j}$ $j\not=1,2$ are no longer represented by irreducible curves.

In general, a result of Demazure states that exceptional divisors that are represented by irreducible curves
are characterized by the fact that they intersect non negatively ($\geq 0$) all the irreducible roots \cite[Proposition~5.5]{CorayTsfasman}.

Another general result states that the set of irreducible roots (the effective classes represented by an irreducible curve)
is necessarily a free family in~$\cldivweil(X\otimes\overline{\FF}_q)$ (see loc. cit.). In particular, there are at most $r$ effective
irreducible roots. The lattice generated by the effective roots plays a crucial role.

\begin{defi}
Let~$X$ be a generalized del Pezzo surface over $k$ and let~$\overline{X} = X\otimes_k \overline{k}$ be its extension to~$\overline{k}$.
We denote by~$\overline{\RR}$ the sub-lattice of~$\cldivweil(\overline{X})$ generated by the effective roots and by~$\RR$
sub-lattice of~$\cldivweil(X)$ defined by~$\RR = \overline{\RR}^\Gamma$.
\end{defi}

Following Coray and Tsfasman (see loc. cit.), an important invariant of a weak Del Pezzo surface is the graph of negative curves, which is an analog of the intersection graph
of the exceptional divisors/curves introduced above in the ordinary case. To take into account the fact that the surface may be non
ordinary, the set of
vertices is modified: the vertices corresponding
to reducible exceptional divisors are cancelled, while vertices corresponding to effective and irreducible roots are added.

\subsection{The anticanonical model~$X_s$}

\subsubsection{The morphism induced by the anticanonical divisor~$-K_X$}

Let~$X$ be a del Pezzo surface of degree~$d$, with~$3\leq d\leq 6$ and whose canonical divisor is denoted by~$K_X$.
In the ordinary case, the anticanonical divisor~$-K_X$ is known to be very ample and it induces a projective embedding
of~$X$ into~$\P^d = \P(H^0(X,-K_X))$ (see~\S\ref{sEvaluationCodes} for a review about the space of global sections of a divisor)
In the non ordinary case, the anticanonical class~$-K_X$ is no longer ample but its linear system remains base point free and gives a morphism from~$X$ to
a projective space:

\begin{defi}
Let~$X$ be a weak del Pezzo surface of degree~$d$, with~$3\leq d\leq 6$. The image~$\plongement(X)$,
where~$\plongement : X \to \P\left(H^0(X,-K_X)\right) = \P^d$ is the projective morphism associated to the anticanonical
divisor~$-K_X$ is called the anticanonical model of~$X$ and is denoted by~$X_s$. We put~$K_{X_s} = \plongement_*(K_X)$.
\end{defi}

This kind of del Pezzo surface corresponds to the definition~8.1.5 in Dolgachev \cite{Dolgachev}.
The fifth talk of Demazure \cite[Expos\'e~V]{Demazure} on del Pezzo surfaces contains all the main properties of this
anticanonical model.

\begin{prop}
The morphism~$\plongement$ satisfies:
\begin{enumerate}
\item it is not a projective embedding (since~$-K_X$ is not ample) but the image~$X_s$ is a normal surface
whose singularities are rational double points;
\item it is the minimal desingularization of~$X_s$, it contracts all the irreducible effective roots on~$X$ into the singular points
and nothing else;
\item the Weil divisor~$K_{X_s}=\plongement_*(K_X)$ is a Cartier divisor  of~$X_s$ which satisfies~$\plongement^*(K_{X_s}) =K_X$;
\end{enumerate}
\end{prop}

For each singularity, the exceptional divisor of its minimal resolution is a sum of irreducible effective roots $(R_i)$ with $R_i\cdot R_j\in\{0,1\}$ for any $i\neq j$. As usual in the $\A\D\E$ classification of rational double points, we describe the type of a singularity by its dual graph: its vertices correspond to the above roots, and there is an edge between the two vertices when the corresponding roots meet. 
In the examples below, the types of rational double points that appear correspond to the graphs:
$$
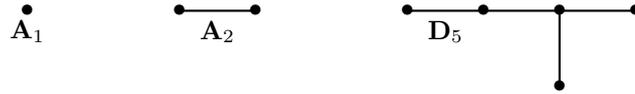
\begin{tikzpicture}[baseline=0]
\draw (0,0) node {$\bullet$} node[below] {$\mathbf{A}_1$};
\draw[thick] (2,0) node {$\bullet$} -- (3,0) node {$\bullet$} node[below,midway] {$\mathbf{A}_2$};
\draw[thick] (5,0) node {$\bullet$} -- (6,0) node {$\bullet$} node[below,midway] {$\mathbf{D}_5$} -- (7,0) node {$\bullet$} -- (8,0) node {$\bullet$};
\draw[thick] (7,0) -- (7,-1) node {$\bullet$};
\end{tikzpicture}
$$
As mentioned in the last item, since the anticanonical model of a non ordinary weak del Pezzo surface is not smooth but only normal,
a Weil divisor may not be Cartier and the class groups of Cartier or Weil divisor may differ.

\subsubsection{Cartier versus Weil divisors and class groups}\label{sCartierWeil}

Let~$X$ be a normal surface; let~$k(X)$ be its field of rational functions and~$\O_X$ its structural sheaf.
We need to review some general facts about divisors in
such surfaces (see Liu \cite[\S7.1 \& 7.2]{Liu} for more details).

$\bullet$ A {\em prime Weil divisor}~$X$ is a prime closed sub-variety of codimension~$1$ and the
group of {\em Weil divisors~$\divweil(X)$} is the free abelian group generated by prime Weil divisors.
A Weil divisor~$D$ can be written~$\sum_i n_i C_i$ where the~$C_i$'s are irreducible curves on~$X$ and where the~$n_i$ are integers of which only a finite number are non zero. Such a divisor is said {\em effective} if~$n_i\geq 0$ for all~$i$.
Since~$X$ is normal, it is regular in codimension~$1$ and to each rational function~$f\in k(X)$, one can associate a Weil
divisor~$(f)$ which is called {\em principal}. The set of principal divisors is a sub-group of~$\divweil(X)$.

$\bullet$ A {\em Cartier divisor}, or a {\em locally principal divisor}~$D$ is a global section of the
sheaf~$k(X)^\times/\O_X^\times$; it consists in a
collection~$(U_i,f_i)_{i\in I}$ where~$(U_i)_{i\in I}$ is an open covering of~$X$ and where the~$f_i$'s are rational functions
such that the quotients~$f_i/f_j$ have neither zeroes nor poles on~$U_i \cap U_j$, i.e. such that~$f_i/f_j\in\O_X^\times(U_i \cap U_j)$.
Two collections~$(U_i, f_i)_{i\in I}$ and~$(V_j, g_j)_{j\in I}$ represent the same Cartier divisor if on~$U_i\cap V_j$, the functions~$f_i$
and~$g_j$ differ by a multiplicative factor in~$\O_X^\times(U_i\cap V_j)$ for every~$i,j$. The set of Cartier divisors can be
turned into an abelian group which we denote by~$\divcartier(X)$. A Cartier divisor is called {\em effective} if
it can be represented by a collection~$(U_i,f_i)$ with~$f_i\in\O_X(U_i)$ for every~$i$. A principal Cartier divisor is
represented by a collection~$(X,f)$, where~$f\in k(X)^\times$. The set of principal divisors is also a sub-group of~$\divcartier(X)$.

$\bullet$ To each Cartier divisor~$D$ one can associate a Weil divisor and this correspondence induces a
group homorphism~$\divcartier(X)\to\divweil(X)$; since~$X$ is supposed to be normal, this morphism is injective
\cite[Chap~7, Prop~2.14]{Liu} and it sends an effective Cartier divisor to an effective Weil one.

$\bullet$ The quotients of the divisor groups~$\divcartier(X)$ and~$\divweil(X)$ by the principal divisors
are denoted~$\cldivcartier(X)$ and~$\cldivweil(X)$. The previous correspondence induces an injective
homorphism~$\cldivcartier(X)\to\cldivweil(X)$.

\medbreak

These are general facts, but in the context of weak del Pezzo surfaces, we are able to be much more explicit. 
In particular, one can relate the two groups~$\cldivcartier(X_s)$ and~$\cldivweil(X_s)$ to the
group~$\cldivweil(X)=\cldivcartier(X)$.

\begin{prop}\label{propCaClCl}
Let~$X$ be a weak del Pezzo surface over~$k$ and let~$X_s$ be its anticanonical model. Over~$\overline{k}$, one has the two exact sequences:
\begin{align}
&0 \longrightarrow \overline{\RR} \longrightarrow \cldivweil(\overline{X}) \longrightarrow \cldivweil(\overline{X}_s) \longrightarrow 0
&
&\Longrightarrow
&
&\cldivweil(\overline{X}_s) = \cldivweil(\overline{X})/\overline{\RR},\label{eqCl}\\
&0 \longrightarrow \cldivcartier(\overline{X}_s) \longrightarrow \cldivweil(\overline{X}) \longrightarrow \Hom\left(\overline{\RR},\Z\right)
&
&\Longrightarrow
&
&\cldivcartier(\overline{X}_s) = \overline{\RR}^\perp,\label{eqPic}
\end{align}
where the arrow~$\cldivweil(\overline{X}) \to \Hom\left(\overline{\RR},\Z\right)$
is given by~$D \mapsto [R \mapsto D\cdot R]$, and where~$\overline{\RR}^\perp = \{D\in \cldivweil(\overline{X}) \mid D\cdot R = 0,\; \forall R\in\overline{\RR}\}$. Over~$k$, one has~$\cldivweil(X)=\cldivcartier(X)=\cldivcartier(\overline{X})^\Gamma=\cldivweil(\overline{X})^\Gamma$ for $X$, and for its anticanonical model:
\begin{align}
&0 \longrightarrow \RR \longrightarrow \cldivweil(X) \longrightarrow \cldivweil(X_s) \longrightarrow 0
&
&\Longrightarrow
&
&\cldivweil(X_s) = \cldivweil(X)/\RR,\label{eqCl_k}\\
&0 \longrightarrow \cldivcartier(X_s) \longrightarrow \cldivweil(X) \longrightarrow \Hom\left(\overline{\RR},\Z\right)^\Gamma
&
&\Longrightarrow
&
&\cldivcartier(X_s) = \RR^\perp\label{eqPic_k}
\end{align}
Moreover we have an isomorphism $\cldivweil(X_s)\simeq \cldivweil(\overline{X}_s)^\Gamma$.
\end{prop}

\begin{proof}
Let~$R$ be the union of effective roots in~$X$ and let~$U  = X \setminus R$ be the open complementary. By a result of
Hartshorne \cite[Chap~II, Prop~6.5]{Hartshorne}, we have the exact sequence:
$$
0 \longrightarrow \RR \longrightarrow \Cl(X) \longrightarrow \Cl(U) \longrightarrow 0.
$$
Let~$U_s$ be the smooth locus of~$X_s$. This open set is of codimension~$2$ in~$X_s$ and
thus~$\Cl(X_s) \simeq \Cl(U_s)$ (see loc. cit.). Since the anticanonical map~$\plongement_{-K_X}$ induces an isomorphism from~$U$ to~$U_s$
one has~$\Cl(U_s)\simeq\Cl(U)$ and the sequence~\eqref{eqCl_k} follows. Sequence~\eqref{eqCl} also follows by extending scalars
to~$\overline{k}$.
%
%
%
%
%

On the other hand, we note that the module~$\overline{\RR}$ being induced \cite[Chap~IV,\S29]{Manin}, we know
that~$H^1(\Gamma,\overline{\RR}) = 0$; thus taking the Galois invariants of~\eqref{eqCl} leads to:
$$
0 \longrightarrow \RR \longrightarrow \cldivweil(\overline{X})^\Gamma \longrightarrow \cldivweil(\overline{X}_s)^\Gamma \longrightarrow
0
$$
Now $X$ is smooth and we have~$\cldivweil(X)=\cldivweil(\overline{X})^\Gamma$ \cite[Tag 0CDS]{stacks}; we deduce the last isomorphism $\cldivweil(X_s)\simeq\cldivweil(\overline{X}_s)^\Gamma$.

The exact sequence~\eqref{eqPic} comes from Bright \cite[Prop~1]{Bright}.
We deduce the equality $\cldivcartier(\overline{X}_s) = \overline{\RR}^\perp$, and from \cite[Tag 0CDS]{stacks},
 we deduce
that~$\cldivcartier(X_s)=\cldivcartier(\overline{X}_s)^\Gamma=(\overline{\RR}^\perp)^\Gamma$.

Finally, taking the Galois invariants in the sequence~\eqref{eqPic}
gives the exact sequence in~\eqref{eqPic_k}. Now since the intersection product is invariant under the Galois action, a divisor in $\cldivweil(X)$ is orthogonal to $\overline{\RR}$ if and only if it is orthogonal to $\RR$, and we get the isomorphism in~\eqref{eqPic_k}.
\end{proof}

\subsubsection{Lattice computations} \label{sToolsLattices}

One of the key step of the study of codes from weak del Pezzo surfaces is the explicit computation of the divisor class groups
as in~\eqref{eqCl} and~\eqref{eqPic}. Such computations take place in the group~$\cldivweil(\overline{X})$, which is known to be a free
$\Z$-module of finite type endowed with the (non degenerate) intersection bilinear form
and involve the root lattices~$\overline{\RR}$, which is given by some explicit generators and which
satisfies~$\overline{\RR}\cap\overline{\RR}^\perp = \{0\}$ (the orthogonal is relative to the intersection pairing).

This is a general issue and let us consider~$C$ (for the ``class group'') a free $\Z$-module of finite rank with a non degenerate
symmetric bilinear form~$(x,y)\mapsto x\cdot y$ (for the intersection product).
Recall that a submodule~$M$ of~$C$ is a {\em direct summand} (or {\em is complemented})
if there exists a submodule~$N$ of~$C$ such that~$C = M\oplus N$; in this
case, the submodules~$M$ and~$N$ are called {\em complementary submodules} of~$C$ \cite[\S3.8,\S6.1]{Adkins_Weintraub}.
Let~$\RR$
be a submodule of~$C$ such that~$\RR \cap \RR^\perp = \{0\}$ (for the root lattice). 

In the context of modules over a principal ideal domain,
contrary to what is happening in vector
spaces over a field, even if~$\RR\cap\RR^\perp = \{0\}$, the orthogonal submodules~$\RR$ and~$\RR^\perp$ may not be complementary submodules. There are at least two different
kinds of obstructions for this.
Either the submodule~$\RR$ is not a direct summand or both submodules~$\RR$ and~$\RR^\perp$ are direct summands but they are not
complementary submodules.

In any case the smallest submodule containing~$\RR$ which is a direct summand
is called the {\em hull} of~$\RR$ and is denoted~$\RR^\sharp$. As for the submodule~$\RR^\perp$, since it is the kernel of a morphism
of free modules, it is always a direct summand. In the same way, even though~$\RR^\sharp$ and~$\RR^\perp$ are direct summand,
they may or may not be complementary submodules. These phenomenes make the description of the exact sequence:
$$
0 \longrightarrow \RR^\perp \longrightarrow C/\RR \longrightarrow C/\RR\oplus\RR^\perp \longrightarrow 0
$$
a little bit tricky  (ie the comparison between the groups of Weil classes and Cartier classes).
The main tool for the explicit computation of this sequence is the {\em Invariant factor theorem for submodules} that will be used twice
(see~\cite[Theorem~6.23]{Adkins_Weintraub}).

$\bullet$ First, we apply this result to the submodule~$\RR\subset C$: there exists a $\Z$-basis~$e_1,\ldots,e_n$ of~$C$
and~$\alpha_1\mid\cdots\mid \alpha_r$ ($r\leq n$) a sequence of
positive integers, called invariant factors, such that~$\RR = \Z \alpha_1 e_1 \oplus \cdots\oplus \Z \alpha_r e_r$
and~$\RR^\sharp = \Z e_1 \oplus \cdots\oplus \Z e_r$. The submodule~$\RR$ is a direct summand of~$C$ if and only
if~$\RR^\sharp = \RR$, if and only if the invariant factors~$\alpha_1,\ldots,\alpha_r$ are all equal to~$1$.

 Put~$M = \Z e_{r+1} \oplus \cdots\oplus \Z e_n$
then~$\RR^\sharp$ and~$M$ are complementary submodules, $C = \RR^\sharp \oplus M$, and the projections~$\iota_{\text{tors}}$
and~$\iota$ onto each factors lead to an isomorphism
\begin{align*}
\begin{array}{rcrcl}
C/\RR & \overset{\simeq}{\longrightarrow} & \RR^\sharp/\RR & \oplus & M\\
x \bmod{\RR} & \longmapsto & \iota_{\text{tors}}(x) \bmod{\RR}&+&\iota(x)    
\end{array}
\end{align*}
In other words, the projection~$\iota_{\text{tors}}$ gives an isomorphism from the torsion submodule of~$C/\RR$
to the quotient module~$\RR^\sharp/\RR$ which is isomorphic to~$\Z/\alpha_1\Z\times\cdots\times\Z/\alpha_r\Z$;
the projection~$\iota$ gives an isomorphism from the torsion-free submodule of~$C/\RR$ to the submodule~$M$ of~$C$.


$\bullet$ Since~$\RR \cap \RR^\perp = \{0\}$, we know that~$\RR^\perp$ canonically embeds in the quotient~$C/\RR$; being free of torsion it is a
submodule of the torsion-free submodule of~$C/\RR$. Via~$\iota$ it thus embeds
in~$M$. Using the Invariant factor theorem again, one can choose the basis~$e_{r+1},\ldots,e_n$ of~$M$, in
such a way that there exists~$\beta_{r+1}\mid\cdots\mid \beta_n$ such that~$\iota(\RR^\perp) = \Z \beta_{r+1}e_{r+1} \oplus\cdots\oplus\Z \beta_ne_n$.
The cokernel~$C/\RR\oplus\RR^\perp$ is then isomorphic to~$\Z/\beta_{r+1}\Z\times\cdots\times\Z/\beta_n\Z$. In particular,
the canonical embedding of~$\RR^\perp$ inside~$M$ induces an isomorphism if and only if~$\beta_{r+1},\ldots,\beta_n$ are all equal to~$1$.

In the sequel, we do not give names to the projection morphisms~$\iota_{\text{tors}},\iota$.

\section{Codes from surfaces: construction and tools for their study} \label{sTools}

In this section~$k$ is a finite field~$\F_q$.


\subsection{Evaluation codes from surfaces} \label{sEvaluationCodes}

Cartier divisors, their spaces of global sections, and
the associated complete linear systems are the main ingredients to define and to characterize the parameters of the {\em evaluation codes}
from an algebraic surfaces. Let us recall the definitions and the basic facts concerning these objects. We
consider~$X$ a (not necessarily smooth, but in fact at least normal here) irreducible surface over~$k$. We denote by~$k(X)$ its function
field and by~$\O_X$ its structural
sheaf. Let~$D = (U_i, f_i)_{i\in I}$ be a Cartier divisor on this surface~$X$.

A {\em global section} of~$D$ is a function~$s\in k(X)$ such that for every~$i\in I$, the product~$sf_i$ is regular on~$U_i$,
that is~$s f_i\in\O_X(U_i)$ for all~$i\in I$. We denote by~$H^0(X,D)$ the set of these sections;  this is a vector space
which is known to have finite dimension.

By definition, if~$s\in H^0(X, D)$ is a global section of~$D$ then
the Cartier divisor~$(U_i, s f_i)_{i\in I}$ is {\em effective}. It can be shown that two global sections of~$D$ lead to the same effective
Cartier divisor if and only if they differ by a non zero constant. This means that there is a one-to-one correspondence between the
projective
space~$\P(H^0(X, D))$ and the set of effective Cartier divisors linearly equivalent to~$D$. This last set is called the {\em complete linear
system} associated to the divisor~$D$ and is currently denoted by~$|D|$, so we have~$|D| = \P(H^0(X, D))$. An important
invariant of the divisor (or linear system) for our purpose is the maximum of rational points that can contain a
curve of~$|D|$; we put:
\begin{equation}\label{eqNqD}
N_q(D) = \max\{\#C(\F_q) \mid C\in|D|\}.
\end{equation}

\begin{defi}\label{defCode}
Let~$X$ be a (not necessarily smooth) irreducible surface over~$k$, let~$D$ be a Cartier divisor of~$X$ and
let~$\Pcal = \{p_1,\ldots,p_n\}$ be a set of rational points of~$X$. The {\bfseries evaluation code}~$\Ccal_X(D,\Pcal)$ is the image of
the evaluation map
$$
\begin{array}{cccc}
H^0(X, D) & \longrightarrow & k^n\\
s & \longmapsto & (sf_{i_p})(p)
\end{array}
$$
where for each point~$p$, the index~$i_p$ is chosen in such a way that~$p \in U_{i_p}$.
\end{defi}

In the preceding definition, the choice of~$i_p$ may be not unique but different choices~$i_p,j_p$ of these indices lead to homothetic
codes since the quotients~$f_{i_p}/f_{j_p}$ are non vanishing regular functions on~$U_{i_p}\cap U_{j_p}$.

The usual parameters
of the evaluation code are related with some invariants of the surface.

\begin{prop} \label{propParameters}
If the evaluation map is injective, a~$\Ccal_X(D,X(\F_q))$ has
\begin{enumerate}
\item {\bfseries length} equal to~$\#X(\F_q)$ the number of rational points of~$X$,
\item {\bfseries dimension} equal to the dimension of the space~$H^0(X,D)$ of global sections,
\item {\bfseries minimum distance} bounded below by~$n - N_q(D)$, where~$N_q(D)$ is defined in~\eqref{eqNqD}.
\end{enumerate}
\end{prop}


Thanks to this proposition, it is worth noticing that bounding below the minimum distance of an evaluation code~$\Ccal_X(D,X(k))$ from a
surface~$X$ reduces to bounding above~$N_q(D)$ the number of points of the curves of the linear system~$|D|$ associated to the divisor~$D$.
The fewer the number of rational points of the curves in the linear system~$|D|$, the higher the minimum distance.

\subsection{Blowing up, divisors and (non) complete linear systems}  \label{sToolsVirtual}

One of the key tools of the construction of the codes from Del Pezzo surfaces are the blowing-up or the blowing down depending
on the sense of the arrows.

Let~$\pi : Y \to X$ be a sequence of blowing ups where all the surfaces involved are supposed to be smooth, except the last one~$X$ which
is only supposed to be normal. Such a morphism leads to two natural maps involving different kinds of divisors and divisor class groups.

$\bullet$ First, starting from a Cartier divisor of~$X$, the pullback~$\pi^* D$ is the Cartier divisor on~$Y$ defined locally
by~$(\pi^{-1}U_i,f_i\circ\pi)$. This lead to a morphism~$\pi^*: \divcartier(X) \longrightarrow \divcartier(Y)$.

$\bullet$ Secondly, it can be shown that if~$C$ irreducible effective Weil divisor of~$Y$, then~$\pi(C)$ is either a point
or an irreducible effective Weil divisor of~$X$. Then the map~$\pi_*: \divweil(Y) \longrightarrow \divweil(X)$
defined by~$\pi_*(C) = 0$ if~$\pi(C)$ is a point and~$\pi_*(C) = \pi(C)$ otherwise extends to a group homorphism
\cite[Chap~9, Lem~2.10]{Liu}.

$\bullet$ Moreover, for every Cartier divisor~$D$ of~$X$, one has~$\pi_*\left(\pi^* D\right) = D$
\cite[Chap~9, Prop~2.11]{Liu} where~$\pi_*$ is applied to the Weil divisor of~$Y$ associated to the Cartier divisor~$\pi^* D$.

$\bullet$ These two maps induce two homomorphisms~$\pi^* : \cldivcartier(X) \to \cldivcartier(Y)$
\cite[Chap~7, Def~1.34]{Liu} and~$\pi_* : \cldivweil(Y) \to \cldivweil(X)$ \cite[Chap~1, Th~1.4]{FultonIT}.

$\bullet$ At the level of global sections and linear systems, the map~$\pi^*$ also induces isomorphisms:
\begin{align*}
&
\begin{array}{ccc}
H^0(X,D) & \overset{\simeq}{\longrightarrow}&H^0(Y,\pi^*D)\\
s & \longmapsto & s\circ\pi
\end{array}
&
\begin{array}{ccc}
\left|D\right|_X & \overset{\simeq}{\longrightarrow}& \left|\pi^*D\right|_Y\\
C & \longmapsto & \pi^*C
\end{array}
\end{align*}
where~$D$ is a Cartier divisor on~$X$ (\cite[Exposé V, Cor~2]{Demazure}). Since~$\pi_*\left(\pi^* C\right) = C$, the inverse
of the right isomorphism is nothing else than~$\pi_*  \left|\pi^*D\right|_Y \longrightarrow \left|D\right|_X$.

\medbreak

We will need to describe the one-to-one correspondence~$\left|D\right|_X \longrightarrow \left|\pi^*D\right|_Y$, when the right divisor
is replaced by a divisor of the form~$\pi^*D - E$. Some natural sublinear systems appear.
Let us go step by step.

$\bullet$ Let~$\pi : Y \to X$ be the blowing-up of a smooth surface~$X$ at a point~$p\in X$ and let~$E$ be its exceptional divisor
on~$Y$. Since the surfaces are supposed to be smooth, we do not have to distinguish Cartier and Weil divisors.
We mainly focus on effective divisors and we call them {\em curves}.
Given~$C$ a curve on~$X$, then the pullback~$\pi^* C$ is called the {\em total transform} of~$C$,
the closure in~$Y$ of~$\pi^{-1}(C\setminus \{p\})$, denoted~$\widetilde{C}$, is called the {\em strict transform} of~$C$.
These two curves on~$Y$ are related by the relation:
\begin{align*}
\pi^*C = \widetilde{C} + m_p(C) E,
\end{align*}
where~$m_p(C)$ denote the multiplicity of~$C$ at the point~$p$. In particular, for any~$n\geq 0$, the divisor~$\pi^* C - n E$
is effective if and only if~$m_p(C) \geq n$.

This permits to relate the complete linear system~$|\pi^*D - nE|$ on~$Y$
to an uncomplete one on~$X$, that is~$|D-np|$ the space of curves of~$|D|$ which pass through~$p$ with multiplicity at least~$n$.
In fact this shows that the map~$C \mapsto \pi^*C - nE$ leads to a one-to-one correspondence
from~$|D-np|$ to~$|\pi^*D - nE|$
(the other way around, it says that the blowing-up permits to turn uncomplete linear systems into complete ones).

$\bullet$ The same is true if we blow up several points.
Let~$\pi : Y \to X$ be the blowing-up of a smooth surface~$X$ at some points~$p_1,\ldots,p_r$, and let~$E_1,\ldots,E_r$ be the
exceptional divisors. For~$D$ a divisor on~$X$. Let~$\left|D - n_1 p_1 - \cdots - n_r p_r\right|$ denotes the sub-linear system of the
complete
linear system~$|D|$ consisting of curves of~$|D|$ which pass through~$p_1,\ldots,p_r$ with multiplicities at least~$n_1,\ldots,n_r$.
The blowing-up permits to turn this incomplete linear system into a complete one: there is a one-to-one
correspondence between (\cite[loc. cit.]{Hartshorne}, \cite{Casas}),
\begin{align}\label{eqVirtualTransform}
&\begin{array}{ccc}
\left|D - n_1 p_1 - \cdots - n_r p_r\right| & \longrightarrow & \left|\pi^*D - n_1 E_1 - \cdots - n_r E_r\right|\\
C & \longmapsto & C^\sharp
\end{array}
&
&\text{where}
&
&C^\sharp \overset{\text{def.}}{=} \pi^* C - n_1 E_1 - \cdots - n_r E_r.
\end{align}
This curve~$C^\sharp$ is sometime called the {\em virtual transform} of~$C$. The total, strict and virtual transforms
are thus related by:
\begin{align*}
&\pi^* C = \widetilde{C} + \sum_{i=1}^r m_{p_i}(C)E_i = C^\sharp + \sum_{i=1}^r n_iE_i
&
&\Longrightarrow
&
&C^\sharp = \widetilde{C} + \sum_{i=1}^r\left(m_{p_i}(C)-n_i\right)E_i.
\end{align*}
In particular the virtual and the strict transforms coincide when~$m_{p_i}(C) = n_i$ for all~$i$.


$\bullet$ This one-to-one correspondence is still true if some points in the sequence of blowing ups are infinitely near points,
that is when
some~$p_j$ lies on the exceptional divisor of the blow up of another point~$p_i$. In order to describe this, we need to carefully define the sub-linear system associated to
a family of infinitely near points. Let us start with only two points: if~$p_1 \prec p_2$, that is if~$p_2$ lies on the exceptional
curve~$E_1$ above~$p_1$, then for~$n_1,n_2>0$, the sub-linear system of curves passing through~$p_1$ and~$p_2$ with multiplicities
at least~$n_1$ and~$n_2$ is defined by:
\begin{align*}
\left|D-n_1p_1-n_2p_2\right|
\overset{\text{def.}}{=}
\left\{C\in\left|D-n_1p_1\right|, m_{p_2}(\pi^*(C)-n_1E_1) \geq n_2\right\}
=
\left\{C\in\left|D-n_1p_1\right|, m_{p_2}(C^\sharp) \geq n_2\right\}.
\end{align*}
In particular the sub-system~$\left|D-p_1-p_2\right|$ contains all the curves of~$|D|$ that pass through~$p_1$ with tangent line at~$p_1$
equal to~$p_2$ union all the curves of~$|D|$ singular at~$p_1$; indeed, in the last
case~$C^\sharp = \pi^*(C) - E_1 = \widetilde{C} + \left(m_{p_1}(C)-1\right)E_1$ has~$E_1$ as a component and thus passes through~$p_2$
(one can check that the conditions~$p_1\in C$ and~$p_2 \in \widetilde{C}$ are not linear, which is why we choose~$p_2\in C^\sharp$ instead).

In the same way, if~$p_1 \prec p_2 \prec \cdots \prec p_r$, one can define recursively, the sub-linear
system~$\left|D-n_1p_1-\cdots-n_rp_r\right|$. With this definition, the one-to-one correspondence $\eqref{eqVirtualTransform}$
is still true.

\medbreak

Let us end by an example: the case~$X = \P^2$. If~$\ell$ denotes the class a line, and if~$E_0$ is the pullback of~$\ell$ in~$Y$,
then the curves of the
complete linear~$|d E_0 - n_1 E_1 - \cdots- n_r E_r|_Y$ on~$Y$ corresponds bijectively
to~$|d\ell - n_1 p_1 - \cdots - n_r p_r|$ the (projective) vector space consisting of plane curves of degree~$d$
passing through~$p_1,\ldots,p_r$ with multiplicities at least~$n_1,\ldots,n_r$. For small degrees, it turns out that
the irreducible decompositions of such curves can be easily described.


\subsection{Blowing up and evaluation codes} \label{sCodeBlowUpDown}

Let us return to codes and compare the evaluation codes~$\Ccal_X(D,X(k))$ and~$\Ccal_Y(\pi^*D,Y(k))$.

\begin{prop}\label{propCodeBlowingup}
Let~$X$ be a normal surface, let~$p$ be a point of~$X$ and let~$\pi : Y \to X$
be the blowing-up of~$X$ at~$p$. We denote by~$\EE$ the divisor sum of the exceptional curves.
\begin{enumerate}
\item\label{itemCodesBlowingupNR} If~$p$ is of degree~$> 1$, then the codes~$\Ccal_X(D,X(k))$ and~$\Ccal_Y(\pi^*D,Y(k))$ are equivalent; moreover
the code~$\Ccal_Y(\pi^*D - n\EE,Y(k))$ can be identified with the sub-code of~$\Ccal_X(D,X(k))$ where only the global section having
multiplicity at least~$n$ at~$p$ are evaluated.
\item\label{itemCodesBlowingupR} If~$p$ is rational, then the code~$\Ccal_Y(\pi^*D - n\EE,Y(k))$ can be identified with the sub-code
of~$\Ccal_X(D,X(k)\setminus\{p\})$ where only the global sections having multiplicity at least~$n$ at~$p$ are evaluated and to which we add
the following $(q+1)$ coordinates: the evaluations at rational points of~$\P^1$ of the homogeneous component of degree~$n$
of the local equation at~$p$ of the section.
\end{enumerate}
\end{prop}

\begin{proof}
\eqref{itemCodesBlowingupNR} The map~$s \mapsto s \circ \pi$ is a one-to-one correspondence from the
spaces of functions~$H^0(X,D)$ to~$H^0(Y,\pi^*D)$. Since the blown points are not rational, the map~$\pi$ induces a one-to-one
correspondence from~$Y(k)$ to~$X(k)$. Thus the codes~$\Ccal_X(D,X(k))$ and~$\Ccal_Y(\pi^*D,Y(k))$ must be equivalent.
By the previous correspondence the global sections of~$H^0(Y,\pi^*D - n\EE)$ are in bijection with the global sections
of~$H^0(X,D)$ that pass through~$p$ with multiplicity at least~$n$ and the last statement follows.

$\eqref{itemCodesBlowingupR}$ The set~$Y(k)$ is in one-to-one correspondence
with~$\left(X(k) \setminus \{p\}\right) \cup \EE(k)$ and we only have to compute the evaluations at the points of~$E(k)$.
We choose an open neighbourhood~$U\subset \A^2_{(x,y)}$ of~$p$ in which~$p = (0,0)$ and~$D$ has local equation~$f(x,y)=0$.
Then~$\pi^{-1}(U) \subset U \times \P^1_{(u:v)}$ with equation~$xv=yu$; there are two affine charts,~$\pi^{-1}(U) = V_1\cup V_2$,
with~$V_1\subset\A^2_{(y,u)}$ (resp.~$V_2 \subset \A^2_{(x,v)}$) with~$\pi(y,u) = (yu,y)$ (resp.~$\pi(x,v) = (x,xv)$). On~$V_1$,
the divisor~$\pi^*D-n\EE$ has local equation~$\frac{f\circ\pi}{y^n} = \frac{f(yu,y)}{y^n}$. Let~$s\circ\pi\in H^0(Y,\pi^*D-n\EE)$
then~$sf\in \O_X(U)$ has multiplicity at least~$n$ at~$p$, that is~$sf(x,y) = p_n(x,y) + p_{n+1}(x,y) + \cdots$, where~$p_n$
is homogeneous of degree~$n$. Thus~$\frac{sf\circ\pi}{y^n} = \frac{p_n(yu,y) + p_{n+1}(yu,y) + \cdots}{y^n}
= p_n(u,1) + y q(u,y)$. Evaluating at the point~$(0,u)\in \EE\cap V_1$, the section- has value~$p_n(u,1)$. The same is true on~$V_2$.
\end{proof}

The examples below provide many examples of this blowing-up operation, especially the one in section~\ref{sDeg4D5}.

%

\section{Anticanonical codes from weak del Pezzo surfaces}\label{sAnticanonicalCodes}

In this section we describe some evaluation codes from weak del Pezzo surfaces, we compute their parameters and for some of them
a generator matrix. The base field is a finite field~$\F_q$ without any other hypothesis excepts sporadically not being too small ($\F_2$
or~$\F_3$).

 In the first subsection, the general construction is given. We also fix many notations that will be used until the end
of the paper.

\subsection{General description of the codes and of the main steps of their studies}

The evaluation codes (definition~\ref{defCode}) studied in the sequel are the ones corresponding to the following choices.

\begin{defi}\label{defiCode}
Let~$X$ be a weak del Pezzo surface over~$\F_q$. We call {\bfseries anticanonical code associated to~$X$} the
evaluation code~$\Ccal_{X_s}\left(-K_{X_s}, X_s(\F_q)\right)$, where~$X_s$ is the anticanonical model of~$X$,~$-K_{X_s}$ is the
anticanonical (Cartier) divisor on~$X_s$, and where~$X_s(\F_q)$ denotes the set of rational points of~$X_s$.
\end{defi}

Note that we could have considered the evaluation codes~$\Ccal_X(-K_X,X(\F_q))$ with the same del Pezzo surfaces, but this leads to
worth codes.

In a concomitant work \cite{Classification}, we have computed explicit models for all the {\em arithmetic types}
of del Pezzo surfaces over a finite field (these types lead to a classification that is coarser than the isomorphism one
but that permit to distinguish the main arithmetic
properties of the weak del Pezzo surfaces).
Taking advantage of this knowledge, we select eight types of weak del Pezzo that are well suited for coding applications.
For each example, our starting point is a blowing-up model of the weak del Pezzo surface, then we study the
parameters length, dimension, minimum distance ($[n,k,\dmin]_q$) of the associated anticanonical code and last we give a generator matrix
(or a program to compute it).

\paragraph{Configuration to blow-up.~---} The explicit description of the surfaces~$X$ and~$X_s$ always starts from the projective
plane~$\P^2$: we first blow up a family of
(possibly infinitely near) points~$p_1,\ldots,p_r$ to obtain a smooth surface~$Y$; then we may blow down a family of (non intersecting)
exceptional curves on~$Y$ to obtain the smooth surface~$X$. Last~$X$ is mapped to a projective space corresponding to the anticanoncial
divisor to lead to the singular surface~$X_s$. To sum up, we have the following diagram:
\begin{equation}\label{eqP2YXXs}
\begin{tikzpicture}[>=latex,baseline=(M.center)]
\matrix (M) [matrix of math nodes,row sep=0.5cm,column sep=0.7cm]
{
|(Y)| Y         &         & \\
                & |(X)| X &  \\
|(P2)| \P^2     &         & |(Xs)| X_s\subset\P^{\deg(X)} \\
};
\draw[->] (Y) -- (P2) node[midway,left] {$\eclatement$} ;
\draw[->] (Y) -- (X) node[midway,above right] {$\contraction$} ;
\draw[->] (X) -- (Xs.north west) node[midway,above right] {$\plongement$} ;
\draw[dashed,->] (P2) -- (Xs) node[midway, below] {$\varepsilon$};
\end{tikzpicture}
\qquad
\begin{array}{rl}
\eclatement & \text{is a sequence of blowing ups at points~$p_1,\ldots,p_r$,}\\
\contraction & \text{is a sequence of contractions of}\\
             & \text{$(-1)$-curves~$F_1,\ldots,F_s$,}\\
\plongement & \text{is the morphism~$\plongement_{-K_X}$ associated to}\\
                                & \text{the anticanonical divisor~$-K_X$ of~$X$,}\\
\deg(X) & \text{is the degree of the del Pezzo surface~$X$, i.e.~$K_X^{\cdot 2}$.}
\end{array}
\end{equation}
All the surfaces and maps are defined over the base field~$\F_q$.
The solid arrows~$\eclatement,\contraction,\plongement$ denote maps that are morphisms whereas the dashed arrow~$\varepsilon$ denotes
a map which is a rational one. The need to introduce the auxiliary surface~$Y$ is due to the fact that some times, the surface~$X$ we
want to work with cannot be constructed directly by blowing up the plane at some points. Some contractions may be necessary
in order to work with applications that are defined over~$\F_q$ (and not only over~$\overline{\F}_q$); however this detour is not always
useful and in some examples, one has~$X=Y$ and the map~$\contraction$ is
only the identity.

\medbreak

Two of the parameters~$[n,k,\dmin]_q$ of the associated anticanonical code are easy to compute.
\begin{itemize}
\item The length is nothing else than~$\#X_s(\F_q)$. Following the process of blowing ups and down above, it is not difficult to
 compute this number since blowing up a point adds $q$ rational points or does not change the number of rational points depending on whether
the point is rational or not.
\item The dimension is nothing else than~$d+1$, where~$d$ is the degree of the del
Pezzo surface~$X$, unless the evaluation map is not injective. This can only occur
if~$\#X_s(\F_q) \leq N_q\left(-K_{X_s}\right)$ and we compute last number to estimate the minimum distance.
It turns out that the evaluation map is always injective except if the base field is~$\F_2$ or~$\F_3$ in some cases that are excluded.
\end{itemize}
As usual, the last parameter, the minimum distance, requires much more preparatory works.

\paragraph{Computation of the divisor class groups.~---} For these computations, the general ambient space is the geometric
divisor class group of~$Y$, which is known to be equal to~$\cldivweil(\overline{Y}) = \Z E_0 \oplus \Z E_1 \oplus \cdots \oplus \Z E_r$,
where, as usual,~$E_i$ denotes the exceptional curve above~$p_i$ in the sequence of blowing ups~$\pi$. In this lattice, one can
easily identify the effective roots in~$Y$, but also in~$X$ and we are able to give a basis of the
sub-lattice~$\overline{\RR}$ generated by the effective roots of~$X$ over~$\overline{\F}_q$. The other (geometric) Cartier and Weil divisor
class groups are then given by:
\begin{align*}
&\cldivweil(\overline{X}) = \left(\Z F_1 \oplus \cdots \oplus\Z F_r\right)^\perp,
&
&\cldivcartier(\overline{X}_s) = \overline{\RR}^\perp,
&
&\cldivweil(\overline{X}_s) = \cldivweil(\overline{X})/\overline{\RR}.
\end{align*}
(the left orthogonal is computed in the whole~$\cldivweil(\overline{Y})$, the middle one in the sub-lattice~$\cldivweil(\overline{X})$).
Using tools of section~\ref{sToolsLattices}, explicit bases and canonical embeddings of these geometric divisor class groups
can be computed. Taking into account the Galois action, one can also give bases and explicit canonical embedding bases of all
the arithmetic divisor class groups.
Depending on the examples, the computations are carried out in the geometric groups~$\cldivweil(\overline{X})$
and the Galois invariants are taken in the last step to return in~$\cldivweil(X)$ or we start to compute the Galois invariants and then
perform all the computations in~$\cldivweil(X)$. Thanks to Proposition~\ref{propCaClCl}, these two ways lead to the same results.

\paragraph{Types of decomposition into irreducible components in~$\left|-K_{X_s}\right|$.~---}
The minimum distance is related to the maximum number of rational points that can contain a (effective) curve in the linear
system~$\left|-K_{X_s}\right|$. To bound above this number of rational points, one way is to study how the curves in this linear
system decompose into irreducible components and use the exact number of points if known or the Weil bound if not on each components.
Thanks to section~\ref{sToolsVirtual}, and since~$\plongement^*K_{X_s} = K_X$,
we have the following one-to-one correspondences:
$$
\begin{array}{ccccc}
\left|-\chi^*K_X\right|_Y
&\overset{\simeq}{\longrightarrow}&
\left|-K_X\right|_X
&\overset{\simeq}{\longrightarrow}&
\left|-K_{X_s}\right|_{X_s}\\
C
&\longmapsto&
\contraction_{*}(C)
&\longmapsto&
\plongement_{*}\left(\contraction_{*}(C)\right)
\end{array}
$$
The first arrow consists in contracting the family of non-meeting exceptional curves~$F_i$, $1\leq i\leq s$, the second in contracting
the effective roots of~$X$. Thus we are reduced
to study the types of decompositions into irreducible components on the smooth surface~$Y$, which is easier. Indeed,
we know that~$-\chi^*K_X = dE_0 - \sum_{i=1}^r n_i E_i$ for some explicit~$d$ and~$n_i$'s; in fact, in all examples,~$d \in \{3,4\}$
and~$n_i \in \{1,2\}$. Since~$Y$ is the blowing up of~$\P^2$ at a family of points, thanks to section~\ref{sToolsVirtual},
curves of~$\left|-K_Y\right|$ are in one-to-one correspondence to the plane curves of a well specified (non complete) linear system
of~$\P^2$:
$$
\begin{array}{ccc}
\left|d\ell- n_1 p_1 - \cdots - n_r p_r\right| & \overset{\simeq}{\longrightarrow} & \left|-\chi^*K_X\right|\\
C&\longmapsto & C^\sharp
\end{array}
$$
We are thus
reduced to list all the types of decompositions into irreducible components of the plane curves of degree~$d$
passing through~$p_i$ with multiplicity~$n_i$. Since~$d\leq 4$, these absolutely irreducible components must be plane lines, conics,
cubics or quartics and an enumeration case by case can be done. More specifically, we follow the steps:
\begin{itemize}
\item degree by degree, we list all the possible absolutely irreducible curves that pass through some of the points~$p_i$;
\item we compute their Galois-orbits since if an absolutely irreducible component not defined
over~$\F_q$ appears in the decomposition with multiplicity~$m$, then the same holds for all its conjugates (this permits to get rid of
many curves because of their too high degree);
\item we combine all these irreducible curves to obtain plane curves in the expected sub-linear system.
\end{itemize}
In order to make easier
this step, we adopt the following notations and conventions. The letters~$\ell,q,c,t$ respectively denote plane lines,
quadrics (or conics), cubics and quartics. The indices below these letters are the numbers of the points through which the curve passes.
For example,~$\ell_{1}$ denotes a line that passes through~$p_1$ (but not through any other point),~$\ell_{123}$ a line that passes through~$p_1,p_2,p_3$
(if it exists), $q_{123456}$ a conic passing through the six points~$p_1,\ldots,p_6$ and~$\ell$ or~$q$ a line or quadric that do not
pass through any~$p_i$. The goal is then to combine all these irreducible plane curves to obtain a curve in the expected linear system.

At the end of this step, we are able to compute the maximum~$N_q\left(-K_{X_s}\right)$ to which the minimum distance is related
(proposition~\ref{propParameters}). Comparing with the number~$\#X_s(\F_q)$, this also permits us to exclude some too small values of~$q$
for which the evaluation map may fail to be injective.

\paragraph{Computation of the global sections from~$\P^2$.~---} Last, if we want to explicitly compute a generator matrix of the code,
we need to exhibit a basis of the sub linear system~$\left|d\ell-n_1p_1-\cdots-n_rp_r\right|$. Then, by construction we know
to which points of~$\P^2$ these functions have to be evaluated; in some cases we also need to add some extra evaluation points
corresponding to points on some exceptional curves. In any cases, one can compute a generator matrix. This last (concrete) description
turns the code into a code close to a Reed-Muller one: the space of polynomials to be evaluated has been restricted, some
of the evaluation points have been deleted, some others have been added.

If some readers want to use our code, we put on the second author's {\tt webpage}, a {\tt magma} program that permits to construct all the codes presented below.

\subsection{Degree~$6$, singularity of type~$\mathbf{A}_1$}\label{sDeg6A1}

This example corresponds to the type number~$3$ in degree~$6$ \cite{Classification}.

\paragraph{Configuration to blow-up.~---}
We blow up~$\P^2$ at three collinear points that are conjugate over~$\F_q$.
\begin{align*}
&\begin{tikzpicture}[baseline=0]
\draw[thick] (-1.5,0)  node[left] {$\ell_{123}$} -- (1.5,0) ;
\draw (-1,0) node {$\bullet$} node[below] {$p_1$};
\draw (0,0) node {$\bullet$} node[below] {$p_2$};
\draw (1,0) node {$\bullet$} node[below] {$p_3$};
\end{tikzpicture}
&
&p_2 = p_1^{\sigma},\quad p_3 = p_1^{\sigma^2}
\end{align*}
The resulting surface is a weak del Pezzo surface~$X$ whose anticanonical model is denoted~$X_s$. It has a unique singular point
of type~$\mathbf{A}_1$ which is necessarily rational.

\paragraph{Computation of the divisor class groups.~---} Over~$\overline{\F}_q$, one
has
\begin{align*}
&\cldivweil(\overline{X}) = \Z E_0 \oplus \Z E_1 \oplus \Z E_2 \oplus\Z E_3
&
&\text{and}
&
&-K_X = 3E_0 - E_1 - E_2 - E_3.
\end{align*}
There is a unique effective root, the strict transform
of the line~$\ell_{123}$ passing through the three points~$p_1,p_2,p_3$, and its class is~$E_0 - E_1 - E_2 - E_3$.
Then
\begin{align*}
\overline{\RR} &= \Z(E_0 - E_1 - E_2 - E_3)
&
\overline{\RR}^\perp &= \left\{a_0E_0 + a_1 E_1 + a_2 E_2 + a_3 E_3 \mid a_0+a_1+a_2+a_3 = 0\right\}\\
&& &=\Z(E_0-E_1) \oplus \Z(E_0-E_2) \oplus \Z(E_0-E_3)
\end{align*}
Both~$\overline{\RR}$ and~$\overline{\RR}^\perp$ are direct summand but~$\overline{\RR}$ and~$\overline{\RR}^\perp$ are not complementary submodules since~$\overline{\RR}\oplus\overline{\RR}^\perp$ is of
index~$2$ in~$\cldivweil(\overline{X})$. For a submodule complement to~$\overline{\RR}$, one can choose:
$$
\begin{array}{rcrcl}
\cldivweil(\overline{X}) &= &\overline{\RR} &\oplus &\left(\Z E_1 \oplus \Z E_2 \oplus \Z E_3\right)\\
a_0E_0+a_1 E_1 + a_2 E_2 + a_3 E_3 &= & a_0(E_0 - E_1 - E_2 - E_3) &+&(a_1+a_0)E_1+(a_2+a_0)E_2+(a_3+a_0)E_3
\end{array}
$$
This leads to the following isomorphism:
$$
\begin{array}{rcl}
  \cldivweil(\overline{X})/\overline{\RR} & \overset{\simeq}{\longrightarrow} & \Z E_1 \oplus \Z E_2 \oplus \Z E_3\\
  a_0E_0+a_1 E_1 + a_2 E_2 + a_3 E_3\bmod\overline{\RR} & \longmapsto & (a_0+a_1)E_1 + (a_0+a_2)E_2 + (a_0+a_3)E_3.
\end{array}
$$
Via this isomorphism, the submodule~$\cldivcartier(\overline{X}_s) = \overline{\RR}^\perp$ identifies with~$\Z(E_1+E_2)\oplus\Z(E_2+E_3)\oplus\Z(E_1+E_3)$
of invariant factors~$1,1,2$ in~$\Z E_1 \oplus \Z E_2 \oplus \Z E_3$.

Over~$\F_q$, to recover the class groups~$\cldivweil(X_s)$ and~$\cldivcartier(X_s)$, we only need to take the invariants under the Galois
action~$(E_0)(E_1E_2E_3)$, what is easy here. One has
\begin{align*}
&\cldivcartier(X_s) = \cldivcartier(\overline{X}_s)^\Gamma = \left(\overline{\RR}^\perp\right)^\Gamma \simeq \Z(3E_0-E_1-E_2-E_3) = \Z(-K_X),
\\
&\cldivweil(X_s) = \cldivweil(\overline{X}_s)^\Gamma = \left(\cldivweil(\overline{X})/\overline{\RR}\right)^\Gamma \simeq \Z(E_1+E_2+E_3).
\end{align*}
With these identifications, the canonical embedding of~$\cldivcartier(X_s)$ into~$\cldivweil(X_s)$ becomes:
$$
\begin{array}{rcccc}
 0&\longrightarrow& \cldivcartier(X_s) & \longrightarrow & \cldivweil(X_s)\\
  &               &-K_X & \longmapsto & 2(E_1 + E_2 + E_3)
\end{array}
$$
Thus both~$\cldivcartier(X_s)$ and~$\cldivweil(X_s)$ are free of rank~$1$, but~$\cldivcartier(X_s)$ is of index~$2$ into~$\cldivweil(X_s)$. This index has the
following consequence: even if~$\cldivcartier(X_s)$ is free of rank one generated by~$-K_{X_s}$, a Cartier divisor may decompose into a sum of equivalent Weil irreducible divisors. This explains why we need to investigate how elements of~$\left|-K_X\right|$ can decompose into
irreducible components and how the non ordinary weak del Pezzo surfaces we consider here differ from ordinary ones (compare
with~\cite{AntiCanonical}).

\paragraph{Types of decomposition into irreducible components in~$\left|-K_{X_s}\right|$.~---} In this example, there is no need to
introduce an auxiliary surface~$Y$ (one has~$Y=X$ and~$\chi$ is the identity with the notation of the beginning of this section).
Since~$-K_X = 3E_0-E_1-E_2-E_3$, the virtual transform composed with the push forward lead to a one-to-one correspondence:
$$
\begin{array}{ccccc}
\left|3\ell - p_1 - p_2 - p_3\right| & \longrightarrow & \left|3E_0 - E_1 - E_2 - E_3\right| & \longrightarrow & \left|-K_{X_s}\right|\\
C & \longmapsto & C^\sharp & \longmapsto & \plongement_*(C^\sharp)
\end{array}
$$
(the left linear system is on~$\P^2$, the middle one on~$X$ and the right one on~$X_s$).
Then we are reduced to list all the types of decompositions into irreducible components of the curves
of~$\left|3\ell - p_1 - p_2 - p_3\right|$, the sub linear system of cubics passing through the points~$p_1,p_2,p_3$.
The orbits of lines, conics, cubics which have degree at most~$3$ and pass through some of the points~$p_i$'s are
\begin{align*}
&\ell_1\cup\ell_2\cup\ell_3,
&
&\ell_{123},
&
&c_{123},
\end{align*}
(of course, implicitly~$\ell_2=\ell_1^\sigma$, $\ell_{3} = \ell_{1}^{\sigma^2}$ where $\sigma$ is a generator of $\Gal(\overline{\F}_q/\F_q)$). Indeed an absolutely irreducible conic
$q_i$ or~$q_{ij}$ cannot be defined over~$\F_q$ and they have at least three conjugates; combining these curves with their conjugates lead
to plane curves of degree greater than~$6$ and thus they cannot appear in our case. A conic~$q_{123}$ cannot be absolutely
irreducible otherwise it would have three intersection points with the line~$\ell_{123}$. Let us combine these rational irreducible
decompositions in order to construct plane curves in the expected sub-linear system.

First suppose that the decomposition contains a line.
\begin{itemize}
\item If this line is~$\ell_1$, then
its conjugates~$\ell_2,\ell_3$ must also be geometric components; the only possibility is~$\ell_1\cup\ell_2\cup\ell_3$ (line~1 in the
tabular below) which is an element of~$\left|3\ell - p_1 - p_2 - p_3\right|$.
\item  A component~$\ell$ cannot be completed by a conic passing through
the three points and thus if there is a line in the geometric decomposition, $\ell_{123}$ must be one of them. Since~$\ell_{123}$
already passes through~$p_1,p_2,p_3$ it can be completed by any conic (irreducible or not); this leads to the decompositions
of the lines~3 to~7 in the tabular below.
\end{itemize}
Last, if the decomposition does not contain any line, it must be an irreducible cubic which
passes through the three points;
this cubic can be smooth or not and we recover the two last lines of the tabular.
$$
{\renewcommand{\arraystretch}{1.5}
\begin{array}{|l|l|l|l||l|}
\hline
&\text{$\left|3\ell-p_1-p_2-p_3\right|$} & \left|-K_X\right| & \left|-K_{X_s}\right| & \text{Max}   \\
&\text{on~$\P^2$}                       & \text{on~$X$}     & \text{on~$X_s$}        & \text{nb. of pts} \\
\hline\hline
1&\ell_1 \cup \ell_2 \cup \ell_3 &\widetilde{\ell}_1 \cup \widetilde{\ell}_2 \cup \widetilde{\ell}_3 & \bigcup_{i=1}^3 \plongement_*(\widetilde{\ell}_i) & 1\\
\hline
2&\ell_{123} \cup q,\; \ell_{123} \cap q \not\subset\P^2(\F_q) &\widetilde{\ell}_{123} \cup \widetilde{q} & \plongement_*(\widetilde{q}) & q+2\\
\hline
3&\ell_{123} \cup q,\; \ell_{123} \cap q \subset\P^2(\F_q) &\widetilde{\ell}_{123} \cup \widetilde{q} & \plongement_*(\widetilde{q}) & q\\
\hline
4&\ell_{123} \cup \ell \cup \ell',\;  \ell_{123}\cap\ell\cap\ell' \not= \emptyset &\widetilde{\ell}_{123} \cup \widetilde{\ell} \cup \widetilde{\ell}' & \plongement_*(\widetilde{\ell}) \cup \plongement_*(\widetilde{\ell}') & 2q+1\\
\hline
5&\ell_{123} \cup \ell \cup \ell',\;  \ell_{123}\cap\ell\cap\ell' = \emptyset &\widetilde{\ell}_{123} \cup \widetilde{\ell} \cup \widetilde{\ell}' & \plongement_*(\widetilde{\ell}) \cup \plongement_*(\widetilde{\ell}') & 2q\\
\hline
6&2\ell_{123} \cup \ell &2\widetilde{\ell}_{123} \cup \widetilde{\ell} \cup \bigcup_{i=1}^3 E_i & \plongement_*(\widetilde{\ell}) \cup \bigcup_{i=1}^3\plongement_*(E_i)& q+1\\
\hline
7&3\ell_{123} &3\widetilde{\ell}_{123} \cup \bigcup_{i=1}^3 2E_i &   \bigcup_{i=1}^3 2\plongement_*(E_i) & 1\\
\hline
8&c_{123},\; \text{singular} &\widetilde{c}_{123} & \plongement_*(\widetilde{c}_{123}) & q+2\\
\hline
9&c_{123},\;\text{smooth} &\widetilde{c}_{123}  & \plongement_*(\widetilde{c}_{123}) & \Nqg{q}{1}\\
\hline
\end{array}}
$$
Some comments about the three first columns of the previous tabular. The unique irreducible effective root of~$X$ is nothing else than the strict
transform~$\ell_{123}$ and this explains why this curve disappears in the third column: this line on~$X$
is mapped by~$\plongement_*$ to the unique singular point~$s\in X_s$. Note also that except in the cases~6 and~7,
all the curves have exactly multiplicities~$1$ at the~$p_i$ and thus their strict or virtual transforms are equal. On the contrary,
in the remaining cases, the curves on~$X$ are the virtual transforms of the ones on~$\P^2$.
Last, in the decomposition~$\plongement_*(\widetilde{\ell})\cup\plongement_*(\widetilde{\ell}')$, it is worth noticing that irreducible
components involves divisors that are not Cartier divisors but only Weil ones on~$X_s$. Indeed the class of~$\widetilde{\ell}$
in~$\cldivweil(X)$ is~$E_0$, which is mapped to~$E_1 + E_2 + E_3$ in~$\cldivweil(X_s)$, which is not an element of~$\cldivcartier(X_s)$
(equivalently~$E_0\not\in\RR^\perp$).

Now we make some comments on the numbers of rational points.

{\bfseries Case~$1$.} Since the lines~$\ell_1,\ell_2,\ell_3$ are conjugate a rational point on their union must be at their intersection
which contains at most one point. On~$X$ the strict transforms~$\widetilde{\ell}_1,\widetilde{\ell}_2,\widetilde{\ell}_3$
do not meet the root~$\ell_{123}$ and the contraction does not add any point.

{\bf Cases~$2,3,4$ \& $5$.} If the two points of~$q\cap\ell_{123}$ are not rational then they are still unrational
on~$\widetilde{q}\cap\widetilde{\ell}_{123}$ and they are contracted to the singular point~$s$
in~$X_s$ and thus the image~$\plongement_*(\widetilde{q})$ has one more rational point; otherwise, if the two points of~$q\cap\ell_{123}$
are rational then they are
contracted in~$X_s$ and thus the image~$\plongement_*(\widetilde{q})$ looses a rational point. The same is true on lines~$4$ and~$5$.

{\bf Cases~$6$ \& $7$.} The line~$\widetilde{\ell}_{123}$ is contracted by~$\plongement_*$ and there are no rational points on
the lines~$E_i$.

{\bf Cases~$8$ \& $9$.} The starting cubic~$c_{123}$ has $(q+1)$ or less than~$\Nqg{q}{1}$ rational points depending on whether it is
singular or smooth. On~$X_s$ the number of rational points of~$\plongement_*(\widetilde{c}_{123})$ is increased by~$1$ since the
line~$\ell_{123}$ meets the cubic at three conjugate points.
The multiplicity of intersection of~$c_{123}$ and~$\ell_{123}$ at each
point~$p_i$ is one (since otherwise, these two curves would have too many intersection points counting with multiplicities). Therefore,
the blowing ups at~$p_1,p_2,p_3$ separate the strict transforms~$\widetilde{\ell}_{123}$ and~$\widetilde{c}_{123}$.
Thus~$\widetilde{c}_{123}$ and~$\plongement_*(\widetilde{c}_{123})$ are isomorphic and have the same number of rational points.
Finally, we remark that for every~$q$, one has~$\Weil \leq 2q+1$ (with equality if and only if~$q\in\{2,3,4\}$)
and thus:
$$
N_q\left(-K_{X_s}\right) = 2q+1.
$$
Last we note that, except for the cases~$1$ and~$9$, all the maximum numbers of points are in fact exact
numbers of points. Thus we are not far from having the distribution of weights if the code. 

\medbreak

Since~$p_1,p_2,p_3$ are not rational, the three blowing ups do not add any rational point and~$\#X(\F_q) = q^2+q+1$. Then,
the root~$\widetilde{\ell}_{123}$ is contracted via the anticanonical morphism and thus~$\#X_s(\F_q) = q^2+1$.
Except if~$q=2$, one has~$\#X_s(\F_q) > N_q(-K_{X_s})$ and the evaluation map is injective. With this choice of weak del Pezzo surface,
the code of definition~\ref{defiCode} satisfies the following proposition.

\begin{prop}
Let~$p_1,p_2,p_3$ be conjugate collinear point in~$\P^2_{\F_q}$, with~$q\not=2$. The anticanonical
code associated to the weak del Pezzo surface obtained
by blowing up these points has parameters~$[q^2+1,7,q^2-2q]$.
\end{prop}

\paragraph{Computation of the global sections from~$\P^2$.~---} To construct this kind of codes, one can choose~$\ell_{123}$
to be the line of equation~$Y=0$ in~$\P^2$. For
any~$\zeta \in \F_{q^3} \setminus \F_q$, the point~$p_1=(\zeta:0:1)\in\P^2$ is a degree~$3$ point whose
conjugates~$p_2=(\zeta^\sigma:0:1)$ and~$p_3=(\zeta^{\sigma^2}:0:1)$ are also in~$\ell_{123}$.
Let~$X^3+a_2 X^2 + a_1 X + a_0 \in\F_q[X]$ be the minimal polynomial of~$\zeta$ over~$\F_q$. Then, we easily verify that
$$
\left|3\ell- p_1 - p_2 - p_3\right|
=
\left\langle Y^3,Y^2X,Y^2Z,YX^2,YZ^2,YXZ,X^3+a_2 X^2Z + a_1 XZ^2 + a_0Z^3\right\rangle_{\F_q}.
$$
Last, the evaluation points are nothing else than the points of~$\P^2(\F_q) \setminus \ell_{123}(\F_q)$, plus
one point of~$\ell_{123}(\F_q)$ since the strict
transform of~$\ell_{123}$ is contracted via the anticanonical morphism. Let us denote~$(x_i:1:z_i)$, $1\leq i\leq q^2$, the first $q^2$ points,
and let us choose~$(0:0:1)\in\ell_{123}(\F_q)$, then the corresponding generator matrix of the code is:
\begin{align*}
&\begin{pmatrix}
1 & \cdots & 1 & 0\\
x_1 & \cdots & x_{q^2} & 0\\
z_1 & \cdots & z_{q^2} & 0\\
x_1^2 & \cdots & x_{q^2}^2 & 0\\
z_1^2 & \cdots & z_{q^2}^2 & 0\\
x_1z_1 & \cdots & x_{q^2}z_{q^2} & 0\\
P(x_1,1,z_1) & \cdots & P(x_{q^2},1,z_{q^2}) & a_0
\end{pmatrix},
&
&P(X,Y,Z) = X^3+a_2 X^2Z + a_1 XZ^2 + a_0Z^3.
\end{align*}
We recover the classical Reed-Muller code on~$\A^2$ of degree~$2$, augmented by one point.


\subsection{Degree~$5$, singularity of type~$2\mathbf{A}_1$}\label{sDeg52A1}

This example corresponds to the type number~$5$ in degree~$5$ \cite{Classification}.

\paragraph{Configuration to blow-up.~---} We blow up~$p_1 \prec p_2$ and~$p_3\prec p_4$, where~$p_1,p_3$ and~$p_2,p_4$ are
conjugate points of degree~$2$.
\begin{align*}
&\begin{tikzpicture}[baseline=0]
  \def\a{1.5} 
  \def\b{1.0} 
  \coordinate (p) at (-\a,0); 
  \coordinate (psigma) at (\a,0); 
  \draw[thick,<->,gray] (-2,-0.5) node[black,below left] {$\ell_{12}$} -- (-1,0.5) node[black,left,near end,xshift=-0.5em,yshift=0.5em] {$p_2$};
  \draw[thick,<->,gray] (2,-0.5) node[black,below right] {$\ell_{34}$} -- (1,0.5) node[black,right,near end,xshift=0.5em,yshift=0.5em] {$p_4$};
  \draw (p) node {$\bullet$} node[below,yshift=-0.5em] {$p_1$};
  \draw (psigma) node {$\bullet$} node[below,yshift=-0.5em] {$p_3$};
  \draw[thick] (-2.5,0) node[left] {$\ell_{13}$} -- (2,0);
\end{tikzpicture}
&
&\begin{array}{l}p_3=p_1^\sigma\\p_4=p_2^\sigma.\end{array}
\end{align*}
Since the points~$p_2,p_4$ are infinitely near the points~$p_1,p_3$, they are represented by tangent lines or directions on the picture above. The anticanonical model~$X_s$ has two singular points of type~$\mathbf{A}_1$ that are conjugate points.

\paragraph{Computation of the divisor class groups.~---} Over~$\overline{\F}_q$, one has:
\begin{align*}
&\cldivweil(\overline{X})
=
\Z E_0 \oplus \Z E_1 \oplus \Z E_3 \oplus \Z E_2 \oplus \Z E_4
&
&\text{and}
&
&-K_X=3E_0 - E_1 - E_2 - E_3 - E_4.
\end{align*}
There are two conjugate effective roots, the strict transforms of~$E_1$ and~$E_3$ in the sequence of blowing ups; their classes
are~$E_1-E_2$ and~$E_3-E_4$ in such a way that:
\begin{align*}
\overline{\RR} &= \Z (E_1-E_2) \oplus \Z (E_3-E_4),
&
\overline{\RR}^\perp
&=\left\{a_0E_0 + a_1 E_1 + a_2 E_2 + a_3 E_3 + a_4 E_4 \mid a_1=a_2, \; a_3 = a_4\right\}\\
&&&= \Z E_0 \oplus \Z (E_1+E_2) \oplus \Z (E_3+E_4).
\end{align*}
The sub-module~$\overline{\RR}$ is a direct summand, and as a complementary sub-module one can choose:
\begin{align*}
\begin{array}{{rcrcl}}
\cldivweil(\overline{X}) & = & \overline{\RR} & \oplus & \Z E_0 \oplus \Z E_2 \oplus \Z E_4\\
\sum_{i=0}^4 a_i E_i & = & a_1(E_1-E_2) + a_3(E_3-E_4) & + & a_0 E_0 + (a_1+a_2) E_2 + (a_3+a_4)E_4.
\end{array}
\end{align*}
We deduce the isomorphism:
$$
\begin{array}{rcl}
\cldivweil(\overline{X}_s) \simeq \cldivweil(\overline{X})/\overline{\RR} & \overset{\simeq}{\longrightarrow} & \Z E_0 \oplus \Z E_2 \oplus \Z E_4\\
\sum_{i=0}^4 a_i E_i\bmod{\overline{\RR}}  & \longmapsto & a_0 E_0 + (a_1+a_2) E_2 + (a_3+a_4)E_4
\end{array}.
$$
Since~$\cldivcartier(\overline{X}_s) \simeq \overline{\RR}^\perp$, this class group is a rank~$3$ free sub-group of~$\cldivweil(\overline{X})$.
Via the previous isomorphism it is mapped to the sub-group~$\Z E_0 \oplus \Z 2E_2 \oplus \Z 2E_4$,
of invariant factors~$1,2,2$.

The arithmetic groups~$\cldivcartier(X_s)$ and~$\cldivweil(X_s)$ can be computed by taking the invariants under the Galois
action which is~$(E_0)(E_1E_3)(E_2E_4)$. Via the previous isomorphism, if we set $\EE:= E_2+E_4$, the canonical embedding~$0\rightarrow\cldivcartier(X_s)\rightarrow\cldivweil(X_s)$
is only:
\begin{align*}
\underbrace{\Z E_0 \oplus \Z 2\EE}_{\simeq \cldivcartier(X_s)}
\subset
\underbrace{\Z E_0 \oplus \Z \EE}_{\simeq \cldivweil(X_s)}.
\end{align*}
In other terms, $\cldivcartier(X_s)$ and~$\cldivweil(X_s)$ are both free of rank~$2$ and via the canonical embedding, the first one has invariant
factors~$1,2$ inside the second one.

\paragraph{Types of decomposition into irreducible components in~$\left|-K_{X_s}\right|$.~---}
Since the two class groups are not rank one, one expects to find a wide variety of possible decompositions into irreducible components
for the curves in the linear system~$\left|-K_{X_s}\right|$. In order to list all these types, we start form~$\P^2$ and use the one-to-one
correspondences:
$$
\begin{array}{ccccc}
\left|3\ell - p_1 - p_3 - p_2 - p_4\right|
& \longrightarrow
& \left|3E_0 - E_1 - E_3 - E_2 - E_4\right|
& \longrightarrow
& \left|-K_{X_s}\right|\\
C & \longmapsto & C^\sharp & \longmapsto &\plongement_*\left(C^\sharp\right).
\end{array}
$$
The curves of the left linear system are nothing else than the plane cubics over~$\F_q$ passing through~$p_1,p_3$ that are either smooth
at~$p_1,p_3$ with tangent lines~$p_2,p_4$ respectively or singular at these points.

Our notations are the same: $\ell_{13}$ is the line~$(p_1p_3)$ which is rational, $\ell_{12}$ and~$\ell_{34}$ are the lines~$(p_1p_2)$
respectively~$(p_3p_4)$ (that is the lines of~$\P^2$ passing through~$p_1$, respectively~$p_3$, whose strict transform pass
through~$p_2$, respectively~$p_4$); these last two lines are conjugate. The orbits of lines, conics, cubics
having degree less than~$3$ and passing through some of the points~$p_i$'s are
\begin{align*}
&\ell_1\cup\ell_3,
&
&\ell_{13},
&
&\ell_{12}\cup\ell_{34},
&
&q_{13},
&
&q_{1234},
&
&c_{13},
&
&c_{1234}
\end{align*}
(of course, implicitly~$\ell_3=\ell_1^\sigma$, $\ell_{34} = \ell_{12}^\sigma$). We have just to combine these rational irreducible
decompositions in order to construct plane curves in the expected sub-linear system.

Suppose that there is at least one line in the absolute irreducible decomposition.
\begin{itemize}
\item If this line is~$\ell_{12}$, then by rationality,~$\ell_{34}$ is also an absolute irreducible component.
Since~$\ell_{12}\cup\ell_{34}$ already passes
through~$p_1,p_2,p_3,p_4$, one can complete by any rational line~$\ell$ or by the line~$\ell_{13}$
(see cases~$1$ and~$2$ in the tabular below).
\item 
If this line is~$\ell_{13}$, then the two incidence conditions at~$p_1$ and~$p_3$ are satisfied.
The complement component must be a (maybe reducible) conic whose strict transform passes through~$p_2$ and~$p_4$; this conic
must necessarily pass through~$p_1,p_3$.
This conditions suffice since the union of~$\ell_{13}$ with any conic passing through~$p_1,p_3$ is singular. The complement
can be the union~$\ell_{12}\cup\ell_{34}$ (same as case~$2$), or~$\ell_1\cup\ell_3=\ell_1^\sigma$, or~$\ell_{13}$ itself union any other line,
or twice~$\ell_{13}$, or~$q_{13}$, or~$q_{1234}$.
\item If this line is~$\ell$ a line that does not pass through the~$p_i$'s, then the complement conic must be either~$\ell_{12}\cup\ell_{34}$
as in first case, or a conic passing through the four points.
\end{itemize}
Last, if there is not any line in the absolute irreducible decomposition, then the cubic must be absolutely irreducible and it
has to pass through the four points.
%

That being, the possible cubics are listed below. The irreducible effective roots of~$X$ are the (conjugate) strict
transforms~$\widetilde{E}_1$ and~$\widetilde{E}_3$; since they do not meet, their contraction lead to two (conjugate) singular
points~$s$ and~$s^\sigma$ on~$X_s$.
$$
{\renewcommand{\arraystretch}{1.5}
\begin{array}{|c|l|l|l||l|l||l|l|}
\hline
&\left|3\ell-\sum_{i_1}^4 p_i\right|  &\left|-K_X\right| & \left|-K_{X_s}\right|  & \text{Max}   \\
&\text{on~$\P^2$}                   &\text{on~$X$}     &\text{on~$X_s$}        & \text{nb. of pts}\\
\hline\hline
1&\ell_{12}\cup\ell_{34}\cup\ell
&\widetilde{\ell_{12}}\cup\widetilde{\ell_{34}}\cup\widetilde{\ell}
&\plongement_*(\widetilde{\ell_{12}})\cup\plongement_*(\widetilde{\ell}_{34})\cup\plongement_*(\widetilde{\ell})
& q+2
\\
2&\ell_{12}\cup\ell_{34}\cup\ell_{13}
&\widetilde{\ell_{12}}\cup\widetilde{\ell_{34}}\cup\widetilde{\ell}_{13}\cup\widetilde{E}_1\cup\widetilde{E}_3\cup E_2\cup E_4
&\plongement_*(\widetilde{\ell_{12}})\cup\plongement_*(\widetilde{\ell}_{34})\cup\plongement_*(\widetilde{\ell}_{13})\cup\plongement_*(E_2)\cup\plongement_*(E_4)
& q+2
\\
\hline
3&\ell_{13} \cup \ell_1\cup \ell_3
&\widetilde{\ell}_{13}\cup\widetilde{\ell}_1\cup\widetilde{\ell}_3\cup\widetilde{E}_1\cup\widetilde{E}_3
&\plongement_*(\widetilde{\ell}_{13})\cup\plongement_*(\widetilde{\ell}_1)\cup\plongement_*(\widetilde{\ell}_3)
& q+2
\\
4&2\ell_{13} \cup \ell
&2\widetilde{\ell}_{13}\cup\widetilde{\ell}\cup\widetilde{E}_1\cup\widetilde{E}_3
&2\plongement_*(\widetilde{\ell}_{13})\cup\plongement_*(\widetilde{\ell})
& 2q+1
\\
5&3\ell_{13}
&3\widetilde{\ell}_{13}\cup 2\widetilde{E}_1\cup 2\widetilde{E}_3\cup E_2\cup E_4
&3\plongement_*(\widetilde{\ell}_{13}) \cup \plongement_*(E_2)\cup\plongement_*(\widetilde{E}_2)
& q+1
\\
6&\ell_{13} \cup q_{13}
&\widetilde{\ell}_{13}\cup \widetilde{q}_{13}\cup \widetilde{E}_1\cup \widetilde{E}_3
&\plongement_*(\widetilde{\ell}_{13})\cup\plongement_*(\widetilde{q}_{13})
& 2q+2
\\
7&\ell_{13} \cup q_{1234}
&\widetilde{\ell}_{13}\cup \widetilde{q}_{1234}\cup \widetilde{E}_1\cup \widetilde{E}_3\cup E_2\cup E_4
&\plongement_*(\widetilde{\ell}_{13})\cup\plongement_*(\widetilde{q}_{1234})\cup\plongement_*(E_2)\cup\plongement_*(E_4)
& 2q+2
\\
\hline
8&\ell \cup q_{1234}
&\widetilde{\ell}\cup \widetilde{q}_{1234}
&\plongement_*(\widetilde{\ell})\cup\plongement_*(\widetilde{q}_{1234})
& 2q+2
\\
\hline
9&c_{1234}
&\widetilde{c}_{1234}
&\plongement_*(\widetilde{c}_{1234})
&\Nqg{q}{1}\\
\hline
\end{array}}
$$
We draw all the preceding decompositions in order to illustrate what is going on.
The blowing up~$\pi : X \rightarrow \P^2$ is decomposed into two blowing ups~$\pi=\pi_2\circ\pi_1$, where~$\pi_1 : X_1 \to \P^2$ is the
blowing up at~$p_1$ and~$p_3$, and where~$\pi_2 : X \to X_1$ is the blowing up at~$p_2$ and~$p_4$. The left column is the drawing
of the starting configuration in~$\P^2$, the middle one the configuration after having blowing up~$p_1$ and~$p_3$, the right one
the configuration in~$X$. The operation from a column to the next one is the virtual transform. 
Curves drawn in gray are not part of virtual transform, curves drawn in red are the effective roots; these curves are contracted
in~$X_s$ (but we do not draw this step). In brackets, to the right of the name
of a curve, we put its self-intersection. 
We draw all the cases of the preceding tabular, except the cubic case.

In any case, one can verify that the union of the black
curves passes through~$p_1,p_3,p_2,p_4$ and that the divisor class is equal to~$-K_X=3E_0-E_1-E_3-E_2-E_4$.

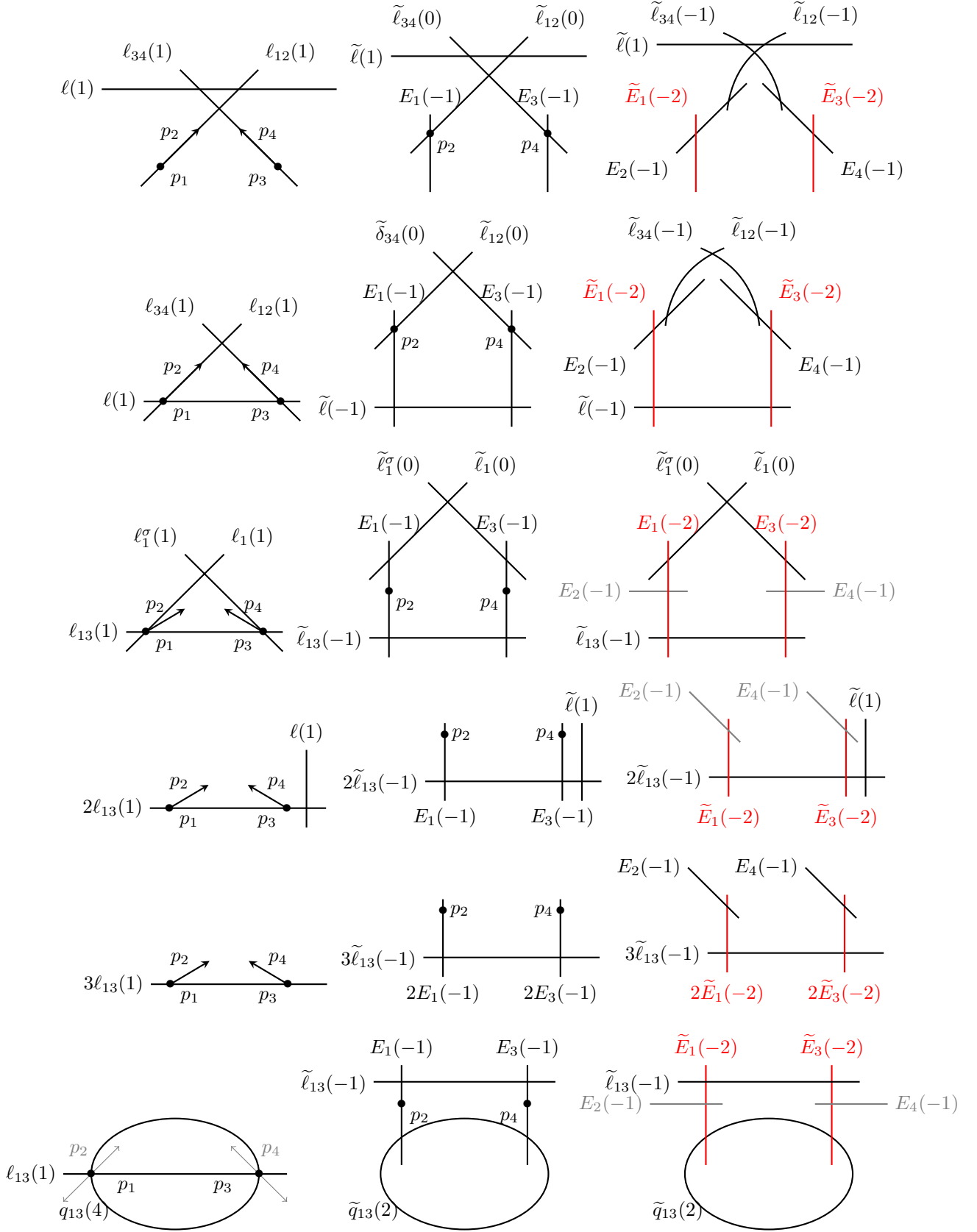
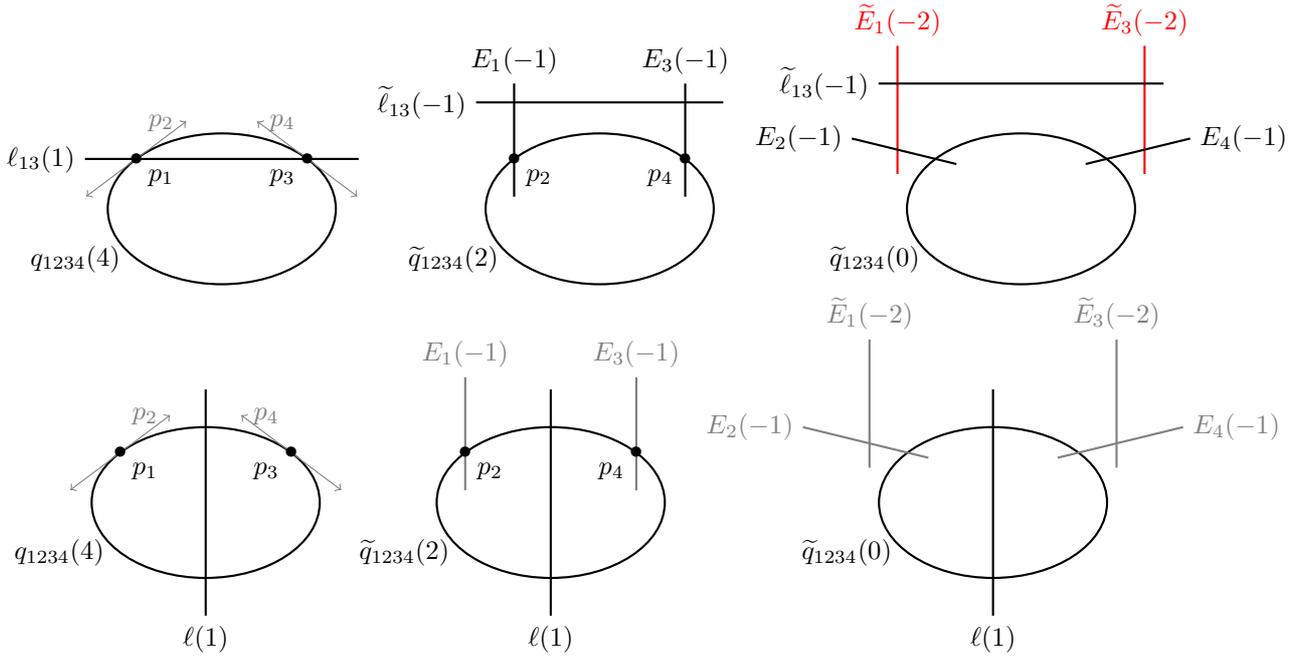
\begin{center}
\begin{tikzpicture}[scale=0.7]
\draw[thick] (-3,2) -- (3,2);
\draw[thick] (-3,2) node[left] {$\ell(1)$};
\draw[thick] (-1.5,0) node[below,xshift=1em] {$p_1$} node {$\bullet$};
\draw[thick] (1.5,0) node[below,xshift=-1em] {$p_3$} node {$\bullet$};
\draw[thick] (-1.5,0) -- ++(-.5,-.5);
\draw[thick] (-1.5,0) -- ++(2.5,2.5) node[above right] {$\ell_{12}(1)$};
\draw[thick,>=stealth,->] (-1.5,0) -- ++(1,1) node[midway,above,xshift=-0.5em] {$p_2$};
\draw[thick] (1.5,0) -- ++(.5,-.5);
\draw[thick] (1.5,0) -- ++(-2.5,2.5) node[above left] {$\ell_{34}(1)$};
\draw[thick][>=stealth,->] (1.5,0) -- ++(-1,1) node[midway,above,xshift=0.5em] {$p_4$};
\end{tikzpicture}
\begin{tikzpicture}[scale=0.7]
\draw[thick] (-2.5,4) -- (2.5,4);
\draw[thick] (-2.5,4) node[left] {$\widetilde{\ell}(1)$};
\draw[thick] (-1.5,2) node[below right] {$p_2$} node {$\bullet$};
\draw[thick] (1.5,2) node[below left] {$p_4$} node {$\bullet$};
\draw[thick] (-1.5,2) -- ++(-0.5,-0.5);
\draw[thick] (-1.5,2) -- ++(2.5,2.5) node[above right] {$\widetilde{\ell}_{12}(0)$};
\draw[thick] (1.5,2) -- ++(0.5,-0.5);
\draw[thick] (1.5,2) -- ++(-2.5,2.5) node[above left] {$\widetilde{\ell}_{34}(0)$};
\draw[thick] (-1.5,.5) -- (-1.5,2.5) node[above] {$E_1(-1)$};
\draw[thick] (1.5,.5) -- (1.5,2.5) node[above] {$E_3(-1)$};
\end{tikzpicture}
\begin{tikzpicture}[scale=0.7]
\draw[thick] (-2.5,4.3) -- (2.5,4.3);
\draw[thick] (-2.5,4.3) node[left] {$\widetilde{\ell}(1)$};
\draw[thick] (-1.5,2) -- ++(-0.5,-0.5) node[below left] {$E_2(-1)$};
\draw[thick] (-1.5,2) -- ++(1.3,1.3);
\draw[thick] (1.5,2) -- ++(0.5,-0.5) node[below right] {$E_4(-1)$};
\draw[thick] (1.5,2) -- ++(-1.3,1.3);
\draw[thick,red] (-1.5,.5) -- (-1.5,2.5) node[above left] {$\widetilde{E}_1(-2)$};
\draw[thick,red] (1.5,.5) -- (1.5,2.5) node[above right] {$\widetilde{E}_3(-2)$};
\draw[thick] (-.7,2.6) to[bend left] ++(1.5,2) node[above right] {$\widetilde{\ell}_{12}(-1)$};
\draw[thick] (.7,2.6) to[bend right] ++(-1.5,2) node[above left] {$\widetilde{\ell}_{34}(-1)$};
\end{tikzpicture}

\begin{tikzpicture}[scale=0.7]
\draw[thick] (-2,0) -- (2,0);
\draw[thick] (-2,0) node[left] {$\ell(1)$};
\draw[thick] (-1.5,0) node[below,xshift=1em] {$p_1$} node {$\bullet$};
\draw[thick] (1.5,0) node[below,xshift=-1em] {$p_3$} node {$\bullet$};
\draw[thick] (-1.5,0) -- ++(-0.5,-0.5);
\draw[thick] (-1.5,0) -- ++(2,2) node[above right] {$\ell_{12}(1)$};
\draw[thick,>=stealth,->] (-1.5,0) -- ++(1,1) node[midway,above,xshift=-0.5em] {$p_2$};
\draw[thick] (1.5,0) -- ++(0.5,-0.5);
\draw[thick] (1.5,0) -- ++(-2,2) node[above left] {$\ell_{34}(1)$};
\draw[thick][>=stealth,->] (1.5,0) -- ++(-1,1) node[midway,above,xshift=0.5em] {$p_4$};
\end{tikzpicture}
\begin{tikzpicture}[scale=0.7]
\draw[thick] (-2,0) -- (2,0);
\draw[thick] (-2,0) node[left] {$\widetilde{\ell}(-1)$};
\draw[thick] (-1.5,2) -- ++(-0.5,-0.5);
\draw[thick] (-1.5,2) -- ++(2,2) node[above right] {$\widetilde{\ell}_{12}(0)$};
\draw[thick] (1.5,2) -- ++(0.5,-0.5);
\draw[thick] (1.5,2) -- ++(-2,2) node[above left] {$\widetilde{\delta}_{34}(0)$};
\draw[thick] (-1.5,2) node[below right] {$p_2$} node {$\bullet$};
\draw[thick] (1.5,2) node[below left] {$p_4$} node {$\bullet$};
\draw[thick] (-1.5,-0.5) -- (-1.5,2.5) node[above] {$E_1(-1)$};
\draw[thick] (1.5,-0.5) -- (1.5,2.5) node[above] {$E_3(-1)$};
\end{tikzpicture}
\begin{tikzpicture}[scale=0.7]
\draw[thick] (-2,0) -- (2,0);
\draw[thick] (-2,0) node[left] {$\widetilde{\ell}(-1)$};
\draw[thick] (-1.5,2) -- ++(-0.5,-0.5) node[below left] {$E_2(-1)$};
\draw[thick] (-1.5,2) -- ++(1.3,1.3);
\draw[thick] (1.5,2) -- ++(0.5,-0.5) node[below right] {$E_4(-1)$};
\draw[thick] (1.5,2) -- ++(-1.3,1.3);
\draw[thick,red] (-1.5,-0.5) -- (-1.5,2.5) node[above left] {$\widetilde{E}_1(-2)$};
\draw[thick,red] (1.5,-0.5) -- (1.5,2.5) node[above right] {$\widetilde{E}_3(-2)$};
\draw[thick] (-1.2,2.1) to[bend left] ++(1.5,2) node[above right] {$\widetilde{\ell}_{12}(-1)$};
\draw[thick] (1.2,2.1) to[bend right] ++(-1.5,2) node[above left] {$\widetilde{\ell}_{34}(-1)$};
\end{tikzpicture}

\begin{tikzpicture}[scale=0.7]
\draw[thick] (-2,0) -- (2,0);
\draw[thick] (-2,0) node[left] {$\ell_{13}(1)$};
\draw[thick] (-1.5,0) node[below,xshift=1em] {$p_1$} node {$\bullet$};
\draw[thick] (1.5,0) node[below,xshift=-1em] {$p_3$} node {$\bullet$};
\draw[thick] (-1.5,0) -- ++(-0.5,-0.5);
\draw[thick] (-1.5,0) -- ++(2,2) node[above right] {$\ell_1(1)$};
\draw[thick,>=stealth,->] (-1.5,0) -- ++(1,0.6) node[midway,above,xshift=-0.5em] {$p_2$};
\draw[thick] (1.5,0) -- ++(0.5,-0.5);
\draw[thick] (1.5,0) -- ++(-2,2) node[above left] {$\ell_1^\sigma(1)$};
\draw[thick][>=stealth,->] (1.5,0) -- ++(-1,0.6) node[midway,above,xshift=0.5em] {$p_4$};
\end{tikzpicture}
\begin{tikzpicture}[scale=0.7]
\draw[thick] (-2,0) -- (2,0);
\draw[thick] (-2,0) node[left] {$\widetilde{\ell}_{13}(-1)$};
\draw[thick] (-1.5,2) -- ++(-0.5,-0.5);
\draw[thick] (-1.5,2) -- ++(2,2) node[above right] {$\widetilde{\ell}_1(0)$};
\draw[thick] (1.5,2) -- ++(0.5,-0.5);
\draw[thick] (1.5,2) -- ++(-2,2) node[above left] {$\widetilde{\ell}_1^\sigma(0)$};
\draw[thick] (-1.5,1.2) node[below right] {$p_2$} node {$\bullet$};
\draw[thick] (1.5,1.2) node[below left] {$p_4$} node {$\bullet$};
\draw[thick] (-1.5,-0.5) -- (-1.5,2.5) node[above] {$E_1(-1)$};
\draw[thick] (1.5,-0.5) -- (1.5,2.5) node[above] {$E_3(-1)$};
\end{tikzpicture}
\begin{tikzpicture}[scale=0.7]
\draw[thick] (-2,0) -- (2,0);
\draw[thick] (-2,0) node[left] {$\widetilde{\ell}_{13}(-1)$};
\draw[thick] (-1.5,2) -- ++(-0.5,-0.5);
\draw[thick] (-1.5,2) -- ++(2,2) node[above right] {$\widetilde{\ell}_1(0)$};
\draw[thick] (1.5,2) -- ++(0.5,-0.5);
\draw[thick] (1.5,2) -- ++(-2,2) node[above left] {$\widetilde{\ell}_1^\sigma(0)$};
\draw[thick,gray] (-1,1.2) -- (-2.5,1.2) node[left] {$E_2(-1)$};
\draw[thick,gray] (1,1.2) -- (2.5,1.2) node[right] {$E_4(-1)$};
\draw[thick,red] (-1.5,-0.5) -- (-1.5,2.5) node[above] {$E_1(-2)$};
\draw[thick,red] (1.5,-0.5) -- (1.5,2.5) node[above] {$E_3(-2)$};
\end{tikzpicture}

\begin{tikzpicture}[scale=0.7]
\draw[thick] (-2,0) -- (2.5,0);
\draw[thick] (-2,0) node[left] {$2\ell_{13}(1)$};
\draw[thick] (-1.5,0) node[below,xshift=1em] {$p_1$} node {$\bullet$};
\draw[thick] (1.5,0) node[below,xshift=-1em] {$p_3$} node {$\bullet$};
\draw[thick] (2,-0.5) -- ++(0,2) node[above] {$\ell(1)$};
\draw[thick,>=stealth,->] (-1.5,0) -- ++(1,0.6) node[midway,above,xshift=-0.5em] {$p_2$};
\draw[thick][>=stealth,->] (1.5,0) -- ++(-1,0.6) node[midway,above,xshift=0.5em] {$p_4$};
\end{tikzpicture}
\begin{tikzpicture}[scale=0.7]
\draw[thick] (-2,0) -- (2.5,0);
\draw[thick] (-2,0) node[left] {$2\widetilde{\ell}_{13}(-1)$};
\draw[thick] (-1.5,-0.5) node[below] {$E_1(-1)$} -- ++(0,2) ;
\draw[thick] (1.5,-0.5) node[below] {$E_3(-1)$} -- ++(0,2) ;
\draw[thick] (-1.5,1.2) node[right] {$p_2$} node {$\bullet$};
\draw[thick] (1.5,1.2) node[left] {$p_4$} node {$\bullet$};
\draw[thick] (2,-0.5) -- ++(0,2) node[above] {$\widetilde{\ell}(1)$};
\end{tikzpicture}
\begin{tikzpicture}[scale=0.7]
\draw[thick] (-2,0) -- (2.5,0);
\draw[thick] (-2,0) node[left] {$2\widetilde{\ell}_{13}(-1)$};
\draw[thick,red] (-1.5,-0.5) node[below] {$\widetilde{E}_1(-2)$} -- ++(0,2) ;
\draw[thick,red] (1.5,-0.5) node[below] {$\widetilde{E}_3(-2)$} -- ++(0,2) ;
\draw[thick,gray] (-1.2,0.9) -- (-2.5,2.2) node[left] {$E_2(-1)$};
\draw[thick,gray] (1.8,0.9) -- (0.5,2.2) node[left] {$E_4(-1)$};
\draw[thick] (2,-0.5) -- ++(0,2) node[above] {$\widetilde{\ell}(1)$};
\end{tikzpicture}

\begin{tikzpicture}[scale=0.7]
\draw[thick] (-2,0) -- (2.5,0);
\draw[thick] (-2,0) node[left] {$3\ell_{13}(1)$};
\draw[thick] (-1.5,0) node[below,xshift=1em] {$p_1$} node {$\bullet$};
\draw[thick] (1.5,0) node[below,xshift=-1em] {$p_3$} node {$\bullet$};
\draw[thick,>=stealth,->] (-1.5,0) -- ++(1,0.6) node[midway,above,xshift=-0.5em] {$p_2$};
\draw[thick][>=stealth,->] (1.5,0) -- ++(-1,0.6) node[midway,above,xshift=0.5em] {$p_4$};
\end{tikzpicture}
\begin{tikzpicture}[scale=0.7]
\draw[thick] (-2,0) -- (2.5,0);
\draw[thick] (-2,0) node[left] {$3\widetilde{\ell}_{13}(-1)$};
\draw[thick] (-1.5,-0.5) node[below] {$2E_1(-1)$} -- ++(0,2) ;
\draw[thick] (1.5,-0.5) node[below] {$2E_3(-1)$} -- ++(0,2) ;
\draw[thick] (-1.5,1.2) node[right] {$p_2$} node {$\bullet$};
\draw[thick] (1.5,1.2) node[left] {$p_4$} node {$\bullet$};
\end{tikzpicture}
\begin{tikzpicture}[scale=0.7]
\draw[thick] (-2,0) -- (2.5,0);
\draw[thick] (-2,0) node[left] {$3\widetilde{\ell}_{13}(-1)$};
\draw[thick,red] (-1.5,-0.5) node[below] {$2\widetilde{E}_1(-2)$} -- ++(0,2) ;
\draw[thick,red] (1.5,-0.5) node[below] {$2\widetilde{E}_3(-2)$} -- ++(0,2) ;
\draw[thick] (-1.2,0.9) -- (-2.5,2.2) node[left] {$E_2(-1)$};
\draw[thick] (1.8,0.9) -- (0.5,2.2) node[left] {$E_4(-1)$};
\end{tikzpicture}

\begin{tikzpicture}
  \def\a{1.5} 
  \def\b{1.0} 
  \def\q{2} 
  \def\x{{\a^2/\q}} 
  \def\y{{\b*sqrt(1-(\a/\q)^2}} 
  \coordinate (O) at (0,0); 
  \coordinate (p) at (-\a,0); 
  \coordinate (psigma) at (\a,0); 
  \draw[thick] (O) ellipse({\a} and {\b});
  \draw[<->,gray] (-2,-0.5) -- (-1,0.5) node[left,near end,xshift=-0.5em,yshift=0.5em] {$p_2$};
  \draw[<->,gray] (2,-0.5) -- (1,0.5) node[right,near end,xshift=0.5em,yshift=0.5em] {$p_4$};
  \draw (-\x,-\y) node[xshift=-0.5cm] {$q_{13}(4)$};
  \draw (p) node {$\bullet$} node[below right,xshift=1em] {$p_1$};
  \draw (psigma) node {$\bullet$} node[below left,xshift=-1em] {$p_3$};
  \draw[thick] (-2,0) node[left] {$\ell_{13}(1)$} -- (2,0);
\end{tikzpicture}
\begin{tikzpicture}
  \def\a{1.5} 
  \def\b{1.0} 
  \def\q{2} 
  \def\x{{\a^2/\q}} 
  \def\y{{\b*sqrt(1-(\a/\q)^2}} 
  \coordinate (O) at (0,0); 
  \coordinate (q) at (-\q,0); 
  \coordinate (p) at (-\x,\y); 
  \coordinate (qsigma) at (\q,0); 
  \coordinate (psigma) at (\x,\y); 
  \draw[thick] (O) ellipse({\a} and {\b});
  \draw (-\x,-\y) node[xshift=-0.5cm] {$\widetilde{q}_{13}(2)$};
  \draw[thick] ($(p)-(0,0.5)$) -- ($(p)+(0,1.3)$) node[above] {$E_1(-1)$};
  \draw[thick] ($(psigma)-(0,0.5)$) -- ($(psigma)+(0,1.3)$) node[above] {$E_3(-1)$};
  \draw ($(-\x,\y)+(0,0.60)$) node {$\bullet$} node[below=7,right=1] {$p_2$};
  \draw ($(\x,\y)+(0,0.60)$) node {$\bullet$} node[below=7,left=1] {$p_4$};
  \draw[thick] ($(p)+(-0.5,1)$) node[left] {$\widetilde{\ell}_{13}(-1)$} -- ($(psigma)+(0.5,1)$);
\end{tikzpicture}
\begin{tikzpicture}
  \def\a{1.5} 
  \def\b{1.0} 
  \def\q{2} 
  \def\x{{\a^2/\q}} 
  \def\y{{\b*sqrt(1-(\a/\q)^2}} 
  \coordinate (O) at (0,0); 
  \coordinate (q) at (-\q,0); 
  \coordinate (p) at (-\x,\y); 
  \coordinate (qsigma) at (\q,0); 
  \coordinate (psigma) at (\x,\y); 
  \draw[thick] (O) ellipse({\a} and {\b});
  \draw (-\x,-\y) node[xshift=-0.5cm] {$\widetilde{q}_{13}(2)$};
  \draw[thick,red] ($(p)-(0,0.5)$) -- ($(p)+(0,1.3)$) node[above] {$\widetilde{E}_1(-2)$};
  \draw[thick,red] ($(psigma)-(0,0.5)$) -- ($(psigma)+(0,1.3)$) node[above] {$\widetilde{E}_3(-2)$};
  \draw[thick,gray] ($(-\x,\y)+(-1,0.60)$) node[left] {$E_2(-1)$} -- ($(-\x,\y)+(0.3,0.60)$) ;
  \draw[thick,gray] ($(\x,\y)+(-0.3,0.60)$) -- ($(\x,\y)+(1,0.60)$) node[right] {$E_4(-1)$};
  \draw[thick] ($(p)+(-0.5,1)$) node[left] {$\widetilde{\ell}_{13}(-1)$} -- ($(psigma)+(0.5,1)$);
\end{tikzpicture}

\begin{tikzpicture}
  \def\a{1.5} 
  \def\b{1.0} 
  \def\q{2} 
  \def\x{{\a^2/\q}} 
  \def\y{{\b*sqrt(1-(\a/\q)^2}} 
  \coordinate (O) at (0,0); 
  \coordinate (q) at (-\q,0); 
  \coordinate (p) at (-\x,\y); 
  \coordinate (qsigma) at (\q,0); 
  \coordinate (psigma) at (\x,\y); 
  \draw[thick] (O) ellipse({\a} and {\b});
  \draw (-\x,-\y) node[xshift=-0.8cm] {$q_{1234}(4)$};
  \draw[<->,gray] ($(q)!0.25!(p)$) -- ($(q)!1.75!(p)$) node[above,near end] {$p_2$};
  \draw[<->,gray] ($(qsigma)!0.25!(psigma)$) -- ($(qsigma)!1.75!(psigma)$)  node[above,near end] {$p_4$};
  \draw (p) node {$\bullet$} node[below=7,right=1] {$p_1$};
  \draw (psigma) node {$\bullet$} node[below=7,left=1] {$p_3$};
  \draw[thick] ($(p)!-0.3!(psigma)$) node[left] {$\ell_{13}(1)$} -- ($(psigma)!-0.3!(p)$);
\end{tikzpicture}
\begin{tikzpicture}
  \def\a{1.5} 
  \def\b{1.0} 
  \def\q{2} 
  \def\x{{\a^2/\q}} 
  \def\y{{\b*sqrt(1-(\a/\q)^2}} 
  \coordinate (O) at (0,0); 
  \coordinate (q) at (-\q,0); 
  \coordinate (p) at (-\x,\y); 
  \coordinate (qsigma) at (\q,0); 
  \coordinate (psigma) at (\x,\y); 
  \draw[thick] (O) ellipse({\a} and {\b});
  \draw (-\x,-\y) node[xshift=-0.8cm] {$\widetilde{q}_{1234}(2)$};
  \draw[thick] ($(p)-(0,0.5)$) -- ($(p)+(0,1)$) node[above] {$E_1(-1)$};
  \draw[thick] ($(psigma)-(0,0.5)$) -- ($(psigma)+(0,1)$) node[above] {$E_3(-1)$};
  \draw (p) node {$\bullet$} node[below=7,right=1] {$p_2$};
  \draw (psigma) node {$\bullet$} node[below=7,left=1] {$p_4$};
  \draw[thick] ($(p)+(-0.5,0.75)$) node[left] {$\widetilde{\ell}_{13}(-1)$} -- ($(psigma)+(0.5,0.75)$);
\end{tikzpicture}
\begin{tikzpicture}
  \def\a{1.5} 
  \def\b{1.0} 
  \def\q{2} 
  \def\x{{\a^2/\q}} 
  \def\y{{\b*sqrt(1-(\a/\q)^2}} 
  \coordinate (O) at (0,0); 
  \coordinate (q) at (-\q,0); 
  \coordinate (p) at (-\x,\y); 
  \coordinate (qsigma) at (\q,0); 
  \coordinate (psigma) at (\x,\y); 
  \draw[thick] (O) ellipse({\a} and {\b});
  \draw (-\x,-\y) node[xshift=-0.8cm] {$\widetilde{q}_{1234}(0)$};
  \draw[thick,red] ($(p)+(-0.5,-0.2)$) -- ($(p)+(-0.5,1.5)$) node[above] {$\widetilde{E}_1(-2)$};
  \draw[thick,red] ($(psigma)+(0.5,-0.2)$) -- ($(psigma)+(0.5,1.5)$) node[above] {$\widetilde{E}_3(-2)$};
  \draw[thick] ($(p)+(-0.75,1)$) node[left] {$\widetilde{\ell}_{13}(-1)$} -- ($(psigma)+(0.75,1)$);
  \draw[thick] ($(p)!-0.2!(-2.5,1)$) -- ($(p)!0.8!(-2.5,1)$) node[left] {${E_2(-1)}$};
  \draw[thick] ($(psigma)!-0.2!(2.5,1)$) -- ($(psigma)!0.8!(2.5,1)$) node[right] {${E_4(-1)}$};
\end{tikzpicture}

\begin{tikzpicture}
  \def\a{1.5} 
  \def\b{1.0} 
  \def\q{2} 
  \def\x{{\a^2/\q}} 
  \def\y{{\b*sqrt(1-(\a/\q)^2}} 
  \coordinate (O) at (0,0); 
  \coordinate (q) at (-\q,0); 
  \coordinate (p) at (-\x,\y); 
  \coordinate (qsigma) at (\q,0); 
  \coordinate (psigma) at (\x,\y); 
  \draw[thick] (O) ellipse({\a} and {\b});
  \draw (-\x,-\y) node[xshift=-0.8cm] {$q_{1234}(4)$};
  \draw[<->,gray] ($(q)!0.25!(p)$) -- ($(q)!1.75!(p)$) node[above,near end] {$p_2$};
  \draw[<->,gray] ($(qsigma)!0.25!(psigma)$) -- ($(qsigma)!1.75!(psigma)$)  node[above,near end] {$p_4$};
  \draw (p) node {$\bullet$} node[below=7,right=1] {$p_1$};
  \draw (psigma) node {$\bullet$} node[below=7,left=1] {$p_3$};
  \draw[thick] (0,-1.5) node[below] {$\ell(1)$} -- (0,1.5);
\end{tikzpicture}
\begin{tikzpicture}
  \def\a{1.5} 
  \def\b{1.0} 
  \def\q{2} 
  \def\x{{\a^2/\q}} 
  \def\y{{\b*sqrt(1-(\a/\q)^2}} 
  \coordinate (O) at (0,0); 
  \coordinate (q) at (-\q,0); 
  \coordinate (p) at (-\x,\y); 
  \coordinate (qsigma) at (\q,0); 
  \coordinate (psigma) at (\x,\y); 
  \draw[thick] (O) ellipse({\a} and {\b});
  \draw (-\x,-\y) node[xshift=-0.8cm] {$\widetilde{q}_{1234}(2)$};
  \draw[thick,gray] ($(p)-(0,0.5)$) -- ($(p)+(0,1)$) node[above] {$E_1(-1)$};
  \draw[thick,gray] ($(psigma)-(0,0.5)$) -- ($(psigma)+(0,1)$) node[above] {$E_3(-1)$};
  \draw (p) node {$\bullet$} node[below=7,right=1] {$p_2$};
  \draw (psigma) node {$\bullet$} node[below=7,left=1] {$p_4$};
  \draw[thick] (0,-1.5) node[below] {$\ell(1)$} -- (0,1.5);
\end{tikzpicture}
\begin{tikzpicture}
  \def\a{1.5} 
  \def\b{1.0} 
  \def\q{2} 
  \def\x{{\a^2/\q}} 
  \def\y{{\b*sqrt(1-(\a/\q)^2}} 
  \coordinate (O) at (0,0); 
  \coordinate (q) at (-\q,0); 
  \coordinate (p) at (-\x,\y); 
  \coordinate (qsigma) at (\q,0); 
  \coordinate (psigma) at (\x,\y); 
  \draw[thick] (O) ellipse({\a} and {\b});
  \draw (-\x,-\y) node[xshift=-0.8cm] {$\widetilde{q}_{1234}(0)$};
  \draw[thick,gray] ($(p)+(-0.5,-0.2)$) -- ($(p)+(-0.5,1.5)$) node[above] {$\widetilde{E}_1(-2)$};
  \draw[thick,gray] ($(psigma)+(0.5,-0.2)$) -- ($(psigma)+(0.5,1.5)$) node[above] {$\widetilde{E}_3(-2)$};
  \draw[thick,gray] ($(p)!-0.2!(-2.5,1)$) -- ($(p)!1!(-2.5,1)$) node[left] {${E_2(-1)}$};
  \draw[thick,gray] ($(psigma)!-0.2!(2.5,1)$) -- ($(psigma)!1!(2.5,1)$) node[right] {${E_4(-1)}$};
  \draw[thick] (0,-1.5) node[below] {$\ell(1)$} -- (0,1.5);
\end{tikzpicture}
\end{center}
Some comments about the number of points.

{\bf Cases~$1$, $2$, \& $3$.} The unions of lines~$\ell_{12}\cup\ell_{34}$, or~$\ell_1\cup\ell_3$ (recall that~$\ell_3=\ell_1^\sigma$),
contain a unique rational point, the intersection point of the two lines. Except~$\ell$ (cases~$1$ and~$2$) or~$\ell_{13}$ (case~$3$),
all the lines in the decomposition are not defined over~$\F_q$ and di not contain any rational points. Thus to the previous single point
we have to add the $(q+1)$ rational points of the line~$\ell$ or~$\ell_{13}$.

{\bf Case~$4$.} The two (black) components have $(q+1)$ rational points but they meet at a rational point, thus their union
contains~$2q+1$ rational points.

{\bf Case~$5$.} The only component that contains rational points is~$\plongement_*(\widetilde{\ell}_{13})$.

{\bf Cases~$6$, $7$, \& $8$.} In these cases, they are two disjoint components that contain $(q+1)$ rational points.

Finally:
$$
N_q\left(-K_{X_s}\right) = 2q+2.
$$

\medbreak

Since none of the points~$p_i$ is rational, the surfaces~$X$ and~$X_s$ still have~$q^2+q+1$ points. For every~$q$,
one has~$\#X_s(\F_q) > N_q\left(-K_{X_s}\right)$ and the evaluation map is always injective. The parameters of the code are thus
given by:

\begin{prop}
Let~$p_1 \prec p_2$ and~$p_3\prec p_4$ be such that~$p_1,p_3$ and~$p_2,p_4$ are conjugate points of degree~$2$.
The anticanonical code of the weak del Pezzo surface obtained by blowing up these points has parameters~$[q^2+q+1,6,q^2-q-1]$.
\end{prop}

\paragraph{Computation of the global sections from~$\P^2$.~---} Let~$d\in\F_q$ be a non-square and put~$\zeta = \sqrt{d}\in\F_{q^2}$. We
choose~$p_1 = (\zeta:0:1)$ and~$p_2=(\zeta:1)$. Then~$p_3 = (-\zeta:0:1)$, the line~$\ell_{13} = (p_1p_3)$ has equation~$Y=0$,
the line~$\ell_{12} = (p_1p_2)$ corresponds to the zeros of the linear form~$L=X-\zeta(Y+Z)$ and the
line~$\ell_{34} = (p_3 p_4)$ corresponds to the zeros of the linear form~$L^\sigma=X+\zeta(Y+Z)=0$.
The linear forms~$Y,L,L^\sigma$ generate the global sections of~$\ell$ and we easily prove that
\begin{align*}
\left|3\ell - p_1 - p_3 - p_2 - p_4\right|
&=
\left\langle Y^3,Y^2L,Y^2L^\sigma,YLL^\sigma,L^2L^\sigma,L(L^\sigma)^2\right\rangle_{\overline{\F}_q}\\
&=
\left\langle Y^3,Y^2(L+L^\sigma),Y^2\zeta(L-L^\sigma),YLL^\sigma,LL^\sigma(L+L^\sigma),LL^\sigma\zeta(L-L^\sigma)\right\rangle_{\F_q}\\
&=
\left\langle Y^3,Y^2X,Y^2(Y+Z),Y\Pi,X\Pi,(Y+Z)\Pi\right\rangle_{\F_q},
\end{align*}
where~$\Pi = L L^\sigma = X^2-d(Y+Z)^2$.
The evaluation points are nothing else than all the points of~$\P^2(\F_q)$.

\subsection{Degree~$4$, singularity of type~$\mathbf{A}_1$}\label{sDeg4A1}

This example corresponds to the type number~$8$ in degree~$4$ \cite{Classification}.

\paragraph{Configuration to blow-up and down.~---}
We blow up six points on a conic, the first two~$p_1$ and~$p_2$ being conjugate (or rational), the last four~$p_3,p_4,p_5,p_6$
being conjugate. This leads to a degree~$3$ weak Del Pezzo surface~$Y\overset{\pi}{\longrightarrow}\P^2$ as in~\eqref{eqP2YXXs}.
In this surface, the strict transform of the line~$\ell_{12}$
passing through~$p_1$ and~$p_2$ is a rational $(-1)$-curve that can be contracted: the codomain of the
contraction~$Y \overset{\chi}{\longrightarrow} X$ is the degree~$4$ weak Del Pezzo surface we want to work with in this section.
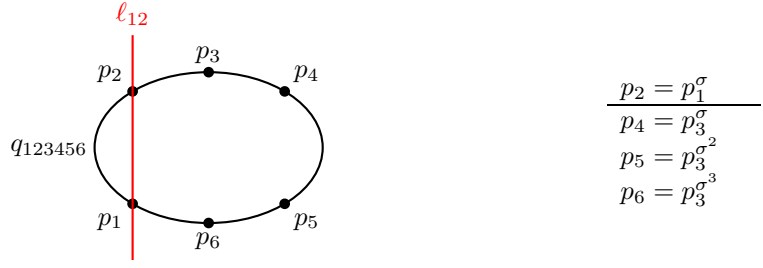
\begin{align*}
&\begin{tikzpicture}[baseline=0]
  \def\a{1.5} 
  \def\b{1.0} 
  \def\x{-1}\def\xx{0}\def\xxx{1}
  \coordinate (p1) at (\x,{-\b*sqrt(1-(\x/\a)^2)});
  \coordinate (p2) at (\x,{\b*sqrt(1-(\x/\a)^2)});
  \coordinate (p3) at (\xx,{\b*sqrt(1-(\xx/\a)^2)});
  \coordinate (p6) at (\xx,{-\b*sqrt(1-(\xx/\a)^2)});
  \coordinate (p4) at (\xxx,{\b*sqrt(1-(\xxx/\a)^2)});
  \coordinate (p5) at (\xxx,{-\b*sqrt(1-(\xxx/\a)^2)});
  \draw[thick] (0,0) ellipse({\a} and {\b});
  \fill (p1) circle (2pt) node[below left] {$p_1$};
  \fill (p2) circle (2pt) node[above left] {$p_2$};
  \fill (p3) circle (2pt) node[above] {$p_3$};
  \fill (p4) circle (2pt) node[above right] {$p_4$};
  \fill (p5) circle (2pt) node[below right] {$p_5$};
  \fill (p6) circle (2pt) node[below] {$p_6$};
  \draw[thick,red] ($(p1)!-0.5!(p2)$) -- ($(p2)!-0.5!(p1)$) node[above] {$\ell_{12}$};
  \draw (-2,0) node[xshift=-0.1cm] {$q_{123456}$};
\end{tikzpicture}
&
&
\begin{array}{ll}
p_2=p_1^\sigma&\\
\hline
p_4=p_3^\sigma\\
p_5=p_3^{\sigma^2}\\
p_6=p_3^{\sigma^3}
\end{array}
\end{align*}
The anticanonical model of this weak del Pezzo surface~$X$ has a unique singular point.

\paragraph{Computation of the divisor class groups.~---}
On~$\overline{Y}$, one has~$\cldivweil(\overline{Y}) = \bigoplus_{i=0}^6 \Z E_i$ and 27 exceptional classes of divisor, namely the 6 exceptional lines~$E_i$, $1\leq i\leq 6$,
the~$15$ strict transforms of the lines passing through two of the six points, $E_{ij} = E_0-E_i-E_j$, $1\leq i<j\leq 6$, and the
$6$ strict transforms of the quadrics passing through five of the six points, $Q_{i} = 2E_0- \sum_{j\not= i} E_j$, $1\leq 6$.
Due to weaknesses, among these classes, the quadric ones are not represented by irreducible curves.
Indeed~$2E_0- \sum_{j\not= i} E_j = E_i + \left(2E_0- \sum_{j= 1}^6 E_j\right)$ and the last class is nothing else
than the class of the unique effective root, the strict transform of the quadric~$q_{123456}$ passing through all the six points.

The group~$\cldivweil(\overline{X})$ can be identified with~$\Z(E_0-E_1-E_2)^\perp$ via the orthogonal projection onto this space. This
projection is given by:
$$
\begin{array}{ccrcl}
  \cldivweil(\overline{Y}) & = & \Z(E_0-E_1-E_2) &\oplus &\left(\Z (E_0-E_1) \oplus \Z (E_0-E_2) \oplus \Z E_3 \oplus \cdots \oplus \Z E_6\right)\\
  \sum_{i=0}^6 a_i E_i & = & (-a_0-a_1-a_2)(E_0-E_1-E_2) &+& \left[(a_0+a_2)(E_0-E_1) + (a_0+a_1)(E_0-E_2) +\sum_{i=3}^6 a_i E_i\right]
\end{array},
$$
and thus
\begin{align*}
\cldivweil(\overline{X})
=
\Z L_1 \oplus \Z L_2 \oplus \Z E_3 \oplus \cdots \oplus \Z E_6,
&
&\text{where}
&
&L_i = E_0-E_i,\quad i=1,2.
\end{align*}
In particular, the anticanonical divisors are related by
$$
-K_Y
=
3E_0 - \sum_{i=1}^6 E_i = -(E_0-E_1-E_2) + \underbrace{2L_1+2L_2 - \sum_{i=3}^6 E_i}_{-K_X}
$$
Only the exceptional classes of~$\overline{Y}$ that do not meet~$E_0-E_1-E_2$ are mapped to exceptional classes on~$\overline{X}$;
for~$3\leq i\leq 6$, this leaves
the classes:
\begin{align*}
&E_i \longmapsto E_i,
&
&E_{1i} \longmapsto L_1-E_j,
&
&E_{2i} \longmapsto L_2-E_j,
&
&Q_i \longmapsto L_1+L_2-\sum_{j\in\{3,\ldots,6\}\setminus\{i\}}E_j.
\end{align*}
As for the unique effective root of~$\overline{Y}$, it is mapped to the root~$L_1+L_2-E_3-E_4-E_5-E_6$ and the last four exceptional
classes are not represented by irreducible curves. Thus one has
\begin{align*}
&\overline{\RR} = \Z (L_1+L_2-E_3-\cdots-E_6),
&
&\overline{\RR}^\perp = \left\{a_1L_1+a_2L_2+\sum_{i=3}^6 E_i \mid a_1+a_2+\sum_{i=3}^6 a_i = 0\right\}.
\end{align*}
In order to take into account the Galois action, which acts via~$(L_1L_2)(E_3E_4E_5E_6)$,
we put~$\L = L_1 + L_2$ and~$\EE = \sum_{i=3}^6 E_i$. We easily verify that
\begin{align*}
&\cldivweil(X) = \cldivcartier(X) = \Z \L \oplus \Z \EE = \Z(\L-\EE)\oplus\Z\EE,
&
&\RR = \Z (\L-\EE),
&
&\RR^\perp = \Z(2\L-\EE)=\Z K_X.
\end{align*}
This leads to the following isomorphism:
$$
\begin{array}{ccc}
\cldivweil(X_s) \simeq \cldivweil(X)/\RR & \overset{\simeq}{\longrightarrow} & \Z\EE\\
a\L+b\EE \bmod{\RR} & \longmapsto & (a+b)\EE
\end{array}
$$
Via this isomorphism, the sub-module~$\cldivcartier(X_s)=\RR^\perp = \Z (2\L-\EE) = \Z K_X$ is mapped to~$\Z\EE$ itself.
In conclusion both~$\cldivcartier(X_s)$ and~$\cldivweil(X_s)$ are free $\Z$-module of rank~$1$ and the canonical embedding turns to be an isomorphism.

\paragraph{Types of decomposition into irreducible components in~$\left|-K_{X_s}\right|$.~---} Recall that
$$
-K_X
=
2L_1+2L_2 - \sum_{i=3}^6 E_i
=
4E_0 -2E_1 -2E_2 - \sum_{i=3}^6 E_i.
$$
Global sections of~$-K_X$ are thus related to quartics of~$\P^2$. More precisely, one has the following one-to-one correspondences:
$$
\begin{array}{ccccc}
\left|4\ell - 2p_1 - 2p_2 - \sum_{i=3}^6 p_i\right|_Y
& \longrightarrow
& \left|4E_0 -2E_1 -2E_2 - \sum_{i=3}^6 E_i\right|_X
& \longrightarrow
& \left|-K_{X_s}\right|_{X_s}\\
C & \longmapsto & \chi_*\left(C^\sharp\right) & \longmapsto &\plongement_*\left(\chi_*\left(C^\sharp\right)\right),
\end{array}
$$
and we need to list all the quadrics of~$\P^2$ having multiplicity at least~$2$ at~$p_1$ and~$p_2$ and passing through the~$p_i$
for~$3\leq i\leq 6$. Note that in the correspondences above, we skip the surface~$Y$. Recall that, as in~\eqref{eqP2YXXs},
we have~$\P^2 \overset{\pi}{\longleftarrow} Y \overset{\chi}{\longrightarrow} X$ and the morphism~$\chi$ here is the contraction
of the strict transform of the line passing through~$p_1$ and~$p_2$.


The orbits of lines, respectively conics, having degree less than~$4$ and passing through some of the points~$p_i$'s are
\begin{align*}
&\ell_1\cup\ell_2,
&
&\ell_3\cup\ell_4\cup\ell_5\cup\ell_6,
&
&\ell_{12},
&
&\ell_{35}\cup\ell_{46},
&
&\ell_{13}\cup\ell_{24}\cup\ell_{15}\cup\ell_{26},
&
&\ell_{14}\cup\ell_{25}\cup\ell_{16}\cup\ell_{23},
\end{align*}
respectively:
\begin{align*}
&q_1\cup q_2,
&
&q_{12},
&
&q_{35}\cup q_{46},
&
&q_{3456},
&
&q_{1235}\cup q_{1246},
&
&q_{123456}.
\end{align*}
The only orbits of cubics or quartics having degree less than~$4$ that pass through the points~$p_i$ are~$c_{123456}$ and~$t_{123456}$. We now combine the rational irreducible decompositions in order to construct plane curves in the expected sub-linear system.

First, suppose that the decomposition into absolute irreducible components contains a line
\begin{itemize}
\item If this line joins one of the first two points
to one of the last four points, i.e. a line~$\ell_{ij}$ with~$i\in\{1,2\}$ and~$j\in\{3,4,5,6\}$, then this line has degree~$4$ and it
turns out that its orbit under the Galois action lies in the linear system;
(cases~$1$ and~$2$ in the tabular below).
\item  If this line is~$\ell_{12}$, which is rational, then this line appears with multiplicity at
most~$2$. If~$2\ell_{12}$ is a part of the decomposition then the complementary conic must be rational and pass through the last four
points: the conic must be~$\ell_{35}\cup\ell_{46}$ or~$q_{3456}$ or~$q_{123456}$ (cases~$3,4,5$). If~$\ell_{12}$ has multiplicity~$1$, then the complementary cubic passes through the six points. Since, except~$\ell_{12}$, all the lines passing through
some~$p_i$ have even degree, this cubic cannot be a union of three lines. The remaining cases are thus~$q_{123456}\cup\ell$
or~$c_{123456}$ (cases~$6$ and~$7$).
\item If the line is~$\ell_i$ for~$i\geq 3$, then it has degree (at least)~$4$ and its orbit under
Galois has degree~$4$ (or greater) without passing through~$p_1$ and~$p_2$. It does not work.
\item If the line is~$\ell_1$
then its conjugate~$\ell_1^\sigma$ passes through~$p_2$; the complement is a conic passing through the six points, and it
must be~$q_{123456}$ (case~$8$).
\item  Last if this line is~$\ell$ a line passing through none of the six points then the complementary cubic
passes through the six points with multiplicity~$2$ at~$p_1$ and~$p_2$. Since an irreducible plane cubic has at most one singular
point, this cubic must be reducible and it is the union of a line and a conic, whose meeting points are the singular points, that
is~$p_1$ and~$p_2$. Thus the line must be~$\ell_{12}$, and the conic is~$q_{123456}$ and we recover case~$6$.
\end{itemize}

Secondly, suppose that there are only two absolutely irreducible conics in the decomposition. For the union of these two
conics to be singular at~$p_1$ and~$p_2$, they must pass through~$p_1$ and~$p_2$. Taking into account the rationality,
there are only three possibilities, cases~$9,10,11$.

Last if the quartic is absolutely irreducible, then it must pass through the six points with multiplicity~$2$ at~$p_1$ and~$p_2$.

In the tabular below, we summarize all the possibilities. As noted below, the strict transform~$\widetilde{\ell}_{12}$ in~$Y$
is contracted in~$X$ via the morphism~$Y\overset{\chi}{\longrightarrow} X$; this explains why the curve disappears in the middle column.
Then from~$X$ to~$X_s$,
it is the irreducible effective root~$\widetilde{q}_{123456}$ that is contracted by the morphism~$X \overset{\varphi}{\longrightarrow} X_s$.
Thus on~$X_s$, there
are two specific rational points,~$p$ the image of the contraction of~$\widetilde{\ell}_{12}$ and~$s$
the image of the contraction of~$\widetilde{q}_{123456}$.
$$
{\renewcommand{\arraystretch}{1.5}
\begin{array}{|r|l|l|l||l|l||l|l|}
\hline
&\left|4\ell -2p_1-2p_2-\sum_{i=3}^6 p_i\right|  & \left|-K_X\right|& \left|-K_{X_s}\right| & \text{Max}   \\
&\text{on~$\P^2$}                 & \text{on~$X$}             & \text{on~$X_s$}                & \text{nb. of pts}\\
\hline\hline
1&
\ell_{13}\cup\ell_{24}\cup\ell_{15}\cup\ell_{26}
&\widetilde{\ell}_{13}\cup\widetilde{\ell}_{24}\cup\widetilde{\ell}_{15}\cup\widetilde{\ell}_{26}
&\plongement_*(\widetilde{\ell}_{13})\cup\plongement_*(\widetilde{\ell}_{24})\cup\plongement_*(\widetilde{\ell}_{15})\cup\plongement_*(\widetilde{\ell}_{26})
&0
\\
2&
\ell_{14}\cup\ell_{25}\cup\ell_{16}\cup\ell_{23}
&\widetilde{\ell}_{14}\cup\widetilde{\ell}_{25}\cup\widetilde{\ell}_{16}\cup\widetilde{\ell}_{23}
&\plongement_*(\widetilde{\ell}_{14})\cup\plongement_*(\widetilde{\ell}_{25})\cup\plongement_*(\widetilde{\ell}_{16})\cup\plongement_*(\widetilde{\ell}_{23})
&0
\\
\hline
3&
2\ell_{12}\cup\ell_{35}\cup\ell_{46}
&\widetilde{\ell}_{35}\cup\widetilde{\ell}_{46}
&\plongement_*(\widetilde{\ell}_{35})\cup\plongement_*(\widetilde{\ell}_{46})
&2
\\
4&
2\ell_{12} \cup q_{3456}
&\widetilde{q}_{3456}
&\plongement_*(\widetilde{q}_{3456})
&q+2
\\
5&
2\ell_{12}\cup q_{123456}
&\widetilde{q}_{123456}\cup E_1 \cup E_2
&\{p,s\}\cup \plongement_*(E_1) \cup \plongement_*(E_2)
&2
\\
6&
\ell_{12}\cup q_{123456}\cup \ell
&\widetilde{q}_{123456}\cup\widetilde{\ell}
&\plongement_*(\widetilde{\ell})
&q+2
\\
7&
\ell_{12} \cup c_{123456}
&\widetilde{c}_{123456}
&\plongement_*(\widetilde{c}_{123456})
&\Nqg{q}{1}
\\
\hline
8&
\ell_{1}\cup \ell_{1}^\sigma\cup q_{123456}
&\widetilde{\ell}_{1}\cup\widetilde{\ell}_{1}^\sigma\cup\widetilde{q}_{123456}
&\plongement_*(\widetilde{\ell}_{1})\cup\plongement_*(\widetilde{\ell}_{1}^\sigma)
&2
\\
\hline
9&
q_{12} \cup q_{123456}
&\widetilde{q}_{12}\cup\widetilde{q}_{123456}
&\plongement_*(\widetilde{q}_{12})
&q+2
\\
10&
q_{1235} \cup q_{1246}
&\widetilde{q}_{1235}\cup\widetilde{q}_{1246}
&\plongement_*(\widetilde{q}_{1235})\cup\plongement_*(\widetilde{q}_{1246})
&2
\\
11&
2q_{123456}
&\widetilde{q}_{123456} \cup \bigcup_{i=3}^6 E_i
&\{s\}  \cup \bigcup_{i=3}^6 \plongement_*(E_i)
&1
\\
\hline
12&
t_{123456} \text{ singular at~$p_1,p_2$}
&\widetilde{t}_{123456}
&\plongement_*(\widetilde{t}_{123456})
&\Nqg{q}{1}
\\
\hline
\end{array}}
$$
Some comments about the numbers of points are in order.

{\bf Cases~1 \& 2.} The four lines~$\ell_{13},\ell_{24},\ell_{15},\ell_{26}$ are conjugate, they do not meet and thus their union in~$\P^2$
does not contain any rational point. In the blowing-up of~$\P^2$ at the six points, the strict transforms of the
lines~$\ell_{13},\ell_{24},\ell_{15},\ell_{26}$ no longer meet the strict transform of~$\ell_{12}$. Thus the contraction of
this line does not add any rational point. Since none of the lines~$\ell_{13},\ell_{24},\ell_{15},\ell_{26}$ can be a tangent line
to~$q_{123456}$ at some~$p_i$ (otherwise the line and the quadric would have too many intersection points by Bezout), blowing-up
the~$p_i$, $1\leq i\leq 6$, separates the strict transforms of the lines and the strict transform of the quadric. The contraction of
this quadric neither add rational points. This proves that these configurations do not contain any rational point.

{\bf Case~3.} The lines~$\ell_{35}$ and~$\ell_{46}$ are conjugate to each other. Their intersection point is the unique rational point
of their union in~$\P^2$. This point is still on~$\plongement_*(\widetilde{\ell}_{35})\cup\plongement_*(\widetilde{\ell}_{46})$.
The added point is~$p$ which comes from the contraction of~$\widetilde{\ell}_{12}$. Note that the contraction of~$\widetilde{q}_{123456}$ does
not add any point since the lines~$\ell_{35}$ and~$\ell_{46}$ are separated from~$\widetilde{q}_{123456}$ in the blow-ups. This is because
none of these lines can be a tangent to~$q_{123456}$ at one of the six points (otherwise the line and the quadric would have too
many intersection points by Bezout).

{\bf Case~4.} The conics~$q_{123456}$ and~$q_{3456}$ are separated by the blowing up of the last four points and thus the point~$s$
does not belong to the final section. The strict transforms of~$\ell_{12}$ and of~$q_{3456}$ meet at two points that are mapped to~$p$
by the contraction of~$\widetilde{\ell}_{12}$. If these two points are not rational,~$p$ is an added rational point of the final section.

{\bf Case~5 \& 11.} The resulting sections contain some exceptional curves~$E_i$ in their supports since the multiplicities at some
points~$p_i$ are strictly greater than the ones expected. Since the points~$p_i$ are not rational, none of the~$E_i$ contain
rational points and the only rational points are~$\{p,s\}$ or~$\{s\}$.

{\bf Case~6.} The point~$p$ lies on~$\plongement_*(\widetilde{\ell})$ but it comes from the intersection point between~$\ell$ and~$\ell_{12}$
(which cannot be one of the~$p_i$ since it is a rational point) so no points are added in the contraction of $\ell_{12}$.
As for the point~$s$, it lies also on~$\plongement_*(\widetilde{\ell})$ and it could add one more point
if~$\ell$ meet~$q_{123456}$ at two conjugate points.

{\bf Case~7.} The cubic must be smooth at~$p_1$ and~$p_2$; indeed if it would
be singular at one of these points, by Galois conjugation, it would be singular at both of them and it would have
too many singular points. Thus, to make the multiplicity greater than~$2$ at~$p_1, p_2$, the complementary line must be~$\ell_{12}$.
This line meets the cubic at~$p_1$ and~$p_2$ and a third point which must be rational and not on~$q_{123456}$. Therefore, the contraction
of~$\widetilde{\ell}_{12}$ pass through~$p$ but does not add any rational points to~$\plongement_*(\widetilde{\ell}_{12})$.
As for the contraction of the strict transform of~$q_{123456}$, it does not add points either since blowing up the six points
separates the cubic and~$q_{123456}$.


{\bf Case~8.} The meeting point of the curves~$\ell_1$ and~$\ell_1^\sigma$ is necessarily rational and it is the unique rational point
of their union in~$\P^2$. The strict transforms~$\widetilde{\ell}_1$ and~$\widetilde{\ell}^\sigma_1$ do not meet~$\widetilde{\ell}_{12}$ and
thus the point~$p$ does belong to the final section. The contraction of the root~$\widetilde{q}_{123456}$ add the point~$s$ except
if the meeting of the curves~$\ell_1$ and~$\ell_1^\sigma$ already belongs to~$q_{123456}$.

{\bf Case~9.} The conics~$q_{12}$ and~$q_{123456}$ are separated by the blowing ups. If the conic~$q_{12}$ is chosen in such a way that the
tangent line at~$p_1$ equals~$\ell_{12}$, then~$\widetilde{q}_{12}$ and~$\widetilde{\ell}_{12}$ meet at two unrational points in~$Y$ and the
contraction of~$\widetilde{\ell}_{12}$ adds the rational point~$p$ to~$\widetilde{q}_{12}$ in~$X$. This explains why the final
number of points is~$(q+1)+1$.

{\bf Case~10.} The two conics are conjugate. Besides~$p_1$ and~$p_2$, they meet at two other points (Bezout) that can be
rational. If so these points are the
only points of~$\P^2(\F_q)$ that belong to the union of the two conics. Blowing up the six points disconnect the two conics
from the strict transforms of~$\ell_{12}$ and~$q_{123456}$. So no points are added.

{\bf Case~12.} Since the quartic has at least two singular points, it geometric genus must be at most~$1$.

As predicted by the class group computations, all the curves on~$X_s$ in the linear system are irreducible; they are not necessarily
absolutely irreducible but it turns out that curves that are not absolutely irreducible never contain too many rational points.

In any case, one has
$$
N_q\left(-K_{X_s}\right) \leq N_q(1).
$$

\medbreak

Since none of the points~$p_i$ is rational, the surface~$Y$ has~$q^2+q+1$ rational points. Since~$X$ is obtained by
contracting~$\widetilde{\ell}_{12}$, it contains~$q^2+1$ rational points. In the same way, after contracting~$\widetilde{q}_{123456}$,
the surface~$X_s$ has~$q^2-q+1$ rational points. Since~$q^2-q+1 \leq N_q(1)$ for~$q\in\{2,3\}$, the evaluation map may be non injective
and we do not consider the codes with these two values.

\begin{prop}
Suppose that~$q\not=2,3$.
Let~$p_1,\ldots,p_6\in\P^2_{\F_q}$ be six conconic points, such that~$p_1,p_2$ and~$p_3,p_4,p_5,p_6$ are conjugate.
The anticanonical code of the weak del Pezzo surface obtained
by blowing up these six points and then blowing down the strict transform of the line~$(p_1p_2)$ has
parameters~$[q^2-q+1,5,\geq q^2 -q +1 - \Nqg{q}{1}]$.
\end{prop}

\paragraph{Computation of the global sections from~$\P^2$.~---} Let~$Q$ be a quadratic polynomial that defines~$q_{123456}$ and~$L_{ij}$
a linear form that defines the line~$\ell_{ij}$. Then
$$
\left|
4\ell - 2p_1 - 2p_2 - \sum_{i=3}^6 p_i
\right|
=
\langle QL_{12}X,QL_{12}Y,QL_{12}Z,L_{12}^2L_{35}L_{46},L_{13}L_{24}L_{15}L_{26}\rangle_{\F_q}
$$
The three first sections are clearly linearly independent. The fourth one cannot be a linear combination
of the three first ones since otherwise~$Q$ would be reducible. Last, the fifth one cannot be a linear combination of the four first ones
since otherwise~$L_{12}$ would divide~$L_{13}L_{24}L_{15}L_{26}$.

\subsection{Degree~$4$, singularity of type~$4\mathbf{A}_1$} \label{sDeg44A1}

This example corresponds to the type number~$48$ in degree~$4$ \cite{Classification}.
We recover an example already studied by Koshelev \cite[\S1.2]{Koshelev}. Our point of view slightly
differs from Koshelev's one, so even if this example appears in the literature, we choose to go into details.

\paragraph{Configuration to blow-up and down.~---} The context is still the one described in~\eqref{eqP2YXXs} with a non
trivial map~$Y \overset{\chi}{\longrightarrow} X$.

Let~$p_1,p_2=p_1^{\sigma},p_3=p_1^{\sigma^2},p_4=p_1^{\sigma^3}\in\P^2$ be four conjugate points in general position (no three of them are
collinear) and, as usual, let~$\ell_{12},\ell_{23},\ell_{34},\ell_{14}$ denote the lines~$(p_1p_2),(p_2p_3),(p_3p_4),(p_1p_4)$; they are conjugate.
Let~$p_5$ be the intersection point of~$\ell_{12}$ and~$\ell_{34}$ and~$p_6$ be the intersection point of~$\ell_{23}$ and~$\ell_{14}$.
They are also conjugate and we denote by~$\ell_{56}$ the rational line passing through~$p_5,p_6$.
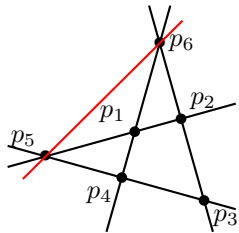
\begin{align*}
&\begin{tikzpicture}[baseline=0]
\def\l{2.5};
\def\h{0.7};
\coordinate (p5) at (0,0);\coordinate (a5) at (\l,\h);\coordinate (b5) at (\l,-\h);
\coordinate (p6) at (1.5,1.5);\coordinate (a6) at (1.5+\h,1.5-\l);\coordinate (b6) at (1.5-\h,1.5-\l);
\draw[thick] ($(p5)!-0.2!(a5)$)--(a5);\draw[thick] ($(p5)!-0.2!(b5)$)--(b5);
\draw[thick] ($(p6)!-0.2!(a6)$)--(a6);\draw[thick] ($(p6)!-0.2!(b6)$)--(b6);
\fill (intersection cs:first line={(p5)--(a5)},second line={(p6)--(b6)}) circle (2pt) node[above left] {$p_1$};
\fill (intersection cs:first line={(p5)--(a5)},second line={(p6)--(a6)}) circle (2pt) node[above right] {$p_2$};
\fill (intersection cs:first line={(p5)--(b5)},second line={(p6)--(a6)}) circle (2pt) node[below right] {$p_3$};
\fill (intersection cs:first line={(p5)--(b5)},second line={(p6)--(b6)}) circle (2pt) node[below left] {$p_4$};
\draw (p5) node {$\bullet$} node[above left] {$p_5$};
\draw (p6) node {$\bullet$} node[right] {$p_6$};
\draw[thick,red] ($(p5)!-0.2!(p6)$) -- ($(p6)!-0.2!(p5)$);
\end{tikzpicture}
&
&\begin{array}{l}
p_2=p_1^\sigma\\
p_3=p_1^{\sigma^2}\\
p_4=p_1^{\sigma^3}\\
p_6=p_5^\sigma
\end{array}
\end{align*}
We blow up these six points to obtain a degree~$3$ weak del Pezzo surface~$Y$.
The strict transform of the line~$\ell_{56}$, of class~$E_0-E_5-E_6$, is an exceptional curve that can be contracted to obtain
the degree four weak del Pezzo surface~$X$ we consider here. The anticanonical model of this surface has four
singular points of type~$\mathbf{A}_1$ (since the four irreducible effective roots do not intersect, see below).

\paragraph{Computation of the divisor class groups.~---} Over~$\overline{\F}_q$, we know
that~$\cldivweil(\overline{Y}) = \bigoplus_{i=0}^6 \Z E_i$ and that~$-K_Y=3E_0 - \sum_{i=1}^6 E_i$. Moreover
the surface~$Y$ has four irreducible effective roots, the strict transforms
of the lines~$\ell_{125},\ell_{236},\ell_{345},\ell_{146}$ whose conjugate classes in~$\cldivweil(\overline{Y})$ are:
\begin{align*}
&E_0-E_1-E_2-E_5,
&
&E_0-E_2-E_3-E_6,
&
&E_0-E_3-E_4-E_5,
&
&E_0-E_1-E_4-E_6.
\end{align*}
The group~$\cldivweil(\overline{X})$ identifies
with~$\Z(E_0-E_5-E_6)^\perp$ inside~$\cldivweil(\overline{Y})$. Since
\begin{align}\label{eqDeg4Type48OrthoDecomp}
\begin{array}{rcrcl}
\cldivweil(\overline{Y}) & = &\Z(E_0-E_5-E_6) & \overset{\perp}{\oplus} & \left[\Z(E_0-E_5)\oplus\Z(E_0-E_6) \oplus \bigoplus_{i=1}^4\Z E_i \right]\\
\sum_{i=0}^6 a_i E_i & = &(-a_0-a_5-a_6)(E_0-E_5-E_6) & + & (a_0+a_6)(E_0-E_5)+(a_0+a_5)(E_0-E_6)+\sum_{i=1}^4 a_i E_i
\end{array}
\end{align}
one has
$$
\cldivweil(\overline{X})
=
\Z(E_0-E_5)\oplus\Z(E_0-E_6) \oplus \Z E_1 \oplus \Z E_2 \oplus \Z E_3 \oplus \Z E_4.
$$
In particular,
$$
-K_X
=
2(E_0-E_5)+2(E_0-E_6)-E_1-E_2-E_3-E_4
=
4E_0-2E_5-2E_6-E_1-E_2-E_3-E_4.
$$
On~$\overline{X}$, there are still four effective
roots, the image by the contraction of the effective roots on~$\overline{Y}$:
\begin{align*}
\overline{\RR}
=
\Z (E_0-E_1-E_2-E_5)
\oplus
\Z (E_0-E_2-E_3-E_6)
\oplus
\Z (E_0-E_3-E_4-E_5)
\oplus
\Z (E_0-E_1-E_4-E_6).
\end{align*}
On~$\overline{X}$ there are $16$ exceptional classes,
\begin{align*}
&(E_0-E_i)-E_j, \; i\in\{5,6\},\, j\in\{1,2,3,4\},
&
&(E_0-E_5)+(E_0-E_6)-E_{i_1}-E_{i_2}-E_{i_3}, \; \{i_1,i_2,i_3\} \subset \{1,2,3,4\},
\end{align*}
and~$E_1,E_2,E_3,E_4$. Only these last four classes are represented by irreducible exceptional curves. We recover the graph number~$9$
of the Proposition~6.1 of Coray \& Tsfasman \cite{CorayTsfasman}.

The Galois group acts as as a $4$-cycle on the roots and as~$\left(E_1E_2E_3E_4\right)$ on the $(-1)$-curves.
Let us put~$\Fcal = (E_0-E_5)+(E_0-E_6)$ and~$\EE = E_1+E_2+E_3+E_4$, in such a way that~$\Fcal^{\cdot 2} = 2$,
$\EE^{\cdot 2} = -4$ and~$\Fcal\cdot\EE = 0$. Then one has:
\begin{align*}
&\begin{cases}\RR = \Z 2(\Fcal-\EE)\\\cldivweil(X) = \Z (\Fcal-\EE) \oplus \Z\EE\end{cases}
&
&\Longrightarrow
&
&\begin{array}{ccrcl}
\cldivweil(X_s) = \cldivweil(X)/\RR &\overset{\simeq}{\longrightarrow} & \Z/2\Z (\Fcal-\EE) &\oplus &\Z\EE\\
a\Fcal+b\EE & \longmapsto & a(\Fcal-\EE) \bmod{\RR} &+& (a+b)\EE
\end{array}
\end{align*}
As for the Cartier class group, we find~$\cldivcartier(X_s) = \RR^\perp = \Z (2\Fcal-\EE) = \Z(-K_X)$ which embeds in~$\cldivweil(X_s)$
via~$-K_X \mapsto \EE$. Thus, via the canonical embedding,~$\cldivcartier(X_s)$ and the free part of~$\cldivweil(X_s)$ are isomorphic.

\paragraph{Types of decomposition into irreducible components in~$\left|-K_{X_s}\right|$.~---} The situation looks like the preceding
one except that the multiplicity is at points~$p_5,p_6$ here. One has:
$$
\begin{array}{ccccc}
\left|4\ell - \sum_{i=1}^4 p_i - 2p_5 - 2p_6 \right|_Y
& \longrightarrow
& \left|4E_0 -2E_5 -2E_6 - \sum_{i=1}^4 E_i\right|_X
& \longrightarrow
& \left|-K_{X_s}\right|_{X_s}\\
C & \longmapsto & \chi_*\left(C^\sharp\right) & \longmapsto &\plongement_*\left(\chi_*\left(C^\sharp\right)\right).
\end{array}
$$
Thus we are reduced to list all the types of irreducible decompositions of quadrics passing through the six points, the last two
with multiplicities at least~$2$.

The orbits of lines of degree less than~$4$ that involve the six points are
\begin{align*}
&\ell_5\cup\ell_6,
&
&\ell_1\cup\ell_2\cup\ell_3\cup\ell_4,
&
&\ell_{56},
&
&\ell_{13}\cup\ell_{24},
&
&\text{and}
&
&\ell_{125}\cup\ell_{236}\cup\ell_{345}\cup\ell_{146}.
\end{align*}
There are only two ways (cases~$1$ and~$2$ below) to combine these configurations in order to obtain a curve in the expected linear system.

The orbits of conics of degree less than~$4$ that involve the six points are~$q_{1234}, q_{56}$ and~$q_{1356}\cup q_{2456}$.
There are only two ways (cases~$3$ and~$4$ below) to combine the configurations of lines and conics in order to obtain a curve in the
expected linear system.

If the decomposition contains a cubic, it must be smooth at $p_5$ and $p_6$ and the complement must be~$\ell_{56}$; this is case~$5$. 
This leads to the list below. Let us note that on~$X$ the curve~$\widetilde{\ell}_{56}$ in~$Y$ is contracted
by~$\chi$. On~$X_s$, this contraction is mapped to a smooth rational point~$p$.
This surface contains also four singular points~$s_i$, $1\leq i\leq 4$, coming from the
contraction of the four roots; they are conjugate and of degree~$4$.
$$
{\renewcommand{\arraystretch}{1.5}
\begin{array}{|r|l|l|l||l|l||l|l|}
\hline
&\left|4\ell - \sum_{i=1}^4 p_i -2p_5-2p_6\right|  &\left|-K_X\right| & \left|-K_{X_s}\right| & \text{Max}   \\
&\text{on~$\P^2$}                 & \text{on~$X$}             & \text{on~$X_s$}  & \text{nb. of pts}\\
\hline\hline
1&
2\ell_{56}\cup\ell_{13}\cup\ell_{24}
&\widetilde{\ell}_{13}\cup\widetilde{\ell}_{24}
&\plongement_*(\widetilde{\ell}_{13})\cup\plongement_*(\widetilde{\ell}_{24})
&2
\\
\hline
2&
\ell_{125}\cup\ell_{236}\cup\ell_{345}\cup\ell_{146}
&\widetilde{\ell}_{125}\cup\widetilde{\ell}_{236}\cup\widetilde{\ell}_{345}\cup\widetilde{\ell}_{146}\cup\bigcup_{i=1}^4 E_i
&\{s_i\}\cup\bigcup_{i=1}^4 \plongement_*(E_i)
&0
\\
\hline
3&
2\ell_{56}\cup q_{1234}
&\widetilde{q}_{1234}
&\plongement_*(\widetilde{q}_{1234})
&q+2
\\
\hline
4&
q_{1356}\cup q_{2456}
&\widetilde{q}_{1356}\cup\widetilde{q}_{2456}
&\plongement_*(\widetilde{q}_{1356})\cup\plongement_*(\widetilde{q}_{2456})
& 2
\\
\hline
5&
c_{123456} \cup \ell_{56}
&\widetilde{c}_{123456}
&\plongement_*(\widetilde{c}_{123456})
&\Nqg{q}{1}
\\
\hline
6&
t_{123456} \text{ singular at~$p_5,p_6$}
&\widetilde{t}_{123456}
&\plongement_*(\widetilde{t}_{123456})
&\Nqg{q}{1}
\\
\hline
\end{array}}
$$
Some comments on the number of rational points.

{\bf Case~1.} The lines~$\ell_{13}$ and~$\ell_{24}$ are conjugate, their meeting point is the unique rational point of their union.
After the contraction of~$\widetilde{\ell}_{56}$, these two lines have one more rational point in common, the point~$p$.

{\bf Case~3.} If~$q_{1234}$ meets~$\ell_{56}$ at two conjugate points, the contraction of~$\widetilde{\ell}_{56}$ adds the point~$p$
to the other rational points.

{\bf Case~4.} Blowing up~$p_5$ and~$p_6$ separates the strict transforms~$\widetilde{q}_{1356}$ and~$\widetilde{q}_{2456}$ from~$\ell_{56}$.
So the contraction of this curve do not add any point. On~$\P^2$ the union of conics~$\widetilde{q}_{1356}$ and~$\widetilde{q}_{2456}$
has at most two rational points: their meeting points that differ from~$p_5,p_6$ if they are rational.

{\bf Case~5.} Necessarily the cubic is smooth at~$p_5,p_6$ and the tangent lines at these points cannot be equal to~$\ell_{56}$. Therefore
blowing up~$p_5,p_6$ separates the curves~$\widetilde{\ell}_{56}$ and~$\widetilde{c}_{123456}$ above~$p_5$ and~$p_6$.
Besides~$p_5,p_6$ the curves~$\ell_{56}$ and~$c_{123456}$ meet at a third point (Bezout) which is necessarily rational.
Via the contraction of~$\widetilde{\ell}_{56}$, this line concentrates at this third point and no points are added on~$\widetilde{c}_{123456}$.

{\bf Case~6.} Necessarily blowing up~$p_5$ and~$p_6$ disconnects the strict transforms~$\widetilde{t}_{123456}$ and~$\widetilde{\ell}_{56}$.
Moreover, it desingularizes the quartic~$t$ since the singularities at~$p_5$ and~$p_6$ must be ordinary. Blowing up~$p_1,\ldots,p_4$
also disconnects~$\widetilde{t}$ from all the effective roots. Finally~$\plongement_*({t}_{123456})$ turns to be an elliptic curve.

Finally~$N_q(-K_{X_s}) \leq N_q(1)$.

\medbreak

Since none of the points~$p_i$ is rational, $\#X(\F_q) = \#\P^2(\F_q) = q^2+q+1$. Then contracting the rational
line~$\widetilde{\ell}_{56}$ decreases the number by~$q$ and~$\#X_s(\F_q) = q^2+1$. Except for~$q=2$, this number
is stricly greater that~$N_q(1)$ and the evaluation map is injective.

\begin{prop}
Suppose~$q\not=2$.
Let~$p_1,p_2=p_1^{\sigma},p_3=p_1^{\sigma^2},p_4=p_1^{\sigma^3}\in\P^2_{\F_q}$ be four conjugate points in general position (no three of them are
collinear) and let~$p_5$ (resp.~$p_6$) be the point of intersection of the lines~$(p_1p_2)$ and~$(p_3p_4)$ (resp.~$(p_2p_3)$ and~$(p_1p_4)$).
The anticanonical code of the weak del Pezzo surface obtained
by blowing up these six points and then blowing down the strict transform of the line~$(p_5p_6)$ has
parameters~$[q^2+1,5,\geq q^2 + 1 -\Nqg{q}{1}]$.
\end{prop}

As proved by Koshelev \cite[\S1.2]{Koshelev}, for some values of~$q$, the minimum distance can be improved by one. Since the argument is very nice, we choose to briefly sketch it below. The idea is to prove that all the elliptic curves
in our linear system must have a rational $2$-torsion point and thus an even number of points. Since for some~$q$, the
maximum~$N_q(1)$ is odd, this means that~$N_q\left(-K_{X_s}\right) < N_q(1)$ and our bound for the minimum distance can be improved by~$1$.
Let~$c$ be
a cubic passing through the six points. Then, for any choice of the origin, the alignments of points permit to show that:
\begin{align*}
&
\begin{array}{rl}
&p_1+p_2+p_5\\=&p_1+p_4+p_6\\=&p_2+p_3+p_6\\=&p_3+p_4+p_5
\end{array}
&
&\Longrightarrow
&
\begin{array}{rl}
&p_1-p_3\\=&p_2-p_4\\=&p_4-p_2\\=&p_6-p_5
\end{array}
&
&\Longrightarrow
&
&2(p_2-p_4) = 0,
\end{align*}
and these points must be rational since~$p_2-p_4 = p_6-p_5$ with~$p_5,p_6$ conjugate. The case of the quartics in the linear system
works in the same way but it is a little bit more technical since we need to know the group law on this kind of curve. 

\paragraph{Computation of the global sections from~$\P^2$.~---} In this example, we do not find a nice explicit basis for the global sections. Instead, we choose to present a {\tt magma} code that permits to construct the generator matrix.

\lstset{frame=lines,linewidth=15cm,basicstyle=\tiny}
\lstset{numbers=left, numberstyle=\tiny, stepnumber=5, numbersep=5pt}
\begin{center}
\begin{lstlisting}
clear ;

q := 7 ;
Fq<xi> := FiniteField(q) ;
P2<X,Y,Z> := ProjectiveSpace(Fq, 2) ; XYZ<X,Y,Z> := CoordinateRing(P2) ;
phi := map < P2 -> P2 | [X^q,Y^q,Z^q] > ;

// Some basic routines such as verifying if some points are in general position
load "Utilities.magma" ;

// Choose randomly a degree 4 point in general position.
repeat
  p1 := Random(P2(ext< Fq | 4>)) ;
  p2 := phi(p1) ; p3 := phi(p2) ; p4 := phi(p3) ;
  Cl1234 := Cluster(p1) ;
until EnPositionGenerale([p1,p2,p3,p4]) ;

//// To compute the points p5, p6 we need the extend the scalars to Fq4
P2_Fq4 := BaseChange(P2, ext < Fq | 4 >) ;

p1_Fq4 := P2_Fq4!ElementToSequence(p1) ; p2_Fq4 := P2_Fq4!ElementToSequence(p2) ;
p3_Fq4 := P2_Fq4!ElementToSequence(p3) ; p4_Fq4 := P2_Fq4!ElementToSequence(p4) ;

L12 := Scheme(P2_Fq4, Sections(LinearSystem(LinearSystem(P2_Fq4, 1), [p1_Fq4,p2_Fq4]))[1]) ; 
L34 := Scheme(P2_Fq4, Sections(LinearSystem(LinearSystem(P2_Fq4, 1), [p3_Fq4,p4_Fq4]))[1]) ; 

p5_Fq4 := Points(L12 meet L34)[1] ;
p5 := P2(ext < Fq | 2 >)!ElementToSequence(p5_Fq4) ;
//// End of the computation over Fq4

p6 := phi(p5) ;
Cl56 := Cluster(p5) ;
Cl56_square := Cluster(P2, Ideal(Cl56)^2) ;

L := LinearSystem(LinearSystem(P2, 4), Cl1234) ;
L := LinearSystem(L, Cl56_square) ;
TheSections := Sections(L) ;

L56 := Scheme(P2, Sections(LinearSystem(LinearSystem(P2, 1), Cl56))[1]) ;
S := (Points(P2) diff Points(L56)) join {@ Points(L56)[1]@} ;

G := Matrix(Fq,5,1+q^2, &cat [[Evaluate(f,ElementToSequence(p)) : p in S] : f in TheSections]) ;
TheCode := LinearCode(G) ;

printf "Generator matrix G =\n%o\n", G ;
lg := Length(TheCode) ;  dim := Dimension(TheCode) ; d_min := MinimumWeight(TheCode) ;
printf "q = %o,\n[n,k,d] = [%o, %o, %o]\n", q, lg, dim, d_min ;
printf "What was provided (not taking into account Kashelev remark) = [%o, 5, %o]",
                                                             q^2+1, q^2 - q - Floor(2*Sqrt(q)) ;
\end{lstlisting}  
\end{center}

\subsection{Degree~$4$, singularity of type~$\mathbf{A}_2$} \label{sDeg4A2}

This example corresponds to the type number~$30$ in degree~$4$ \cite{Classification}.
This type works almost as the one described in section~\ref{sDeg44A1}.

\paragraph{Configuration to blow-up and down.~---}
Let~$p_1,p_2=p_1^{\sigma},p_3=p_1^{\sigma^2},p_4=p_1^{\sigma^3}\in\P^2$ be four conjugate points in general position (no three of them are
collinear). The two lines~$(p_1p_3)$ and~$(p_2p_4)$ are conjugate. We choose a degree~$2$ point~$p_5$ on~$(p_1p_3)$ and we
let~$p_6 = p_5^\sigma$ which lies on~$(p_2p_4)$.
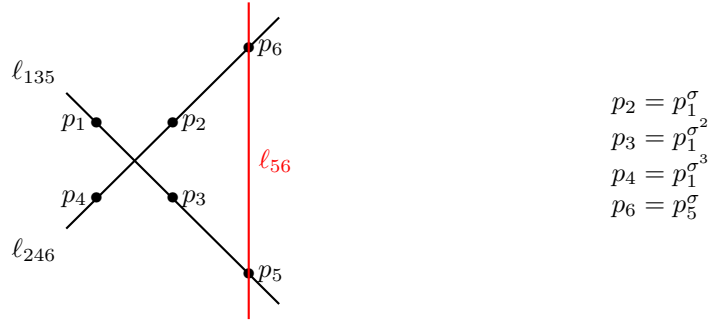
\begin{align*}
&\begin{tikzpicture}[baseline=0]
\def\x{0.5};
\def\h{1}
\coordinate (p1) at (-\x,\x);
\coordinate (p2) at (\x,\x);
\coordinate (p3) at (\x,-\x);
\coordinate (p4) at (-\x,-\x);
\coordinate (p6) at (\x+\h,\x+\h);
\coordinate (p5) at (\x+\h,-\x-\h);
\draw (p1) node {$\bullet$} node[left] {$p_1$};
\draw (p2) node {$\bullet$} node[right] {$p_2$};
\draw (p3) node {$\bullet$} node[right] {$p_3$};
\draw (p4) node {$\bullet$} node[left] {$p_4$};
\draw (p5) node {$\bullet$} node[right] {$p_5$};
\draw (p6) node {$\bullet$} node[right] {$p_6$};
\draw[thick] ($(p1)!-0.2!(p5)$) node[above left] {$\ell_{135}$} -- ($(p5)!-0.2!(p1)$);
\draw[thick] ($(p4)!-0.2!(p6)$) node[below left] {$\ell_{246}$} -- ($(p6)!-0.2!(p4)$);
\draw[thick,red] ($(p5)!-0.2!(p6)$) -- ($(p6)!-0.2!(p5)$) node[midway,right] {$\ell_{56}$};
\end{tikzpicture}
&
&\begin{array}{l}
p_2=p_1^\sigma\\
p_3=p_1^{\sigma^2}\\
p_4=p_1^{\sigma^3}\\
p_6=p_5^\sigma
\end{array}
\end{align*}
The surfaces of the diagram~\eqref{eqP2YXXs} are the following:
we blow up the six points to obtain the degree~$3$ weak del Pezzo surface~$Y$. On this surface, the strict transform of the
line~$\ell_{56}$, of class~$E_0-E_5-E_6$, is an exceptional curve that can be contracted to obtain
the weak degree four weak del Pezzo surface~$X$ defined over $\F_q$. The anticanonical model has a unique singular point of type~$\mathbf{A}_2$
(since there are only two irreducible effective root that meet, see below).

\paragraph{Computation of the divisor class groups.~---}  Over~$\overline{\F}_q$, we know
that~$\cldivweil(\overline{Y}) = \bigoplus_{i=0}^6 \Z E_i$ and that~$-K_Y=3E_0 - \sum_{i=1}^6 E_i$. There are only two
irreducible effective roots on~$\overline{Y}$, the strict transforms
of the lines~$\ell_{135}$ and~$\ell_{246}$ whose conjugate classes in~$\cldivweil(\overline{Y})$
are~$E_0-E_1-E_3-E_5$ and~$E_0-E_2-E_4-E_6$.

The group~$\cldivweil(\overline{X})$ identifies with~$\Z(E_0-E_5-E_6)^\perp$ inside~$\cldivweil(\overline{Y})$. We recover the same orthogonal
decomposition as in~$\eqref{eqDeg4Type48OrthoDecomp}$. In particular, we still
have~$-K_X=4E_0 - \sum_{i=1}^4 E_i - 2E_5 - 2E_6$. 
On~$\overline{X}$ there are still two (conjugate) irreducible effective roots, of classes~$E_0-E_1-E_3-E_5$ and~$E_0-E_2-E_4-E_6$.

We follow the same computation as in section~\ref{sDeg44A1} and we still put~$\EE = (E_0-E_5)+(E_0-E_6)$ and~$\Fcal=E_1+E_2+E_3+E_4$.
Then, for~$X$ we have:
\begin{align*}
&-K_X = 2\Fcal - \EE,
&
&\RR = \Z(\Fcal-\EE),
&
&\text{and}
&
&\cldivweil(X) = \Z(\Fcal-\EE)\oplus\Z\EE,
\end{align*}
and for~$X_s$ we deduce that:
\begin{align*}
&\cldivcartier(X_s) = \Z (2\Fcal-\EE) = \Z(-K_X)
&
&\begin{array}{ccc}
\cldivweil(X_s) &\overset{\simeq}{\longrightarrow} & \Z\EE\\
a\Fcal+b\EE & \longmapsto &  (a+b)\EE
\end{array}
&
&
  \begin{array}{ccc}
    \cldivcartier(X_s) & \overset{\simeq}{\rightarrow} & \cldivweil(X_s)\\
     -K_X    & \mapsto & \EE
  \end{array}
\end{align*}
The canonical embedding induces an isomorphism between the two class groups.

\paragraph{Types of decomposition into irreducible components in~$\left|-K_X\right|$.~---} Since the two class groups are isomorphic
and free of rank one, all the sections are necessarily irreducible. However, they can be absolutely reducible and we need to review
all the possibilities.

As in section~\ref{sDeg44A1}, we are reduced to list all the types of irreducible decompositions of quartics passing through the six
points, the last two with multiplicities at least~$2$. We follow the same line.

The orbits of lines of degree less than~$4$ that involve the six points are
\begin{align*}
&\ell_5\cup\ell_6,
&
&\ell_1\cup\ell_2\cup\ell_3\cup\ell_4,
&
&\ell_{56},
&
&\ell_{12}\cup\ell_{23}\cup\ell_{34}\cup\ell_{14},
&
&\ell_{16}\cup\ell_{25}\cup\ell_{36}\cup\ell_{45},
&
&\text{and}
&
&\ell_{135}\cup\ell_{246}.
\end{align*}
There are five ways (cases~$1$ to~$5$ below) to combine these configurations in order to obtain a curve in the expected linear system.

The orbits of conics of degree less than~$4$ that involve the six points are~$q_{1234}$ and~$q_{56}$. Note that
compared to the example of section~\ref{sDeg44A1}, the orbit~$q_{1356}\cup q_{2456}$ does not appear since a conic~$q_{1356}$
cannot be irreducible otherwise it would have three intersection points with the line~$\ell_{135}$.
There are only two ways (cases~$6$ and~$7$ below) to combine the configurations of lines and conics in order to obtain a curve in the
expected linear system.

If the decomposition contains a cubic, it must be smooth at $p_5$ and $p_6$ and the complement must be~$\ell_{56}$; this is case~$8$. 

This leads to the list below. In~$Y$, the strict transforms~$\widetilde{\ell}_{135}$ and~$\widetilde{\ell}_{246}$ are separated from the
strict transform~$\widetilde{\ell}_{56}$. In~$X$ this last curve is contracted to a smooth rational point~$p$.
Last, another difference from the example of section~\ref{sDeg44A1}:
the two effective roots of~$X$ meet and they are thus contracted by the anticanonical morphism~$\plongement_*$
onto the same point~$s$. This point is the unique singular point of~$X_s$ and it is necessarily a rational point.
$$
{\renewcommand{\arraystretch}{1.5}
\begin{array}{|r|l|l|l||l|l||l|l|}
\hline
&\left|4\ell - \sum_{i=1}^4 p_i -2p_5-2p_6\right|  &\left|-K_X\right| & \left|-K_{X_s}\right| & \text{Max}   \\
&\text{on~$\P^2$}                 & \text{on~$X$}             & \text{on~$X_s$}  & \text{nb. of pts}\\
\hline\hline
1&
2\ell_{56}\cup\ell_{135}\cup\ell_{246}
&\widetilde{\ell}_{56}\cup\widetilde{\ell}_{135}\cup\widetilde{\ell}_{246}\cup E_5 \cup E_6
&\plongement_*(E_5) \cup \plongement_*(E_6)\cup \{s, p\} 
& 2 
\\
\hline
2&
\ell_{56}\cup\ell_{135}\cup\ell_{246}\cup\ell
&\widetilde{\ell}_{135}\cup\widetilde{\ell}_{246}\cup\widetilde{\ell}
&\plongement_*(\widetilde{\ell})
&q+2
\\
\hline
3&
\ell_{16}\cup \ell_{25} \cup \ell_{36} \cup \ell_{45}
&\widetilde{\ell}_{16}\cup \widetilde{\ell}_{25} \cup \widetilde{\ell}_{36} \cup \widetilde{\ell}_{45}
&\plongement_*(\widetilde{\ell}_{16})\cup \plongement_*(\widetilde{\ell}_{25}) \cup \plongement_*(\widetilde{\ell}_{36}) \cup \plongement_*(\widetilde{\ell}_{45})
&0
\\
\hline
4&
\ell_{5}\cup\ell_{6}\cup\ell_{135}\cup\ell_{246}
&\widetilde{\ell}_{5}\cup\widetilde{\ell}_{6}\cup\widetilde{\ell}_{135}\cup\widetilde{\ell}_{246}
&\plongement_*(\widetilde{\ell}_{5})\cup \plongement_*(\widetilde{\ell}_{6}) \cup \{s\}
& 2 
\\
\hline
5&
2\ell_{135}\cup2\ell_{246}
&\widetilde{\ell}_{135}\cup\widetilde{\ell}_{246}\cup E_1 \cup E_2 \cup E_3 \cup E_4
&\plongement_*(E_1) \cup \plongement_*(E_2) \cup \plongement_*(E_3) \cup \plongement_*(E_4) 
&1
\\
\hline
6&
\ell_{135}\cup\ell_{246}\cup q_{56}
&\widetilde{\ell}_{135}\cup\widetilde{\ell}_{246}\cup \widetilde{q}_{56}
&\plongement_*(\widetilde{q}_{56}) \cup \{s\}
& q+2
\\
\hline
7&
2\ell_{56}\cup q_{1234}
&\widetilde{q}_{1234}\cup \{p\}
&\plongement_*(\widetilde{q}_{1234}) \cup \plongement_*(\{p\})
&q+2
\\
\hline
8&
c_{123456} \cup \ell_{56}
&\widetilde{c}_{123456}
&\plongement_*(\widetilde{c}_{123456})
&\Nqg{q}{1}
\\
\hline
9&
t_{123456} \text{ singular at~$p_5,p_6$}
&\widetilde{t}_{123456}
&\plongement_*(\widetilde{t}_{123456})
&\Nqg{q}{1}
\\
\hline
\end{array}}
$$
Some comments about the numbers of points.

{\bf Case 1 \& 5.} The exceptional curves~$E_5$ and~$E_6$ do not contain any rational points. The other components are all contracted
to the points~$p$ or~$s$. The same is true in case~$5$, without the point~$p$. 

{\bf Case 2.} The ending curve passes through~$p$ but the contraction of~$\widetilde{\ell}_{56}$ does not add any rational
point since~$\ell_{56}$ and~$\ell$ meet at a rational point. The ending curve passes through~$s$ and the contraction of the roots add
a point if~$\ell$ meet~$\ell_{135}$ and~$\ell_{246}$ outside the meeting point of these two curves.

{\bf Case 3.} All theses lines and~$\ell_{135}$, $\ell_{246}$ and~$\ell_{56}$ are separated by the blowing ups. Since the four lines
cannot contain any rational point, neither does their image in~$X_s$.

{\bf Case 4.} The conjugate lines~$\ell_5$ and~$\ell_6$ contain a unique rational point, their intersection point, to which
is added the point~$s$.

{\bf Case 6.} The curve~$\widetilde{q}_{56}$ no longer meets~$\widetilde{\ell}_{135}$,~$\widetilde{\ell}_{246}$ and~$\widetilde{\ell}_{56}$. The ending curve contains the rational points of~$q_{56}$ plus the point~$s$.

{\bf Case 7.} If~$q_{1234}$ meets~$\ell_{56}$ at two conjugate points then in~$X$, after~$\widetilde{\ell}_{56}$ being contracted,
the strict transform~$\widetilde{q}_{1234}$ passes through~$p$ which is an additional rational point. Necessarily the blowing ups
of~$p_1,\ldots,p_4$ separate the strict transforms~$\widetilde{q}_{1234}$, $\widetilde{\ell}_{135}$ and~$\widetilde{\ell}_{246}$ and
the roots contraction does not add any point.

{\bf Cases 8 \& 9.} Same as~\S\ref{sDeg44A1}.

Finally~$N_q\left(-K_{X_s}\right)\leq N_q(1)$.

\medbreak

As in the previous example, one has~$\#X_s(\F_q) = q^2+1$, and for~$q=2$, the evaluation map may not be injective.

\begin{prop}
Suppose~$q\not=2$.
Let~$p_1,p_2=p_1^{\sigma},p_3=p_1^{\sigma^2},p_4=p_1^{\sigma^3}\in\P^2$ be four conjugate points in general position (no three of them are
collinear), let~$p_5$ be a point of the line~$(p_1p_3)$ inside~$\P^2(\F_{q^2})$ and let~$p_6 = p_5^{\sigma}$ in such a way
that~$p_6$ lies on~$(p_2p_4)$.
The anticanonical code of the weak del Pezzo surface obtained
by blowing up these six points and then blowing down the strict transform of the line~$(p_5p_6)$ has
parameters~$[q^2+1,5,\geq q^2 + 1 - \Nqg{q}{1}]$.
\end{prop}

\paragraph{Computation of the global sections from~$\P^2$.~---} This example looks like the previous one
and we do not find a nice explicit basis for the global sections. A slightly modification of the code given for the previous
example leads to a program which permits to compute a generator matrix.

\subsection{Degree~$4$, singularity of type~$\mathbf{D}_5$} \label{sDeg4D5}

This example corresponds to the type number~$58$ in degree~$4$ \cite{Classification}.

\paragraph{Configuration to blow-up.~---} In this example, the surfaces~$Y$ and~$X$ of diagram~\eqref{eqP2YXXs} are equal and
we obtain directly the surface~$X$ by blowing up~$\P^2$ at five rational points~$p_1 \prec p_2 \prec \cdots\prec p_5$, with~$p_1,p_2,p_3$
collinear. Let us denote by~$\pi_1,\ldots,\pi_5$ these five blowups at~$p_1,\ldots,p_5$ respectively:
$$
\begin{tikzpicture}[>=latex,baseline=(M.center)]
\matrix (M) [matrix of math nodes,row sep=0.5cm,column sep=1cm]
{
|(P2)| \P^2 & |(X1)| X_1 & |(X2)| X_2 & |(X3)| X_3 & |(X4)| X_4 & |(X)| X\\
};
\draw[->] (X1) -- (P2) node[midway,below] {$\pi_1$} ;
\draw[->] (X2) -- (X1) node[midway,below] {$\pi_2$} ;
\draw[->] (X3) -- (X2) node[midway,below] {$\pi_3$} ;
\draw[->] (X4) -- (X3) node[midway,below] {$\pi_4$} ;
\draw[->] (X) -- (X4) node[midway,below] {$\pi_5$} ;
\draw[->] (X) to[bend right=20] (P2) node[midway,yshift=0.6cm] {$\pi$} ;
\end{tikzpicture}
$$
The fact that~$p_1,p_2,p_3$ are collinear means that there is a line~$\ell_{123}$ of~$\P^2$ whose strict transform by~$\pi_1$ passes
through~$p_2$ and whose strict transform by~$\pi_2\circ\pi_1$ passes through~$p_3$. The anticanonical model of this weak del Pezzo
surface has a unique singular point of type~$\mathbf{D}_5$ (since there are five irreducible effective roots whose intersection
graph is~$\mathbf{D}_5$, see the picture at the end of this example).

\paragraph{Computation of the divisor class groups.~---}
Since all the blown-up points are rational, there is no need to work with the base change~$\overline{X}$.
The irreducible effective classes of roots are the strict transform of~$\ell_{123}$ and of~$E_1,E_2,E_3,E_4$, whose
classes in~$\cldivweil(X)$ are:
\begin{align*}
&E_0-E_1-E_2-E_3,
&
&E_1-E_2,
&
&E_2-E_3,
&
&E_3-E_4,
&
&\text{and}
&
&E_4-E_5.
\end{align*}
The submodule~$\RR$, generated by these classes, is a direct summand and for example~$\cldivweil(X) = \RR \oplus \Z E_5$; the projection onto the factor~$\Z E_5$
leads to an isomorphism~$\cldivweil(X)/\RR \to \Z E_5$. As for the submodule~$\RR^\perp$, it is defined by the
equations~$a_1=a_2=\cdots=a_5$ and~$a_0+a_1+a_2+a_3=0$ and thus~$\RR^\perp = \Z K_X$.
Since
$$
-K_X = 3\left(E_0-E_1-E_2-E_3\right) + 2\left(E_1-E_2\right) + 4\left(E_2-E_3\right) + 6\left(E_3-E_4\right) + 5\left(E_4-E_5\right) + 4E_5,
$$
via the preceding isomorphism, the module~$\RR^\perp$ embeds via~$-K_X \mapsto 4E_5$.

In brief, both divisor class groups~$\cldivcartier(X_s)$ and~$\cldivweil(X_s)$ are free rank one $\Z$-modules, the first one being of index~$4$
in the latter via the canonical embedding. 

For this example, it makes sense to reverse the order of the paragraphs and we start
to compute a basis of the global sections.

\paragraph{Computation of the global sections from~$\P^2$.~---} We need to compute a basis of the sublinear system on~$\P^2$
$$
\left|3\ell - p_1 - \cdots - p_5\right|.
$$
So we consider a cubic of~$\P^2_{X,Y,Z}$ whose restriction
to the affine space~$\A^2_{x_1,y_1}$ ($x_1=\frac{X}{Z}$, $y_1=\frac{Y}{Z}$ and~$Z\not=0$) is defined by the equation:
$$
C_1(x_1,y_1)=
a_{30}x_1^3 + a_{21}x_1^2y_1 + a_{20}x_1^2 + a_{12}x_1y_1^2 + a_{11}x_1y_1 + a_{10}x_1 + a_{03}y_1^3 + a_{02}y_1^2 + a_{01}y_1 + a_{00}
=0
$$
We choose~$p_1=(0,0) \in \A^2_{x_1,y_1}$. The cubic passes through the point~$p_1$ if and only if~$a_{00}=0$.

Let~$x_2,y_2$ be the coordinates of the affine chart of the blowing up of~$\A^2_{x_1,y_1}$ at~$p_1$ defined by
$x_1 = x_2$ and~$y_1=x_2y_2$.
In this chart, the exceptional divisor~$E_1$ has equation~$x_2=0$ and the strict transform of~$C_1$ is defined by:
$$
C_2(x_2,y_2)
=
\frac{1}{x_2}C_1(x_2,x_2y_2)
=
a_{30}x_2^2 + a_{21}x_2^2y_2 + a_{20}x_2 + a_{12}x_2^2y_2^2 + a_{11}x_2y_2 + a_{10} + a_{03}x_2^2y_2^3 + a_{02}x_2y_2^2 + a_{01}y_2
=0.
$$
We choose~$p_2=(0,0) \in \A^2_{x_2,y_2}$ which corresponds to the line~$\ell_{123}$ with affine equation~$y_1=0$ in~$\A^2_{x_1,y_1}$
or~$Y. = 0$ in~$\P^2_{X,Y,Z}$.
The cubic passes through the point~$p_2$ if and only if~$a_{10}=0$.

Let~$x_3,y_3$ be the coordinates of the affine chart of the blowing up of~$\A^2_{x_2,y_2}$ at~$p_2$ defined
by~$x_2 = x_3$ and~$y_2=x_3y_3$. In this chart, the exceptional divisor~$E_2$ has equation~$x_3=0$ and the strict transform of~$C_2$ is
defined by:
$$
C_3(x_3,y_3)
=
\frac{1}{x_3}C_2(x_3,x_3y_3)
=
a_{30}x_3 + a_{21}x_3^2y_3 + a_{20} + a_{12}x_3^3y_3^2 + a_{11}x_3y_3 + a_{03}x_3^4y_3^3 + a_{02}x_3^2y_3^2+ a_{01}y_3
=0.
$$
Since~$p_1,p_2,p_3$ are colinear, we have to choose~$p_3=(0,0) \in \A^2_{x_3,y_3}$. The cubic passes through the point~$p_3$
if and only if~$a_{20}=0$.

Let~$x_4,y_4$ be the coordinates of the affine chart of the blowing up of~$\A^2_{x_3,y_3}$ at~$p_3$ defined
by~$x_3 = x_4$ and~$y_3=x_4y_4$. In this chart, the exceptional divisor~$E_3$ has equation~$x_4=0$ and the strict transform of~$C_3$ is
defined by:
$$
C_4(x_4,y_4)
=
\frac{1}{x_4}C_2(x_4,x_4y_4)
=
a_{30} + a_{21}x_4^2y_4 + a_{12}x_4^4y_4^2 + a_{11}x_4y_4 + a_{03}x_4^6y_4^3 + a_{02}x_4^3y_4^2+ a_{01}y_4
=0.
$$
Since~$p_4\in E_3$, we have to choose~$p_4=(0,\alpha) \in \A^2_{x_4,y_4}$ and since~$p_1,p_2,p_3,p_4$ are {\bfseries not} colinear,
necessarily~$\alpha\not=0$. The cubic passes through the point~$p_4$ if and only if~$a_{30}+\alpha a_{01}=0$.

Last, let~$x_5,y_5$ be the coordinates of the affine chart of the blowing up of~$\A^2_{x_4,y_4}$ at~$p_4$ defined
by~$x_4 = x_5$ and~$y_4=\alpha+x_5y_5$. In this chart, the exceptional divisor~$E_4$ has equation~$x_5=0$ and the strict transform of~$C_4$ is
defined by:
\begin{align}\label{eqC5}
\begin{split}
C_5(x_5,y_5)
&=
\frac{1}{x_5}C_2(x_5,x_5y_5)\\
&=
\frac{1}{x_5}\left[a_{30} + a_{21}x_5^2(\alpha+x_5y_5) + a_{12}x_5^4(\alpha+x_5y_5)^2 + a_{11}x_5(\alpha+x_5y_5) + a_{03}x_5^6(\alpha+x_5y_5)^3\right.\\
&\qquad\quad\left. + a_{02}x_5^3(\alpha+x_5y_5)^2 + a_{01}(\alpha+x_5y_5)\right]\\
&\equiv \alpha a_{11} +\alpha a_{21}x_5 + a_{01}y_5 + a_{11}x_5y_5 \bmod{x_5^2\F_q[x_5,y_5]}.
\end{split}
\end{align}
Since~$p_5\in E_4$, one can choose~$p_5=(0,\beta) \in \A^2_{x_5,y_5}$.
The cubic passes through the point~$p_5$ if and only if~$\alpha a_{11}+\beta a_{01}=0$.

To sum up, the global sections are defined by
\begin{align*}
&a_{00}=a_{10}=a_{20}=0,
&
&a_{30}=-\alpha a_{01},
&
&\text{and}
&
&a_{11} = -\frac{\beta}{\alpha}a_{01}.
\end{align*}
The fact that~$\alpha\not=0$ is important here. In the projective setting, this
leads to the basis
\begin{equation}\label{eqSectionsDegre4Type58}
\left|3\ell - p_1 - \cdots - p_5\right|
=
\left\langle
\alpha YZ^2 - \beta XYZ - \alpha^2 X^3,Y^3,X^2Y,XY^2,Y^2Z
\right\rangle_{\F_q}.
\end{equation}

\paragraph{Types of decomposition into irreducible components in~$\left|-K_X\right|$.~---}
Since~$\cldivcartier(X_s)$ if of index~$4$ inside~$\cldivweil(X_s)$, even if these two groups are free of rank~$1$, an irreducible Cartier divisor may
decompose into Weil irreducible components. In order to lower
bound the minimum distance, we need to review all these kinds of decompositions into irreducible components for the curves of the
anticanonical linear system on~$X_s$. As usual, we start form~$\P^2$ and use the one-to-one correspondences:
$$
\begin{array}{ccccc}
\left|3\ell - p_1 - \cdots - p_5\right|_{\P^2} & \longrightarrow & \left|-K_X\right|_{X} &\longrightarrow & \left|-K_{X_s}\right|_{X_s}\\
C & \longmapsto & C^\sharp & \longmapsto &\plongement\left(C^\sharp\right)
\end{array}
$$
where~$C^\sharp$ denotes the virtual transform of~$C$ in the composition of the five blowups.

Thanks to the preceding computation, for every curve in~$\left|3\ell - p_1 - \cdots - p_5\right|_{\P^2}$ there
exists~$\alpha_1,\ldots,\alpha_5 \in \F_q$ such this curves is defined by
$$
\alpha_1\left(\alpha YZ^2 - \beta XYZ - \alpha^2 X^3\right)+\alpha_2 Y^3+\alpha_3 X^2Y + \alpha_4 XY^2 + \alpha_5 Y^2Z = 0.
$$
We deduce that  such a curve can decompose in six different ways, as listed in the tabular below:
\begin{itemize}
\item either a cubic~$c_{12345}$ for which~$p_1$ is a smooth flex point with tangent line equal to~$\ell_{123}$,
if~$\alpha_1\not=0$ (case~$1$);
\item or a cubic singular at~$p_1$ which contains~$\ell_{123}$ as a component, if~$\alpha_1=0$, the complementary component,
of discriminant~$\alpha_3\alpha_5^2$ (up to a constant), being
either
\begin{itemize}
\item a quadric~$q_{12}$ smooth at~$p_1$ with tangent line~$\ell_{123}$, if~$\alpha_3\not=0$ and~$\alpha_5\not=0$ (case~2),
\item or the union of two lines, if~$\alpha_3\not=0$ and~$\alpha_5=0$, $\alpha_3=0$ and~$\alpha_5\not=0$,
$\alpha_3=\alpha_5=0$ and~$\alpha_4\not=0$, $\alpha_3=\alpha_4=\alpha_5=0$
(cases~$3,4,5,6$ respectively).
\end{itemize}
\end{itemize}
$$
{\renewcommand{\arraystretch}{1.5}
\begin{array}{|l|l|l|l||l|}
\hline
&\left|3\ell - \sum_{i=1}^5 p_i\right|  &\left|-K_X\right| & \left|-K_{X_s}\right| & \text{Max}   \\
&\text{on~$\P^2$}                 & \text{on~$X$}             & \text{on~$X_s$}  & \text{nb. of pts}\\
\hline\hline
1
&c_{12345}
&\widetilde{c}_{12345}
&\plongement_*(\widetilde{c}_{12345})
&\Nqg{q}{1}
\\
\hline
2
&\ell_{123} \cup q_{12}
&\widetilde{\ell}_{123}\cup \widetilde{q}_{12}\cup \widetilde{E}_1\cup 2\widetilde{E}_2\cup 2\widetilde{E}_3\cup \widetilde{E}_4
& \plongement_*(\widetilde{q}_{12}) & q+1
\\
\hline
3
&\ell_{123} \cup \ell_1 \cup \ell_1'
&\widetilde{\ell}_{123}\cup 2\widetilde{E}_1\cup 2\widetilde{E}_2\cup 2\widetilde{E}_3\cup \widetilde{E}_4\cup \widetilde{\ell}_1\cup \widetilde{\ell}_1'
&\plongement_*(\widetilde{\ell}_1) \cup \plongement_*(\widetilde{\ell}_1') & 2q+1
\\
4
&2\ell_{123} \cup \ell
&2\widetilde{\ell}_{123}\cup \widetilde{E}_1\cup 2\widetilde{E}_2\cup 3\widetilde{E}_3\cup 2\widetilde{E}_4\cup E_5\cup \widetilde{\ell}
&\plongement_*(E_5) \cup \plongement_*(\widetilde{\ell})
&2q+1
\\
5
&2\ell_{123} \cup \ell_1
&2\widetilde{\ell}_{123}\cup 2\widetilde{E}_1\cup 3\widetilde{E}_2\cup 4\widetilde{E}_3\cup 3\widetilde{E}_4\cup 2E_5\cup \widetilde{\ell}_1
&2\plongement_*(E_5) \cup \plongement_*(\widetilde{\ell}_1)
&2q+1
\\
6
&3\ell_{123}
&3\widetilde{\ell}_{123}\cup 2\widetilde{E}_1\cup 4\widetilde{E}_2\cup 6\widetilde{E}_3\cup 5\widetilde{E}_4\cup 4E_5
&4\plongement_*(E_5)
&q+1
\\
\hline
\end{array}}
$$
Let us give details for the computation of the virtual transform~$\pi^\sharp(C)$ if, for example,~$C=\ell_{123}\cup\ell_1\cup\ell'_1$:
\begin{align*}
C_1&= \pi_{1}^\sharp(C) = \widetilde{\ell}_{123} + \widetilde{\ell}_1 + \widetilde{\ell}'_1 + 2E_1
&& \left[p_1\in\ell_{123}\cap\ell_1\cap\ell'_1 \; \Rightarrow \; m_{p_1}(C)=3\right]\\
C_2&= \pi_2^\sharp(C_1) = \widetilde{\ell}_{123} + \widetilde{\ell}_1 + \widetilde{\ell}'_1 + 2\widetilde{E}_1 + 2E_2
&& \left[p_2\in \widetilde{\ell}_{123}\cap E_1 \; \Rightarrow \; m_{p_2}(C_1) = 3\right]\\
C_3&= \pi_3^\sharp(C_2) = \widetilde{\ell}_{123} + \widetilde{\ell}_1 + \widetilde{\ell}'_1 + 2\widetilde{E}_1 + 2\widetilde{E}_2 + 2E_3
&& \left[p_3\in \widetilde{\ell}_{123}\cap E_2 \; \Rightarrow \; m_{p_3}(C_2) = 3\right]\\
C_4&= \pi_4^\sharp (C_3) = \widetilde{\ell}_{123} + \widetilde{\ell}_1 + \widetilde{\ell}'_1 + 2\widetilde{E}_1 + 2\widetilde{E}_2 + 2\widetilde{E}_3 +E_4 && \left[p_4\in E_3 \;\Rightarrow\; m_{p_4}(C_3) = 2\right]\\
C^\sharp &= \pi_5^\sharp (C_4) = \widetilde{\ell}_{123} + \widetilde{\ell}_1 + \widetilde{\ell}'_1 + 2\widetilde{E}_1 + 2\widetilde{E}_2 + 2\widetilde{E}_3 +\widetilde{E}_4
&& \left[p_5\in E_4 \; \Rightarrow \; m_{p_5}(C_4) = 1\right].
\end{align*}
This leads to the following decomposition of the canonical class into a sum of effective classes
$$
-K_X
=
(E_0-E_1-E_2-E_3)+(E_0-E_1)+(E_0-E_1)+2(E_1-E_2)+2(E_2-E_3)+2(E_3-E_4)+(E_4-E_5)
$$
The intersection graph of the irreducible effective roots in~$X$ is connected (see figure below) and all these curves are contracted
by the morphism~$\plongement$ to a single rational singular pointy~$s$ (of singularity type~$D_5$).
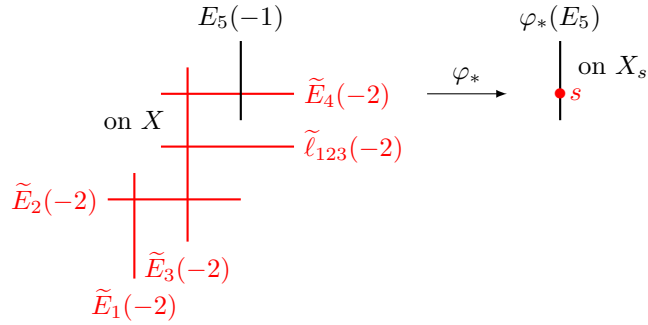
\begin{center}
\begin{tikzpicture}[>=latex,scale=0.7] 
\draw (-1,1.5) node {on~$X$};
\draw[thick,red] (-1,-1.5) node[below] {$\widetilde{E}_1(-2)$} -- (-1,0.5);
\draw[thick,red] (-1.5,0) node[left] {$\widetilde{E}_2(-2)$} -- (1,0);
\draw[thick,red] (0,-0.8) node[below] {$\widetilde{E}_3(-2)$} -- (0,2.5);
\draw[thick,red] (-0.5,1) -- (2,1) node[right] {$\widetilde{\ell}_{123}(-2)$} ;
\draw[thick,red] (-0.5,2) -- (2,2) node[right] {$\widetilde{E}_4(-2)$};
\draw[thick] (1,1.5) -- (1,3) node[above] {$E_5(-1)$};
\draw[->] (4.5,2) -- (6,2) node[midway,above] {$\plongement_*$};
\draw[thick] (7,1.5) -- (7,3) node[above] {$\plongement_*(E_5)$};
\draw[red] (7,2) node {$\bullet$} node[right] {$s$};
\draw (8,2.5) node {on~$X_s$};
\end{tikzpicture}
\end{center}
We comment on the numbers of points.

{\bfseries Cases~2 \& 6.} All the components in~$X$ are roots that are contracted, except~$\widetilde{q}_{12}$ and~$E_5$ respectively.
These two strict transforms meet the tree of roots at only one point and by~$\plongement_*$ they are mapped to isomorphic curves
that pass through~$s$.

{\bfseries Case~3.} Except~$\widetilde{\ell}_1$ and~$\widetilde{\ell}'_1$, all the components on~$X$ are irreducible effective roots and
they are mapped to the point~$s$ by the morphism~$\plongement_*$. After the contraction,
the curves~$\plongement_*(\widetilde{\ell}_1)$ and~$\plongement_*(\widetilde{\ell}'_1)$ meet at this singular point, thus their union
contains~$2q+1$ rational points.

{\bfseries Cases~4 \& 5.} The line~$E_5$ does not intersect the lines~$\widetilde{\ell}$ or~$\widetilde{\ell}_1$ in~$X$.
However, since~$\ell$ and~$\ell_{123}$ meet at some point of~$\P^2$
(not equal to~$p_1$), the lines~$\widetilde{\ell}$ and~$\widetilde{\ell}_{123}$ intersect in~$X$; in the same way
since~$\ell_1$ passes through~$p_1$, the lines~$\widetilde{\ell}_1$ and~$\widetilde{E}_1$ intersect in~$X$. Therefore,
$\plongement_*(E_5)$ and~$\plongement_*(\widetilde{\ell})$ or~$\plongement_*(E_5)$ and~$\plongement_*(\widetilde{\ell}_1)$ both intersect at~$s$.
Thus the two unions has~$2q+1$ rational points.

Finally~$N_q\left(-K_{X_s}\right) = 2q+1$.

\medbreak

The surface~$X_s$ has a unique singular point~$s$. All the irreducible effective roots of~$X$,
that is~$\widetilde{\ell}_{123},\widetilde{E}_1,\ldots,\widetilde{E_4}$ are contracted to this single point.
The last exceptional curve~$E_5$ meets~$E_4$ and thus~$\plongement(E_5)$ passes through~$s$. In conclusion
the rational points of~$X_s(\F_q)$ are in one-to-one
correspondence with~$\left(\P^2(\F_q)\setminus\ell_{123}(\F_q)\right)\cup E_5(\F_q)$, which counts~$q^2+q+1$ elements.
This number is always strictly greater that~$2q+1$ and the evaluation map is always injective.

\begin{prop}
Let~$p_1\prec p_2\prec p_3\prec p_4\prec p_5$ be infinitely near rational points. Suppose that the first three ones are collinear.
The anticanonical code of the weak del Pezzo surface obtained
by blowing up these points has parameters~$[q^2+q+1,5,q^2-q]$.
\end{prop}

The construction of a generator matrix of this code is a nice application of proposition~\ref{propCodeBlowingup}. It
consists of two blocks, the left one, of size~$5\times q^2$, contains the evaluations of the five
global sections of~(\ref{eqSectionsDegre4Type58}) at every point of~$\P^2(\F_q)\setminus\ell_{123}(\F_q)$, the right
one, of size~$5\times(q+1)$ contains the evaluations of the homogeneous parts of degree~$1$ of the five
global sections of~(\ref{eqSectionsDegre4Type58}) at every point of~$\P^1(\F_q)$. Letting~$y_5 = \beta + z_5$ in~\eqref{eqC5}, this
homogeneous part of degree~$1$ equals~$(\alpha a_{21}+\beta a_{11})x_5+a_{01}z_5$. Finally, we get the explicit matrix:
$$
\left(
\begin{array}{cc|cc}
\alpha yz^2 - \beta xyz - \alpha^2 x^3&\vdots
&-\beta^2 u+ \alpha v&\vdots
\\
y^3&\vdots
&0&\vdots
\\
x^2y& (x:y:z)\in\P^2(\F_q) \mid y\not=0
&\alpha u&(u:v)\in\P^1(\F_q)
\\
xy^2&\vdots
&0&\vdots
\\
y^2z&\vdots
&0&\vdots
\end{array}
\right),
$$
where~$\alpha\in\F_q^*$ and~$\beta\in\F_q$.

\subsection{Degree~$3$, singularity of type~$\mathbf{A}_1$} \label{sDeg3A1}

This example corresponds to the type number~$11$ in degree~$3$ \cite{Classification}.

\paragraph{Configuration to blow-up and down.~---} We blow up six conjugate points~$p_1,\ldots,p_6\in\P^2$ on a smooth
conic~$q_{123456}$.
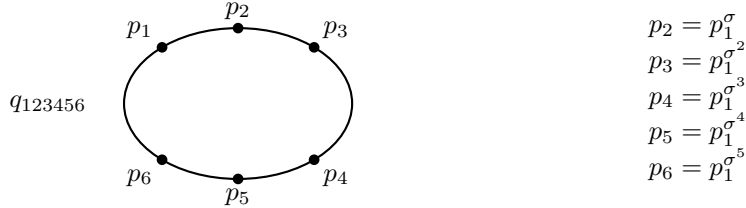
\begin{align*}
&\begin{tikzpicture}[baseline=0]
  \def\a{1.5} 
  \def\b{1.0} 
  \def\x{-1}\def\xx{0}\def\xxx{1}
  \coordinate (p6) at (\x,{-\b*sqrt(1-(\x/\a)^2)});
  \coordinate (p1) at (\x,{\b*sqrt(1-(\x/\a)^2)});
  \coordinate (p2) at (\xx,{\b*sqrt(1-(\xx/\a)^2)});
  \coordinate (p5) at (\xx,{-\b*sqrt(1-(\xx/\a)^2)});
  \coordinate (p3) at (\xxx,{\b*sqrt(1-(\xxx/\a)^2)});
  \coordinate (p4) at (\xxx,{-\b*sqrt(1-(\xxx/\a)^2)});
  \draw[thick] (0,0) ellipse({\a} and {\b});
  \fill (p1) circle (2pt) node[above left] {$p_1$};
  \fill (p2) circle (2pt) node[above] {$p_2$};
  \fill (p3) circle (2pt) node[above right] {$p_3$};
  \fill (p4) circle (2pt) node[below right] {$p_4$};
  \fill (p5) circle (2pt) node[below] {$p_5$};
  \fill (p6) circle (2pt) node[below left] {$p_6$};
  \draw (-2,0) node[xshift=-0.5cm] {$q_{123456}$};
\end{tikzpicture}
&
&\begin{array}{l}p_2=p_1^\sigma\\p_3=p_1^{\sigma^2}\\p_4=p_1^{\sigma^3}\\p_5=p_1^{\sigma^4}\\p_6=p_1^{\sigma^5}\end{array}
\end{align*}
The resulting surface~$X$ is a weak del Pezzo of degree~$3$, whose anticanonical model~$X_s$ has a unique singular point
of type~$\mathbf{A}_1$.

\paragraph{Computation of the divisor class groups.~---} We have:
\begin{align*}
&\cldivweil(\overline{X}) = \sum_{i=0}^6 \Z E_i
&
&\text{and}
&
&\cldivweil(X) = \Z E_0 \oplus \Z \EE,
&
&\text{where~$\EE = \sum_{i=1}^6 E_i$.}
\end{align*}
There is a unique irreducible effective root, the strict transform of the conic~$q_{123456}$, whose
class is~$2E_0 - \EE$. The root module~$\RR$, generated this class, is a direct summand, $\cldivweil(X) = \RR \oplus \Z E_0$. The projection
onto the second factor leads to an isomorphism
$$
\begin{array}{ccc}
  \cldivweil(X)/\RR & \longrightarrow & \Z E_0\\
  E_0\bmod\RR & \longmapsto & E_0\\
  \EE\bmod\RR & \longmapsto & 2E_0
\end{array}.
$$
As for the module~$\RR^\perp$, inside~$\cldivweil(\overline{X})$ it is defined by the single equation~$2a_0+a_1+\cdots+a_6 = 0$;
after taking the Galois invariants, we obtain~$\cldivcartier(X_s) = \Z K_X$, whose image by the previous isomorphism is also~$\Z E_0$.
Therefore~$\cldivcartier(X_s)\simeq\cldivweil(X_s)$ and both of them are free of rank one.

\paragraph{Types of decomposition into irreducible components in~$\left|-K_X\right|$.~---} This proves that all the sections of the
anticanonical divisor are irreducible. As in our previous work \cite{AntiCanonical}, we expect that the curves of
the associated linear system can contain at most~$\Nqg{q}{1}$ rational points.
However we need to investigate the types of irreducible decompositions. Here this is easy since one can check that the only
Galois orbits of lines or conics or cubics that pass through at least one point~$p_i$ are~$\ell_{14}\cup\ell_{25}\cup\ell_{36}$
or~$q_{123456}$ or~$c_{123456}$ (all the others lead to $\F_q$-curves of degree strictly greater than~$3$).
Combining them in order to construct a curve in the expected sub-linear system leads to very few decompositions:
$$
{\renewcommand{\arraystretch}{1.5}
\begin{array}{|l|l|l|l||l|l||l|l|}
\hline
&\left|3\ell - \sum_{i=1}^6 p_i\right|  & \left|-K_X\right| & \left|-K_{X_s}\right| & \text{Max}   \\
&\text{on~$\P^2$}  & \text{on~$X$}  & \text{on~$X_s$}  & \text{nb. of pts}\\
\hline\hline
1
&\ell_{14} \cup \ell_{25} \cup \ell_{36}
&\widetilde{\ell}_{14}\cup \widetilde{\ell}_{25}\cup\widetilde{\ell}_{36}
&\plongement_*(\widetilde{\ell}_{14}) \cup \plongement_*(\widetilde{\ell}_{25}) \cup \plongement_*(\widetilde{\ell}_{36}) & 1
\\
\hline
2
&\ell \cup q_{123456}
&\widetilde{\ell}\cup \widetilde{q}_{123456}
& \plongement_*(\widetilde{\ell})\ni s & q+2
\\
\hline
3
&c_{123456}
&\widetilde{c}_{123456}
&\plongement_*(\widetilde{c}_{123456})
&\Nqg{q}{1}
\\
\hline
\end{array}}
$$
The number of rational point in case~$1$ is at most~$1$ if the three lines meet. In case~$2$, during the process, if the two
meeting points of~$\ell$ and~$q_{123456}$ are not rational, then the singular point~$s$ is an additional rational point
on~$\plongement_*(\widetilde{\ell})$. We deduce that~$N_q\left(-K_{X_s}\right) \leq N_q(1)$.

\medbreak

Since the blown up points are not rational, the blowing ups do not add point on the surface and~$\#X(\F_q) = q^2+q+1$. Then, the
irreducible effective root is contracted and thus~$\#X_s(\F_q) = q^2+1$. If~$q=2$, the evaluation map may fail to be injective.

\begin{prop}
Suppose~$q\not=2$.
Let~$p_1,\ldots,p_6\in\P^2$ be six conjugate points lying on a smooth conic.
The anticanonical code of the weak del Pezzo surface obtained
by blowing up these points has parameters~$[q^2+1,4,\geq q^2+1-\Nqg{q}{1}]$.
\end{prop}

\paragraph{Computation of the global sections from~$\P^2$.~---}

Let~$Q$ denote the conic passing through~$p_1,\ldots,p_6$ and let~$L_{14},L_{25},L_{36}$ be the linear forms whose zeros
are the lines~$\ell_{14}$,~$\ell_{25}$,~$\ell_{36}$. Then
$$
H^0\left(\P^2,3\ell-\textstyle\sum_{i=1}^6 p_i\right)
=
\left\langle XQ,YQ,ZQ,L_{14}L_{25}L_{36}\right\rangle_{\F_q}
$$

\subsection{Degree~$3$, singularity of type~$3\mathbf{A}_2$} \label{sDeg33A2}

This example corresponds to the type number~$76$ in degree~$6$ \cite{Classification}. As in section~\ref{sDeg44A1}, this example
appears in Koshelev's work \cite[\S1.1]{Koshelev} but with another point of view.

\paragraph{Configuration to blow-up.~---}
First we blow up three non-collinear conjugate points~$p_1,p_2,p_3$. This leads to a degree~$6$ del Pezzo surface with
three exceptional conjugate curves~$E_1,E_2,E_3$, the other exceptional curves being the strict transforms~$\ell_{12},\ell_{13},\ell_{23}$ of
the lines joining two of the three points. Then we blow up three other points~$p_4,p_5,p_6$ with~$p_{i+3} \succ p_i$, and
more precisely~$p_4$ is the intersection point of~$E_1$ and~$\widetilde{\ell}_{12}$, $p_5$ is the intersection point of~$E_2$ and~$\widetilde{\ell}_{23}$ and~$p_6$ is the intersection point of~$E_3$ and~$\widetilde{\ell}_{13}$. These points are also conjugate and the resulting
surface~$X$ is a weak degree three del Pezzo surface, with three new exceptional curves~$E_4,E_5,E_6$.
The anticanonical model~$X_s$ has three conjugate singular points of type~$\mathbf{A}_2$.
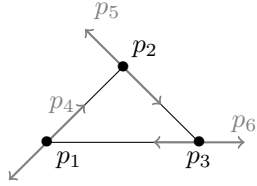
\begin{align*}
&\begin{tikzpicture}[baseline=0]
\coordinate (p1) at (-1,-0.5);
\coordinate (p3) at (1,-0.5);
\coordinate (p2) at (0,0.5);
\draw (p1) -- (p2) -- (p3) -- (p1);
\draw[thick,<->,gray] (-1.5,-1) -- (-0.5,0) node[left] {$p_4$};
\draw[thick,<->,gray] (-0.5,1) node[above right] {$p_5$} -- (0.5,0) ;
\draw[thick,<->,gray] (0.4,-0.5) -- (1.6,-0.5) node[above] {$p_6$};
\draw (p1) node {$\bullet$} node[below right] {$p_1$};
\draw (p2) node {$\bullet$} node[above right] {$p_2$};
\draw (p3) node {$\bullet$} node[below] {$p_3$};
\end{tikzpicture}
&
&\begin{array}{ll}p_2=p_1^\sigma&p_5=p_4^\sigma\\p_3=p_1^{\sigma^2}&p_6=p_4^{\sigma^2}.\end{array}
\end{align*}
The point~$p_4$ lies on the strict transform of the line~$(p_1p_2)$, which we denote by~$\ell_{124}$.
In the same way we introduce the lines~$\ell_{235}$ and~$\ell_{136}$.

\paragraph{Computation of the divisor class groups.~---} There are six irreducible effective roots, the strict transforms of~$E_1,E_2,E_3$
and the strict transforms of~$\ell_{124},\ell_{235},\ell_{136}$; their classes are:
\begin{align*}
&R_1=E_1-E_4,
&
&R_2=E_2-E_5,
&
&R_3=E_3-E_6,
\\
&R'_1=E_0-E_1-E_2-E_4,
&
&R'_2=E_0-E_2-E_3-E_5,
&
&R'_3=E_0-E_1-E_3-E_6.
\end{align*}
The absolute Galois group acts on this six root classes as~$(R_1R_2R_3)(R'_1R'_2R'_3)$ and also on the exceptional
curves as~$(E_1E_2E_3)(E_4E_5E_6)$ (the first three exceptional curves are the total transforms of the exceptional curves on the
degree~$6$ del Pezzo surface, they are no longer irreducible).

We have
\begin{align*}
&\cldivweil(\overline{X}) = \bigoplus_{i=0}^6 \Z E_i
&
&\overline{\RR} = \bigoplus_{i=1}^3 \Z R_i \oplus \Z R'_i
&
&\text{and}
&
&\overline{\RR}^\perp = \Z K_X.
\end{align*}
Let us put:
\begin{align*}
&\EE = E_1+E_2+E_3,
&
&\EE' = E_4+E_5+E_6,
&
&\Rcal = R_1+R_2+R_3 = \EE - \EE',
&
&\Rcal' = R'_1+R'_2+R'_3 = 3E_0 -2\EE - \EE'.
\end{align*}
One easily verify that
\begin{align*}
&\cldivweil(\overline{X})^\Gamma = \Z E_0 \oplus \Z \EE \oplus \Z \EE',
&
&\RR = \overline{\RR}^\Gamma = \Z {\mathcal R} \oplus \Z {\mathcal R}' = \Z (\EE - \EE') \oplus \Z (3E_0 -2\EE - \EE'),
&
&\text{and}
&
&\RR^\perp = \Z K_X.
\end{align*}
It turns out that the submodule~$\RR$ is not a direct summand in~$\cldivweil(X)$; indeed 
$$
\RR = \Z (\EE - \EE') \oplus \Z 3(E_0-\EE)
\subset
\Z (\EE - \EE') \oplus \Z (E_0-\EE) \oplus \Z \EE' = \cldivweil(X)
$$
(we have just replaced~$3E_0-2\EE-\EE'$ by~$(3E_0-2\EE-\EE') - (\EE-\EE')$ in the initial basis). Therefore the projection
onto the two last factors leads to an isomorphism:
$$
\begin{array}{ccrcl}
  \cldivweil(X_s) \simeq \cldivweil(X)/\RR & \longrightarrow & \Z/3\Z (E_0-\EE) &\oplus &\Z\EE'\\
  a_0 E_0 + a\EE + a'\EE'\bmod\RR & \longmapsto & (a_0 \bmod 3)\left(E_0-\EE\right) &+ &(a_0+a+a')\EE'
\end{array}
$$
Via this isomorphism the group~$\cldivcartier(X_s) = \RR^\perp = \Z K_X$ embeds via~$-K_X \mapsto \EE'$; this means that~$\cldivcartier(X_s)$ is isomorphic
to the free part of~$\cldivweil(X_s)$ and these two groups are free of rank one.

\paragraph{Types of decomposition into irreducible components in~$\left|-K_X\right|$.~---}
As in the previous case, the global sections of the divisor~$\left|-K_{X_s}\right|$ are irreducible but not necessarily absolutely
irreducible. As usual, we list the Galois orbits of lines or conics or cubics of degree less than~$3$ that pass through at least
one of the six points. The only possibilities are
\begin{align*}
&\ell_1\cup\ell_2\cup\ell_3,
&
&\ell_{124}\cup\ell_{235}\cup\ell_{136},
&
&q_{123},
&
&c_{123},
&
&c_{123456}.
\end{align*}
(it is important to keep in mind that a curve which passes through~$p_4$ necessarily passes through~$p_1$). There are only
two combinations that lead to a cubic which passes through the six points:
$$
{\renewcommand{\arraystretch}{1.5}
\begin{array}{|l|l|l|l||l|l||l|l|}
\hline
&\left|3\ell - \sum_{i=1}^6 p_i\right|  & \left|-K_X\right| & \left|-K_{X_s}\right| & \text{Max}   \\
&\text{on~$\P^2$}  & \text{on~$X$}  & \text{on~$X_s$}  & \text{nb. of pts}\\
\hline\hline
1
&\ell_{124} \cup \ell_{235} \cup \ell_{136}
&\widetilde{\ell}_{124}\cup \widetilde{\ell}_{235}\cup \widetilde{\ell}_{136}\cup \widetilde{E}_1\cup \widetilde{E}_2\cup \widetilde{E}_3\cup E_4 \cup E_5 \cup E_6
& \plongement_*(E_4) \cup \plongement_*(E_5) \cup \plongement_*(E_6) & 0 
\\
\hline
2
&c_{123456}
&\widetilde{c}_{123456}
&\plongement_*(\widetilde{c}_{123456})
&\Nqg{q}{1}
\\
\hline
\end{array}}
$$
The roots of~$X$ are~$\widetilde{\ell}_{124}$, $\widetilde{E}_2$ (mapped to a singular point~$s\in X_s$),
$\widetilde{\ell}_{235}$, $\widetilde{E}_3$ (mapped to a singular point~$s^{\sigma}\in X_s$), and~$\widetilde{\ell}_{136}$, $\widetilde{E}_1$
(mapped to a singular point~$s^{\sigma^2}\in X_s$). The curves~$E_i$, $i=4,5,6$, are not defined over~$\F_q$ and do not contain any
rational point. In conclusion~$N_q\left(-K_{X_s}\right) \leq N_q(1)$.

Since the points~$p_1,\ldots,p_6$ are not rational the blowing ups do not add any rational point, and since the singular points
are not rational the contractions do not add any rational point also. Thus~$\#X_s(\F_q) = q^2+q+1$, this number is always
strictly greater than~$N_q(1)$ and we deduce the parameters given below.

\begin{prop}
The weak del Pezzo surface of degree~$3$ associated to the configuration specified at the beginning of this section has
parameters~$[q^2+q+1,4,\geq q^2+q+1-\Nqg{q}{1}]$.
\end{prop}

Koshelev \cite[\S1.1]{Koshelev} proves that the minimum distance can be improved by~$1$ for some~$q$ since he shows that
cubics of the considered linear system must have a $3$-torsion point.

\paragraph{Computation of the global sections from~$\P^2$.~---} Let~$L_{12},L_{23},L_{13}$ be the three conjugate linear forms that
respectively define the lines~$\ell_{124},\ell_{235},\ell_{136}$ in~$\P^2$. The family~$L_{12},L_{23},L_{13}$ is a $\overline{\F}_q$-basis
of~$H^0(\P^2,\ell)$, while the family of degree~$3$ monomials in~$L_{12},L_{23},L_{13}$ is a $\overline{\F}_q$-basis
of~$H^0(\P^2,3\ell)$. A cubic in this space can be written:
$$
a_{1}L_{12}^3+a_{2}L_{23}^3+a_{3}L_{13}^3
+b_{1}L_{12}L_{23}^2+c_{1}L_{13}L_{23}^2
+b_{2}L_{12}L_{13}^2+c_{2}L_{23}L_{13}^2
+b_{3}L_{13}L_{12}^2+c_{3}L_{23}L_{12}^2
+d L_{12}L_{23}L_{13}
$$
Such a cubic pass through~$p_1$ if and only if~$a_2=0$ (since~$p_1$ is a common zero of~$L_{12}$ and~$L_{13}$). In the same way it
passes through~$p_2$ and~$p_3$ if and only if~$a_3=0$ and~$a_1= 0$. Now passing through~$p_4$ means that if this cubic is not singular
at~$p_1$ then its tangent line at this point must be~$\ell_{12}$. After deshomogenizing by putting~$L_{23} = 1$ (this is possible
since~$L_{23}$ does not vanish at~$p_1$) this means that the linear component~$b_1L_{12}+c_1L_{13}$ should be proportional to~$L_{12}$;
necessarily~$c_1=0$. In the same way passing through~$p_5$ (resp.~$p_6$) means that~$b_2=0$ (resp.~$c_3=0$). Finally, one has
$$
H^0\left(\P^2,3\ell-\textstyle\sum_{i=1}^6 p_i\right)
=
\left\langle L_{12}L_{23}^2,L_{23}L_{13}^2,L_{13}L_{12}^2,L_{12}L_{23}L_{13}\right\rangle_{\overline{\F}_q}
$$
In order to deduce a $\F_q$-base, we consider~$\theta$ any primitive element of~$\F_{q^3}$ over~$\F_q$. The linear independence of homomorphisms permits to prove that
the matrix~$(\sigma^i(\theta^j))_{1\leq i,j\leq 3}$ is invertible. Let us put:
\begin{align*}
&C_1 = L_{12}L_{23}^2+L_{23}L_{13}^2+L_{13}L_{12}^2
\\
&C_\theta = \theta L_{12}L_{23}^2+\sigma(\theta)L_{23}L_{13}^2+\sigma^2(\theta)L_{13}L_{12}^2
\\
&C_{\theta^2} = \theta^2 L_{12}L_{23}^2+\sigma(\theta^2)L_{23}L_{13}^2+\sigma^2(\theta^2)L_{13}L_{12}^2
\end{align*}
then~$C,C_\theta,C_{\theta^2}$ are defined over~$\F_q$, as the product~$L_{12}L_{23}L_{13}$ and one has:
$$
H^0\left(\P^2,3\ell-\textstyle\sum_{i=1}^6 p_i\right)
=
\left\langle C_1,C_\theta,C_{\theta^2},L_{12}L_{23}L_{13}\right\rangle_{\F_q}
$$
The birational morphism
$$
\begin{array}{ccc}
\P^2 & \dashrightarrow & \P^4\\
(X:Y:Z) & \longmapsto & \left(C_1:C_\theta:C_{\theta^2}:L_{12}L_{23}L_{13}\right)
\end{array}
$$
has~$X_s$ as image in~$\P^4$. Thus, if~$r_1,\ldots,r_{q^2+q+1}$ denote the rational points of~$\P^2$, one of the generating matrix of this code
is nothing else than:
\begin{align*}
\begin{pmatrix}
C_1(r_1) & \cdots & C_1(r_{q^2+q+1})
\\
C_\theta(r_1) & \cdots & C_\theta(r_{q^2+q+1})
\\
C_{\theta^2}(r_1) & \cdots & C_{\theta^2}(r_{q^2+q+1})
\\
L_{12}L_{23}L_{13}(r_1) & \cdots & L_{12}L_{23}L_{13}(r_{q^2+q+1})
\end{pmatrix}
\end{align*}


\bibliographystyle{amsalpha}

\begin{thebibliography}{BCH{\etalchar{+}}20}

\bibitem[AW92]{Adkins_Weintraub}
William~A. Adkins and Steven~H. Weintraub, \emph{Algebra an approach via module
  theory}, Graduate Texts in Mathematics, vol. 136, Springer, 1992.

\bibitem[BCH{\etalchar{+}}20]{AntiCanonical}
R\'{e}gis Blache, Alain Couvreur, Emmanuel Hallouin, David Madore, Jade Nardi,
  Matthieu Rambaud, and Hugues Randriam, \emph{Anticanonical codes from del
  {P}ezzo surfaces with {P}icard rank one}, Trans. Amer. Math. Soc.
  \textbf{373} (2020), no.~8, 5371--5393.

\bibitem[BH22]{Classification}
R\'{e}gis Blache and Emmanuel Hallouin, \emph{Classification of singular del
  pezzo surfaces over finite fields}, submitted, 2022.

\bibitem[Bri13]{Bright}
Martin Bright, \emph{Brauer groups of singular del {P}ezzo surfaces}, Michigan
  Math. J. \textbf{62} (2013), no.~3, 657--664.

\bibitem[CA00]{Casas}
Eduardo Casas-Alvero, \emph{Singularities of plane curves}, London Mathematical
  Society Lecture Note Series, vol. 276, Cambridge University Press, Cambridge,
  2000.

\bibitem[CD13]{CouvreurDuursma}
Alain Couvreur and Iwan Duursma, \emph{Evaluation codes from smooth quadric
  surfaces and twisted {S}egre varieties}, Des. Codes Cryptogr. \textbf{66}
  (2013), no.~1-3, 291--303.

\bibitem[Cou11]{Alain}
Alain Couvreur, \emph{Construction of rational surfaces yielding good codes},
  Finite Fields Appl. \textbf{17} (2011), no.~5, 424--441.

\bibitem[CT88]{CorayTsfasman}
D.~F. Coray and M.~A. Tsfasman, \emph{Arithmetic on singular {D}el {P}ezzo
  surfaces}, Proc. London Math. Soc. (3) \textbf{57} (1988), no.~1, 25--87.

\bibitem[Dem80]{Demazure}
Michel Demazure, \emph{Surfaces de del pezzo {\rm in} {S}\'{e}minaire sur les
  {S}ingularit\'{e}s des {S}urfaces}, Lecture Notes in Mathematics, vol. 777,
  Springer, Berlin, 1980.

\bibitem[Dol12]{Dolgachev}
Igor~V. Dolgachev, \emph{Classical algebraic geometry}, Cambridge University
  Press, Cambridge, 2012, A modern view.

\bibitem[Edo08]{Edoukou}
Fr\'{e}d\'{e}ric A.~B. Edoukou, \emph{Codes defined by forms of degree 2 on
  quadric surfaces}, IEEE Trans. Inform. Theory \textbf{54} (2008), no.~2,
  860--864.

\bibitem[Ful98]{FultonIT}
William Fulton, \emph{Intersection theory}, second ed., Ergebnisse der
  Mathematik und ihrer Grenzgebiete. 3. Folge. A Series of Modern Surveys in
  Mathematics [Results in Mathematics and Related Areas. 3rd Series. A Series
  of Modern Surveys in Mathematics], vol.~2, Springer-Verlag, Berlin, 1998.

\bibitem[Har77]{Hartshorne}
Robin Hartshorne, \emph{Algebraic geometry}, Graduate Texts in Mathematics,
  vol.~52, Springer, 1977.

\bibitem[Kos20]{Koshelev}
Dmitrii Koshelev, \emph{Non-split toric {BCH} codes on singular del {P}ezzo
  surfaces}, IEEE Trans. Inform. Theory \textbf{66} (2020), no.~12, 7341--7347.

\bibitem[Liu02]{Liu}
Qing Liu, \emph{Algebraic geometry and arithmetic curves}, Oxford Graduate
  Texts in Mathematics, vol.~6, Oxford, 2002.

\bibitem[LS18]{LittleSchenck}
John Little and Hal Schenck, \emph{Codes from surfaces with small {P}icard
  number}, SIAM J. Appl. Algebra Geom. \textbf{2} (2018), no.~2, 242--258.

\bibitem[Man74]{Manin}
Yu.I. Manin, \emph{Cubic forms. algebra, geometry, arithmetic}, North Holland
  Publishing Compagny, 1974.

\bibitem[Ser20]{SerreRationalPoints}
Jean-Pierre Serre, \emph{Rational points on curves over finite fields},
  Documents Math\'{e}matiques (Paris) [Mathematical Documents (Paris)],
  vol.~18, Soci\'{e}t\'{e} Math\'{e}matique de France, Paris, [2020] \copyright
  2020, With contributions by Everett Howe, Joseph Oesterl\'{e} and Christophe
  Ritzenthaler, Edited by Alp Bassa, Elisa Lorenzo Garc\'{\i}a, Christophe
  Ritzenthaler and Ren\'{e} Schoof.

\bibitem[{Sta}18]{stacks}
The {Stacks Project Authors}, \emph{\textit{Stacks Project}},
  \url{https://stacks.math.columbia.edu}, 2018.

\bibitem[Zar07]{zarzar}
Marcos Zarzar, \emph{Error-correcting codes on low rank surfaces}, Finite
  Fields Appl. \textbf{13} (2007), no.~4, 727--737.

\end{thebibliography}

\newcommand{\etalchar}[1]{$^{#1}$}
\providecommand{\bysame}{\leavevmode\hbox to3em{\hrulefill}\thinspace}
\providecommand{\MR}{\relax\ifhmode\unskip\space\fi MR }
\providecommand{\MRhref}[2]{%
  \href{http://www.ams.org/mathscinet-getitem?mr=#1}{#2}
}
\providecommand{\href}[2]{#2}

\end{document}